\documentclass[reqno]{amsart}
\usepackage{enumitem} 
\usepackage[margin=1.2in,includeheadfoot]{geometry}

\usepackage{amssymb}
\usepackage{amsaddr}
\usepackage{latexsym}
\usepackage{amsmath}
\usepackage{mathrsfs}

\usepackage{upgreek}
\usepackage{mathabx}
\usepackage{esint}
\usepackage{xcolor}
\usepackage{bbm}
\definecolor{citeblue}{RGB}{0,0,150}
\definecolor{citegreen}{RGB}{0,110,0}
\usepackage[colorlinks = true,
            linkcolor = citeblue,
            urlcolor  =  citeblue,
            citecolor =  citegreen,
        ]{hyperref}   
\usepackage{empheq}
\usepackage{cleveref}
\usepackage{mathtools}
\usepackage{amsthm}
\usepackage[title]{appendix}
\usepackage{amsfonts}
\usepackage{relsize}
\usepackage{esvect}

\makeatletter
\def\l@subsection{\@tocline{2}{0pt}{2pc}{0pc}{}}
\makeatletter
\usepackage{cases}
\makeatletter

\newcommand{\Rmn}[1]{\expandafter\@slowromancap\romannumeral #1@}
\makeatother

\usepackage{graphicx}
\usepackage[left,modulo]{lineno}
\newtheorem{thm}{Theorem}[section]

  \newenvironment{hproof5}{%
  \proof}{\endproof}
 \newtheorem{Th}[thm]{Theorem}

 \newtheorem{Prop}[thm]{Proposition}
 \newtheorem{Co}[thm]{Corollary}
\newtheorem{Lm}[thm]{Lemma}

\newtheorem{Dfi}[thm]{Definition}

\numberwithin{equation}{section}
\newcommand{\be}{\begin{equation}}
\newcommand{\ee}{\end{equation}}
\newcommand{\bg}{\begin{gather}}
\newcommand{\eg}{\end{gather}}
\newcommand{\ba}{\begin{align}}
\newcommand{\ea}{\end{align}}
\newcommand{\bad}{\begin{aligned}}
\newcommand{\ead}{\end{aligned}}
\newcommand{\R}{\mathbb{R}}
\newcommand{\N}{\mathbb{N}}

\newcommand{\Z}{\mathbb{Z}}
\newcommand{\mca}[1]{\mathcal{#1}}

\newcommand{\nf}{\infty}

\def\avint{\mathop{\mathchoice{\,\rlap{-}\!\!\int}
                              {\rlap{\raise.15em{\scriptstyle -}}\kern-.2em\int}
                              {\rlap{\raise.09em{\scriptscriptstyle -}}\!\int}
                              {\rlap{-}\!\int}}\nolimits}

\def\avint{\mathop{\,\rlap{-}\!\!\int}\nolimits}

\def\ovwe{\text{\larger[1.5]$\we$}}

\def\XXint#1#2#3{{\setbox0=\hbox{$#1{#2#3}{\int}$ }
\vcenter{\hbox{$#2#3$ }}\kern-.6\wd0}}

\def\wih{\widehat}
\usepackage{dsfont}
\def\mkt{\mkern2mu}
\def\mko{\mkern1mu}

\def\ga{\gamma}
\def\La{\Lambda}

\def\coeq{\coloneq}
\def\lf{\left}

\def\ot{\otimes}
\def\de{\delta}
\def\rg{\right}

\def\al{\alpha}

\def\la{\lambda}

\def\wti{\widetilde}

\newcommand{\we}{\wedge}

\def\dil{\textup{dil}}
\def\rot{\textup{rot}}
\def\std{\textup{std}}
\def\sym{\textup{sym}}
\def\bwe{\textstyle\bigwedge}
\def\loc{\textup{loc}}

\def\vae{\varepsilon}

\def\vp{\varphi}

\def\II{\mathrm{I\!I}}
\def\bII{\vec{\II}}
\def\om{\omega}
\def\p{\partial}
\def\bn{\vec{n}}

\def\bH{\vec{H}}
\def\bu{\vec{u}}
\def\bv{\vec{v}}
\def\bvt{\vec \vartheta}
\def\vt{\vartheta}
\def\bg{\vec{g}}
\def\bC{\vec{C}}

\def \g{\nabla}

\def\Ga{\Gamma}
\def\bD{\vec{D}}

\def\lan{\langle}
\def\ran{\rangle}
\def\bL{\vec{L}}
\def\bR{\vec{R}}

\def\bX{\vec{X}}
\def\bY{\vec{Y}}
\def\bw{\vec{w}}

\def\vet{\vec \eta}

\def\bc{\vec{c}}
\def\vp{\varphi}
\def\bP{\vec{\phi}}
\def\bP{\vec{\Phi}}
\def\dvol{d\textup{vol}}

\def\si{\sigma}
\def\Si{\Sigma}

\newcommand{\res}{\mathbin{\hbox{\vrule height 5pt width .4pt depth 0pt\vrule height .6pt width 4pt depth 0pt}}} 
\newcommand{\metricsub}[3]{_{
    \mkern-#3mu
    \smash{\lower #1\hbox{$\scriptscriptstyle #2$}}
}}
\newcommand{\resg}{\res_g}
\newcommand{\sbul}{\mathbin{
  \mathchoice{\scalebox{0.7}{$\bullet$}}{\scalebox{0.7}{$\bullet$}}%
             {\scalebox{0.7}{$\bullet$}}{\scalebox{0.7}{$\bullet$}}}
}
\newcommand{\bulg}{\sbul_g}
\newcommand{\dwe}{\mathbin{\dot{\wedge}}}
\newcommand{\dres}{\mathbin{\dot{\res}_g}}
\newcommand{\wres}{\mathbin{\ovs{\ovwe}{\res}_{\mkern -4mu g}}}

\newcommand{\ov}[1]{\overline{#1}}
\newcommand{\msc}[1]{\mathscr{#1}}
\newcommand{\ti}[1]{\tilde{#1}}
\newcommand{\ovs}[2]{\overset{#1}{#2}}

\newcommand{\lap}{\Delta}

\newcommand{\Imm}{\mathrm{Imm}}

\begin{document}
\title[Regularity for scale-invariant energies]{Regularity of critical immersions for scale-invariant curvature energies in even dimensions}
\author{\small Tian Lan}
\address{\smaller Department of Mathematics, ETH Zürich, 101 Rämistrasse, 8092 Zürich, Switzerland}
\date{\today}
\begin{abstract}
For immersions of a closed $2h$-dimensional manifold $\Sigma$ into $\R^m$, we consider curvature energies that are invariant under ambient isometries and dilations, with the principal term $\int_\Sigma |\g^{(h-1)} \bII|_g^2\,\dvol_g$ and lower-order submanifold invariants.
Following the four-dimensional approach developed in joint work with Bernard, Martino, and Rivi\`ere, we prove that every weak critical immersion in the natural Sobolev class $W^{h+1,2}$, whose induced metric and its inverse have $L^\nf$ coefficients, is real-analytic in harmonic coordinates.
The proof combines geometric conservation laws, structural identities, and elliptic estimates with critical Sobolev coefficients to obtain Morrey decay and bootstrap to full regularity.
\end{abstract}
\maketitle
 \vskip0.5cm
\tableofcontents

\section{Introduction}\label{sec:intro}
\subsection{Analytical and variational setting}
\
\vskip5pt
The analysis of critical points of scale-invariant Lagrangians often involves critical Sobolev spaces $W^{s,p}(B^n)$ with $sp=n$ and $1<p<\infty$.
These spaces embed into every finite $L^q$ but not into $L^\infty$, and direct estimates of nonlinear terms may fail to initiate an elliptic bootstrap.
However, certain geometric structures lead to useful reformulations of the Euler--Lagrange equations. Examples include the use of moving frames for harmonic maps from surfaces to Riemannian manifolds~\cite{H2002}, \textit{conservation laws} for two-dimensional conformally invariant nonlinear PDEs~\cite{Riv07,Riv08}, and Coulomb gauges in four-dimensional Yang--Mills theory~\cite{Uhlenbeck82}; see also the survey~\cite{Ri20}.
Wente's inequality~\cite{Wente69}, Hardy-space estimates for div--curl products~\cite{CLMS93}, and Lorentz-space estimates provide key analytic tools for exploiting these structures through \textit{integrability by compensation}.

The conservation-law approach developed by Rivi\`ere~\cite{Riv07} for second-order systems in dimension two was extended by Lamm--Rivi\`ere~\cite{LammRiv08} to fourth-order systems in dimension four, including the biharmonic map equations.
De Longueville--Gastel~\cite{deLonGas21} extended this approach to systems of polyharmonic map type of order $2m$ in dimension $2m$ for every integer $m\ge 3$.
In these works, antisymmetric structure in the least regular coefficient enables the construction of conservation laws and the proof of continuity of weak solutions.

For the \textit{Willmore energy}, Rivi\`ere~\cite{Riv08} used conservation laws to prove regularity of weak critical immersions.
Bernard~\cite{Ber} subsequently derived these conservation laws through Noether's theorem using invariance under ambient translations, rotations, and dilations. 
More recently, the regularity result~\cite[Thm.~1.3]{Riv08} was extended to dimension four by Bernard, the author, Martino, and Rivi\`ere~\cite{BerLanMarRiv26} using conservation laws together with additional structural equations established in~\cite{Bernard25}. See also~\cite{KuwLamLi15,CafStiViv24} for some regularity results on curvature functionals in other variational settings.

In this work, we extend the four-dimensional approach in~\cite{BerLanMarRiv26} to a general class of scale-invariant curvature energies in higher even dimensions. Unlike the systems with a fixed polyharmonic principal part in~\cite{deLonGas21}, the Euler--Lagrange equation for our problem is degenerate due to reparametrization invariance, and its principal coefficients depend on the immersion itself.
Since curvature bounds are invariant under reparametrizations, suitable choices of coordinates are needed to translate these bounds into regularity of the immersion.
For surfaces, \textit{conformal coordinates} provide a natural choice. 
For suitably normalized conformal immersions, Müller--Šverák~\cite{MullerSverak95} obtained local $W^{2,2}$ estimates and continuity of the conformal metric under a sufficiently small local $L^2$ bound on the second fundamental form; see also~\cite{Toro94,H2002} for related results.
In even dimension $n=2h\ge4$, a natural scale-invariant analogue of this curvature control is
\[
\mathcal E_n(\bP):=\sum_{j=1}^{h}\int_\Sigma\big|\nabla^{(j-1)}\bII\big|_g^{\frac n{j}}\,d\mathrm{vol}_g<\infty.
\]
Here and throughout the paper, we write $g=g_{\bP}$ and omit the subscript $\bP$ when there is no ambiguity; see~\eqref{defgdvol},~\eqref{not-met}, and~\eqref{eq:defitecov} for relevant notation.

Since conformal coordinates are generally unavailable in higher dimensions, Martino--Rivi\`ere~\cite{MarRiv2,MarRiv26} used \textit{harmonic coordinates} as a higher-dimensional substitute. 
As in the two-dimensional case, they proved that bounds on $\mca E_n(\bP)$ yield local $W^{h+1,2}$ estimates for the immersion and continuity of the induced metric in suitably chosen harmonic coordinates, provided that $\mca E_n(\bP)$ is sufficiently small on the patch. Thus, we adopt their notion of \textit{weak immersions} as the natural variational framework for scale-invariant curvature energies; see also~\cite{KuwertLi12,Riv14,MonRiv14,Ri16} for related results in dimension two.

\begin{Dfi}\label{defweakimm}
    Let $k\in\N\cup\{0\}$, $n\in \N^+$, $p\in [1,\nf]$, and let $(\Si,\bar g)$ be an $n$-dimensional closed oriented Riemannian manifold. We define
    \begin{align*}
        \mathcal{I}_{k,p}(\Si,\R^m)\coloneq \left\{
        \bP\in W^{k+2,p}(\Si,\R^m) \colon
        \displaystyle \exists\, c_{\bP}\ge 1,\quad c^{-1}_{\bP} \bar g \leq g_{\bP} \leq c_{\bP}\mko \bar g
        \right\}.
    \end{align*}
    Here $g_{\bP}=\bP^*g_{\std}$ denotes the metric induced by $\bP$. We define $\mathcal{I}_{k,p}(B^n,\R^m)$ analogously, replacing $\bar g$ by the Euclidean metric on $B^n$.
\end{Dfi}

We focus on the class $\mathcal I_{h-1,2}(\Si,\R^m)$ with $h\in \N^+$ and $n=2h$. An immersion $\bP\in \mathcal I_{h-1,2}(\Si,\R^m)$ need not be $C^1$ and may fail to define a graph in the original coordinates. 
For $n\ge4$, Martino--Rivi\`ere proved in \cite{MarRiv2} that such a weak immersion induces an atlas of $g_{\bP}$-harmonic coordinates in which the coefficients of $g_{\bP}$ belong to $W^{2,(h,1)}\hookrightarrow C^0$ and the transition maps are $C^1$.

Inspired by the generalized Willmore energies (see Section~\ref{sec:geobac}), we consider geometric energies of the following type, where $F$ is a real polynomial in the partial derivatives of $\bP$ and $(\det g_{ij})^{-1}$ in local coordinates:
\begin{align}\label{eq:firdefE}
    \bP\in \mathcal{I}_{h-1,2}(\Si,\R^m) \mapsto \int_{\Si} \bigg(\big| \g^{(h+1)} \bP\big|_g^2+F\Big( (D^\al \bP)_{|\al|\le h +1},(\det g_{ij})^{-1}\Big)\bigg)\, \dvol_g.
\end{align}
We are interested in the case where the $n$-form in~\eqref{eq:firdefE} is invariant under dilations and isometries of $\R^m$, and invariant under reparametrizations. Weyl's invariant theory~\cite[Thm.~2.9.A]{weyl97}, together with the arguments in~\cite{Atiyah73,grahamkuo26} and Section~\ref{sec:cov}, implies that $F$ is a linear combination of \textit{geometric complete contractions} with respect to $g$ of the form
\begin{equation}\label{eq:comcon}
    \textup{contr}\Big(\big(\g^{(k_1)}\mko\bII\cdot \g^{(k_2)}\mko\bII\big)\ot  \cdots \ot \big(\g^{(k_{2s-1})}\mko\bII\cdot \g^{(k_{2s})}\mko\bII\big) \Big).
\end{equation}
Here $\g^{(k_{2l-1})}\mko\bII\cdot \g^{(k_{2l})}\mko\bII$ is a $(0,k_{2l-1}+k_{2l}+4)$-tensor field for $1\le l\le s$.\footnote{Replacing $\g$ in~\eqref{eq:comcon} by the normal covariant derivative $\g^\perp$ yields the same class of geometric invariants; see Section~\ref{sec:cov}.} We assume $s\ge2$, so that these terms are lower-order compared with the principal term
$\big|\g^{(h+1)}\bP\big|_g^2=\big|\g^{(h-1)}\mko\bII\big|_g^2$. Using integration by parts, we can also assume $0\le k_i\le h-1$ for all $1\le i\le 2s$. Moreover, scale invariance gives $\sum_{i=1}^{2s} (k_{i}+1)=n$. We summarize these conditions:
\begin{align}\label{eq:condsks}
    \begin{dcases}
    s\ge 2,\\
    \forall 1\le i\le 2s,\qquad 0\le k_{i}\le h-1,\\
   \sum_{i=1}^{2s} (k_{i}+1)=n.
    \end{dcases}
\end{align}
We choose a finite index set $\mca L$ and a family $\{P_\la:\la\in \mca L\}$ such that
\begin{equation}\label{defmcaL}
 \{P_\la: \la\in \mca L\} =\big\{\text{all complete contractions of the form~\eqref{eq:comcon} satisfying~\eqref{eq:condsks}}\big\}.
\end{equation}

Here we do not consider the additional contractions involving the tangent or ambient volume forms,
which may change sign under reflections. For a complete contraction of the form~\eqref{eq:comcon}, the total number of tangent indices is $\sum_{i=1}^{2s}k_i+4s$. Hence $\sum_{i=1}^{2s}k_i$ must be even, and the scale-invariance in~\eqref{eq:condsks} then implies that $n$ must be even.

We note that when $n=2$, the set $\mca L$ is empty. When $n=4$, it consists of the eight quartic contractions listed in \cite[Eq.~(I.6)]{BerLanMarRiv26}. 
Given $\vec c=(c_\la)_{\la\in \mca L}$ with $c_\la\in\R$, we define
\begin{align}\label{eq:defE}
   \forall \,\bP\in \mathcal{I}_{h-1,2}(\Si,\R^m),\qquad E_{\vec c}(\bP)=\int_{\Si} \bigg(\big| \g^{(h-1)}\mko\bII\big|_g^2+\sum_{\la\in \mca L} c_\la P_\la(\bP)\bigg) \, \dvol_g.
\end{align}
The class~\eqref{eq:defE} contains all the generalized Willmore energies recalled in Section~\ref{sec:geobac}, after normalization and a suitable choice of $\vec c$.

\subsection{The main result}
\
\vskip5pt
A weak immersion $\bP\in \mathcal I_{h-1,2}(\Si,\R^m)$ is said to be a \textit{weak critical point} of $E_{\vec c}$ if for any $\vec w\in C_c^\infty(\Si,\R^m)$, it holds that
\[
   \frac d{dt}\mkt  E_{\vec c}\big( \bP+t\mkt \vec w \big)\Big|_{t=0}=0.
\]
The main result of the present article is as follows.
\begin{Th}
\label{th-main}
Let $n=2h$ be a positive even integer and $m>n$. Fix a closed oriented smooth manifold $\Si$ of dimension $n$, and a vector $\vec{c}=(c_\la)_{\la\in \mca L}$ with $c_\la\in \R$. Then every weak critical point $\bP\in \mathcal I_{h-1,2}(\Si,\R^m)$ of $E_{\vec c}$ is real-analytic in any $g_{\bP}$-harmonic coordinates.
\end{Th}

When $n=2$, the Gauss equation reduces $E_{\vec c}$, modulo a topological term, to the Willmore energy, and Theorem~\ref{th-main} follows from the regularity theorem in \cite{Riv08}. 
The four-dimensional case was recently proved in \cite{BerLanMarRiv26}. Thus in this work, we focus on the case $n\ge 6$.
\subsection{Geometric background}\label{sec:geobac}
\
\vskip5pt
Invariant theory plays an important role in describing the energies~\eqref{eq:defE} and their pointwise first variations in terms of tensor contractions. 
For a Riemannian manifold $(\Si^n,g)$, the classical Weyl’s invariant theory shows that local scalar invariants depending polynomially on the jets of the metric coefficients and $(\det g)^{-1}$ can be expressed as linear combinations of complete contractions of the Riemann curvature tensor and its covariant derivatives; see for instance \cite{weyl97,Atiyah73}.

In conformal geometry, the Fefferman--Graham ambient metric provides a systematic construction of local conformal invariants~\cite{Fefferman85,FeffermanGraham12}; see also~\cite{Bailey94} for the associated invariant theory.
A different question is to characterize scalar Riemannian invariants of weight $-n$ whose integrals over closed $n$-dimensional manifolds are conformally invariant.
For even $n$, Alexakis's resolution of the Deser--Schwimmer conjecture shows that any such invariant can be decomposed into a local conformal invariant, divergence of a natural vector field, and a multiple of the Chern--Gauss--Bonnet integrand~\cite{DeserSchwimmer93,Alexakis12}.

For an immersed Riemannian submanifold, the induced metric provides intrinsic invariants, while the second fundamental form and the normal connection provide extrinsic invariants.
An early study of local submanifold invariants was given by Gilkey \cite{Gilkey75}. 
More recently, Graham–Kuo~\cite{grahamkuo26} characterized natural submanifold tensors as linear combinations of contractions of covariant derivatives of the ambient curvature and the second fundamental form. 

For closed surfaces of codimension one or two, Mondino--Nguyen \cite{mondino2018} classified global conformal invariants whose integrands are linear combinations of complete contractions of the tangential and normal curvature tensors and the second fundamental form without covariant derivatives. 
In a recent work, Case--Khaitan--Lin--Tyrrell--Yuan~\cite{CaseEtAl26} developed constructions of local and global conformal invariants of submanifolds using extrinsic ambient space and renormalization of curvature integrals.

The Willmore energy is a fundamental example of a global conformal invariant for closed surfaces; see~\cite{LaMaRi26} for a survey of recent developments.
In higher even dimensions, a particularly important family of conformally invariant curvature energies arises from the renormalized area construction introduced by Graham--Witten \cite{graham1999} in the context of the AdS/CFT correspondence.
Later, Graham--Reichert~\cite{grahamreichert2020} derived these higher-dimensional Willmore energies and studied their variational properties.
They identified the first variations of these energies as the obstruction to smoothness of the associated asymptotically minimal extensions and obtained an explicit formula in dimension four. 
For Euclidean ambient space, the energy of an immersion $\bP\in\Imm(\Sigma^4,\R^m)$ is given by
\begin{equation}\label{eq:GR}
    \mathcal{E}_{GR}(\bP)\coloneq \int_{\Sigma^4} \Big(\big|\g^{\perp}\vec H \big|^2_{g} - \big| \bH\cdot \bII \big|^2_{g} + 7\big| \bH \big|^4 \Big)\, d\textup{vol}_{g}.
\end{equation}
Here $\bH$ is the mean curvature vector of $\bP$ and $\g^{\perp}$ denotes the induced connection on the normal bundle. 
In the hypersurface case, this energy was already derived by a different method by Guven~\cite{Guven05} in 2005, although a factor of $-2$ had been dropped in the original calculations, as pointed out in~\cite{grahamreichert2020}. 
The formula~\eqref{eq:GR} was also obtained independently by Zhang \cite{zhang}.
For further studies of this functional, other notions of \textit{generalized Willmore energy}, and connections with GJMS operators and $Q$-curvatures, see~\cite{gover2017,graham2017,gover2020,blitz2023,AG2024,Martino2026,CaseGrahamKuo25}.

For even dimensions $n=2h\ge 4$ and Euclidean ambient spaces, all the generalized Willmore energies discussed above share a principal quadratic term that can be written after normalization in either of the following forms:
\begin{equation}\label{altergbPgH}
    \int_{\Si^n}|\g^{(h+1)}\bP|_g^2\, \dvol_g \qquad\text{ or } \qquad n^2\int_{\Si^n}|\g^{(h-1)}\bH|_g^2\,\dvol_g.
\end{equation}
Indeed, as in~\cite[Eq.~(26)]{Guven05} and~\eqref{intindwbP}, repeated integration by parts and control of the commutators of covariant derivatives imply that these two integrals differ by an integral of a linear combination of lower-order terms in the form~\eqref{eq:comcon}--\eqref{eq:condsks}; see also~\eqref{difcovcan}.\footnote{In dimension two, the two integrals instead differ by a Gauss--Bonnet term, while the class of lower-order terms in~\eqref{defmcaL} is empty.} 
The class~\eqref{eq:defE} retains this principal quadratic term but does not require conformal invariance; we only impose invariance under isometries and dilations of $\R^m$.

Since the ambient space is Euclidean, the Gauss equation (see for instance~\cite{dC92}) gives
\begin{equation}\label{Gaseq}
    \text{Riem}_{ijk\ell}=\bII_{ik}\cdot \bII_{j\ell}-\bII_{i\ell}\cdot \bII_{jk}.
\end{equation}
For $n=2h\ge4$, the Gauss equation and integration by parts imply that the integral of any Riemannian invariant of weight $-n$, when evaluated on closed submanifolds of $\R^m$, can be written as the lower-order part of the integral in~\eqref{eq:defE}, 
for a coefficient vector $\vec c$ independent of $\bP$ and the topology of $\Si$. In particular, this applies to the intrinsic global conformal invariants considered in~\cite{Alexakis12}.

\subsection{Strategy of the proof and new challenges}
\
\vskip5pt
Let $n=2h\ge6$, and we consider a weak critical immersion $\bP\in\mathcal I_{h-1,2}(B^n,\R^m)$ in harmonic coordinates. The first difficulty compared with~\cite{BerLanMarRiv26} is to derive conservation laws with precise control of the principal and lower-order terms. Methods for deriving Euler--Lagrange equations for curvature functionals were developed in~\cite{AmbMan98,Mantegazza02}. 
Conservation laws associated with ambient symmetries were also derived in~\cite{ArrCapGuv00,CapGuv02,Guven05} from pointwise first-variation formulas involving partial derivatives of the scalar Lagrangian with respect to its geometric arguments.

In our weak-immersion framework, we consider a variation $(\bP_t)_{|t|<1}$ in $\mathcal{I}_{h-1,2}(B^n, \R^m)$ that is $C^1$ in $t$, with $\bP_0=\bP$.
In particular, the variation field $\bw= \frac{d}{dt}\bP_t\big|_{t=0}$ lies in $W^{1,\nf}\cap W^{h+1,2}(B^n,\R^m)$.
Since the Gauss map $\bn$ lies only in $L^\infty\cap W^{h,2}$, the normal projection $\pi_N$ need not preserve the variation space. Thus, we compute the first variation for arbitrary ambient variation fields, without the tangential--normal decomposition.
In~\cite{BerLanMarRiv26}, the variation formula was obtained by treating each lower-order term and the principal term separately.
In this work, we use invariant theory to express the pointwise first variation in terms of geometric contractions (see~\eqref{ptvargwgp} and~\eqref{ptvardivgw}). 
This allows us to identify the principal terms and control the lower-order terms to derive the conservation laws and establish the associated estimates.\footnote{For dilations and rotations, we take $\bP_t=(1+t)\bP$ and $\bP_t=Q_t\circ \bP$ respectively for $t$ close to $0$, where $(Q_t)_{t\in \R}$ is a smooth curve in $\text{SO}(m)$ with $Q_0=\mathrm{Id}$.}
In particular, Lemma~\ref{lm-critic-W^{1,1}} gives
\begin{align}
\begin{dcases}\label{inteullag}
    d*_g\vec V=0,\\
    \displaystyle \vec V=d\big(\lap_g^h\bP\big)+\sum_{\ga\in I}a_\ga\vec Q_\ga(\bP),\\
    \displaystyle \vec V\in L^1+W^{-h,2}\big(B^n,\R^m\ot\R^n\big).
\end{dcases}
\end{align}
Here $I$ is a finite index set, the real coefficients $a_\ga$ depend only on $\bc$, and the terms $\vec Q_\ga(\bP)$ are partial contractions of lower differential order in $\bP$ than $d\big(\lap_g^h\bP\big)$.
The precise contraction structure of $\vec Q_\ga(\bP)$ is also needed for a later bootstrap argument based on the Euler--Lagrange equation.

Another challenge compared to~\cite{BerLanMarRiv26} is that the conserved Noether currents and their potentials lie in spaces of increasingly negative Sobolev order as $n$ increases.
The existence and estimates of the potentials thus require higher negative-order elliptic estimates with coefficients in the critical Sobolev space $L^\nf\cap W^{h,2}$.

Our elliptic estimates start from $W^{1,p}$ estimates for divergence-form equations with VMO coefficients.
Difference-quotient arguments yield estimates for higher positive Sobolev orders, while duality and a bounded projection from $W^{k,p}$ onto $W^{k,p}_0$ provide right inverses at negative orders.
We also establish higher-order Morrey-type estimates for divergence-form equations and for $\lap_g^{h-1}$.
The same scheme extends the $W^{1,p}$ solvability theory in~\cite{BerLanMarRiv26} for the Hodge Laplacian to the required higher Sobolev orders. 
In particular, this implies right-inverse estimates for the Hodge--Dirac operator at higher negative Sobolev orders, which will be used to construct the Noether potentials with prescribed exterior derivatives and codifferentials. 

We now return to the Euler--Lagrange equation~\eqref{inteullag}. Taking $\vec K=\lap_g^h\mkt\bP$ and using the product estimates for the lower-order terms $\vec Q_\ga(\bP)$ in~\eqref{inteullag}, we obtain
\[
    d*_g d\vec K\in W^{-h,\frac{2h}{h+1}}+W^{-1,1}\big(B^n,\R^m\ot \bwe^n \R^n\big).
\]
However, the space $W^{-1,1}$ embeds only into $W^{-h,\lf(\frac{2h}{h+1},\nf\rg)}$, while the exponent $\frac{2h}{h+1}$ lies at an endpoint excluded in Lemma~\ref{lm:ellcacciolp} when $\ell=h-1$.
Thus, the Morrey-type estimates do not imply improved regularity directly from~\eqref{inteullag}.

Using the Euler--Lagrange equation and right-inverse estimates for the Hodge--Dirac operator, we construct $\bL\in W^{1-h,(2,\nf)}\big(B^n,\R^m\ot\bwe^2\R^n\big)$ such that
\begin{numcases}{}
 d*_g\bL=*_g\vec V,\label{eq:d*gL=*gV}\\
 d\bL=0.\label{eq:dL0=0int}
\end{numcases}
As in the four-dimensional case, the dilation and rotation invariance, together with the pointwise variation formulas, then yield associated conservation laws, which are used to construct a scalar-valued $2$-form $S$ and a $\bwe^2\R^m$-valued $2$-form $\bR$ satisfying
\begin{align}\label{choiceSR}
 \begin{dcases}
 \delta S=\bL\mkt\dres d\bP+\vt_{\dil},\quad
       &dS=-2\bL\dwe d\bP,\\
 \delta\bR=\bL\wres d\bP+d\bP\we\vec K+\bvt_{\rot},\quad
       &d\bR=-2\bL\ovs{\ovwe}{\we}d\bP.
 \end{dcases}
\end{align} 
Here $\de$ is the \textit{codifferential} with respect to $g$, and $\vt_{\dil}$, $\bvt_{\rot}$ are lower-order terms; see Section~\ref{sec:conlaws}. 
Combined with the structural identities in Section~\ref{sec:strid}, the choices of $dS$ and $d\bR$ in~\eqref{choiceSR} will lead to cancellations among the principal terms and yield an estimate for $\vec K=\lap_g^h \bP$. 
Applying the aforementioned Morrey-type estimate for $\lap_g^{h-1}$ then gives an estimate for $\lap_g\bP$. Now the harmonic coordinate condition implies
\[
    \lap_g\bP=-g^{ij} \p_i\p_j\bP.
\]
Differentiating this identity and rewriting the resulting equation in divergence form for $D\bP$ will yield an estimate for $D^2\bP$, which closes the iteration. Here we note that the product structure of the lower-order terms is essential: for instance, we have
\[
     \|\vt_{\dil}\|_{W^{2-h,\lf(\frac{2h}{h+1},2\rg)}(B^n)}+\|\bvt_{\rot}\|_{W^{2-h,\lf(\frac{2h}{h+1},2\rg)}(B^n)}\le C(\La,n,\bc)\|D^2\bP\|_{W^{h-1,2}(B^n)}^2,
\]
where $\La$ is the ellipticity constant of $g$.
After restricting to a sufficiently small ball and rescaling, we may assume that $\|D^2\bP\|_{W^{h-1,2}(B^n)}$ is small enough to close the Morrey iteration. Once the final Sobolev--Lorentz--Morrey estimate for $D^2\bP$ is obtained, the integrability will be improved by the Riesz potential estimate~\cite{Adams75}, and the smoothness will follow by a bootstrap argument. A classical result by Morrey~\cite{Morrey58} then implies the real-analyticity of $\bP$.

The paper is organized as follows.
Section~\ref{sec-preliminaries} records basic facts on covariant derivatives, as well as Sobolev--Lorentz product estimates used to control the nonlinear terms.
Section~\ref{sec:ellestcri} establishes the elliptic and Hodge--Dirac estimates, and Section~\ref{sec:strid} derives the structural identities.
Section~\ref{sec:pfmainThm} combines these ingredients to prove the main theorem.

\subsubsection*{Acknowledgements.}
The author is grateful to his PhD supervisor, Tristan Rivi\`ere, for introducing him to the higher-dimensional regularity problem in 2024. After establishing all the mathematical results and obtaining a first draft, the author used ChatGPT-6 pro as an auxiliary editorial tool to obtain suggestions for revising and shortening the exposition in Section~\ref{sec:intro}.

\section{Preliminaries}\label{sec-preliminaries}
\subsection{Notation} \label{geono}
\
\vskip5pt
\begin{enumerate}[label=\textbullet, leftmargin=2em, itemindent=0pt, itemsep=1ex]
\item We write $\p_{x^i}= \frac \p{\p x^i}$, and           abbreviate to $\p_i$ when there is no ambiguity. 
\item Let $U\subset \R^n$ be an open subset and $f\colon U\to \R$. A vector $\al=(\al_1,\dots,\al_n)\in \N_0^n$ is called a \textit{multi-index} of order $|\al|=\sum_{i=1}^n \al_i$. We write 
\begin{equation}\label{eq:defDalf}
    D^\al f\coloneq \p_{x^1}^{\al_1}\cdots \p_{x^n}^{\al_n} f.
\end{equation}
Let $k\in \N_0$. We define
$$D^{k} f\coloneq \{D^\al f\colon |\al|=k\},\qquad |D^{k} f|\coloneq \bigg(\sum_{|\al|=k} |D^\al f|^2\bigg)^{\frac 12}.$$
\item In the Euclidean space $\R^n$, we define $B_r(x)\coloneq \{y\in \R^n\colon |y-x|<r\}$, and write $B_r\coloneq B_r(0)$, $B^n\coloneq B_1\subset \R^n$. 
    \item We use $C(\alpha,\beta,\dots)$ to denote a positive constant depending on $\alpha,\,\beta,\dots$ only. Similarly, when a term $A$ depends only on $\alpha,\,\beta,\dots$, we write $A=A(\alpha,\beta,\dots)$.
    \item For open sets $U$ and $V$ of a given manifold, we write $V\Subset U$ if $\overline V$ is compact and $\overline V\subset U$.
    \item Let $\Sigma$ be a closed smooth manifold, $E$ be a vector bundle over $\Sigma$, we denote by $L^{p}(\Si,E)$ the space of $L^{p}$ sections of $E$. This convention also applies to other distribution spaces.
\item We denote by $\mathcal L^n$ the Lebesgue measure on $\R^n$. For $U\subset \R^n$ with $\mathcal L^n(U)<\infty$, $f\in L^1(U)$, we write $\fint_U f\coloneq(\mathcal L^n(U))^{-1}\int _U f$.
\item For an open set $U\subset \R^n$, we denote by $\mathcal D'(U,\R^m)$ the space of $\R^m$-valued distributions on $U$ with the weak$^*$ topology, and we write $\mathcal D'(U)=\mathcal D'(U,\R)$.
 \item 
    Let $(V, g)$ be an $n$-dimensional inner product space with a positively oriented orthonormal basis $(e_1,\dots,e_n)$. The \textit{Hodge star operator} $*_g$ (see for instance \cite[Sec.~3.3]{Jost17}) is defined as the unique linear operator from $\bwe ^k V$ to $\bwe ^{n-k}V$ satisfying 
    \[
    \alpha\we *_g\, \beta=\langle \alpha,\beta  \rangle_g\, e_1\we\cdots \we e_n, \qquad \text{for all }\, \al,\beta\in \bwe^k V.\]
   When $V=\R^m$ and $g=g_{\text{std}}$ is the standard inner product on $\R^m$, we write $\star$ instead of $*_g$.
   \item  Let $U\subset \R^n$ be an open set and let $\bP\colon U\to \R^m$ be an immersion. We denote by $\g$ the Levi--Civita connection of $g=g_{\bP}$, and write $\g_i\coloneq \g_{\p/\p x^i}$.
   \item 
   Let $U\subset \R^n$ be an open set and let $g$ be a Riemannian metric on $U$. On the space of differential $\ell$-forms on $U$, we define the
\textit{codifferential} and \textit{Hodge Laplacian} (see \cite[Defs.~3.3.1--3.3.2]{Jost17}) by
\begin{equation}
\label{eq:defd*glap}
 \delta\coloneq(-1)^\ell*_g^{-1}d\,*_g=(-1)^{n(\ell+1)+1}*_gd\,*_g\qquad \text{and} \qquad \Delta_g\coloneq d\mko \delta+\delta d.
\end{equation} 
 If $\al\in C^\nf\big(U,\bwe^\ell T^* U\big)$ and $\beta \in C^\nf\big(U,\bwe^{\ell+1} T^* U\big)$ with at least one of $\al,\beta$ compactly supported in $U$, then we have
\begin{equation}\label{eq:d^*adjoi}
    \int_U \lan d\al,\beta \ran_g \,\dvol_g = \int_U \lan \al,\delta\beta\ran_g\,\dvol_g.
\end{equation}
When $\ell=0$, $\lap_g=-\g^i\g_i$ coincides with the nonnegative Laplace--Beltrami operator on functions.
 \item Throughout this paper, we use the Einstein summation convention, and set $\delta_i^j=\de_{ij}=\mathbf 1_{i=j}$. 
 \item We leave out the symbols $\otimes$ for sections of $\bwe \R^m\ot \bwe T^*U$ when there is no ambiguity. For instance, we write $\p_i\bP \,dx^j=\p_i\bP\ot dx^j$.
\item Let $U\subset\R^{n}$ be open and $\bP\colon U\to\R^{m}$ be an immersion. Set $g=g_{\bP}=\bP^*g_{\text{std}}$, where $g_{\std}$ is the Euclidean metric on $\R^m$. Denote 
\begin{gather}
\begin{gathered}\label{defgdvol}
   g_{ij}=\p_i\bP\cdot \p_j\bP,\quad (g^{ij})=(g_{ij})^{-1},\quad  \det g=\det(g_{ij}), \\
   d\textup{vol}_g=(\det g)^{\frac 12}dx^1\we\cdots \we dx^n.
   \end{gathered}
\end{gather}
The pullback metric induces a pairing on $T^* U$ with $\lan dx^i,dx^j \ran_g=g^{ij}$; we denote by $|\cdot|_{g}$
the associated pointwise norm. We also write $|\cdot|_{\R^{m}}$ for the
Euclidean norm in $\R^{m}$ (abbreviated to $|\cdot|$ when unambiguous).

Let $\pi_N$ denote the orthogonal projection onto the normal bundle of $\bP(U)\subset \R^m$. We define the \textit{Gauss map} and \textit{second fundamental form} 
\begin{gather}
\begin{gathered}\label{not-met}
\bn_{\bP}\coloneq \star \,\frac{\p_1 \bP \we \cdots \we  \p_n \bP}{|\p_1 \bP\we\cdots\we\p_n \bP|},\\
\bII_{\bP}(X,Y)\coloneq\pi_N X(d\bP(Y)),\qquad \text{for all } X,Y\in T_p U\text{ and all } p\in U.
\end{gathered}
\end{gather}
We shall often omit the subscript ${\bP}$ when there is no ambiguity. We denote 
\[\bII_{ij}\coloneq \bII\Big(\frac \p{\p x^i},\frac{\p}{\p x^j}\Big) \quad\text{and}\quad \bII_i^j\coloneq g^{jk}\bII_{ik}.
\]
The \textit{mean curvature vector} is defined by
\[
\bH\coloneq \frac 1n \mkt \mbox{tr}_{g}\,\bII=\frac 1n\mkt g^{ij}\,\bII_{ij}.
\]

 \item  We define a bilinear map $\dwe\colon\big( \R^m\ot\bwe^{k_1} T^*_p U\big)\times\big(\R^m\ot\bwe^{k_2} T^*_pU\big) \rightarrow  \bwe^{k_1+k_2} T^*_pU$ by 
 \[
              (\vec u_1 \mkt dx_I)\dwe(\vec u_2\mkt dx_J)=( \vec u_1\cdot  \vec u_2) \mkt dx_I\wedge dx_J,
\]
where $\vec u_1,\vec u_2\in\R^m$, $I,J $ are multi-indices. If one of $k_1,k_2$ is $0$, we use $\cdot$ instead of $\dwe$ for convenience. Similarly, for $\vec v,\vec v_1,\vec v_2\in \bwe\R^m$ and multi-indices $I,J$, we define bilinear maps $\we$ and $\overset{\ovwe}{\wedge}$ by 
\begin{align}\label{defbiopes}\begin{aligned}
    dx_I \we (\vec v\, dx_J)&=\vec v\, dx_I \we dx_J,\\
         \vec v_1 \we (\vec v_2\mkt dx_J)&=( \vec v_1 \we  \vec v_2)\mkt  dx_J,\\
         (\vec v_1 \mkt dx_I)\overset{\ovwe}\wedge(\vec v_2\mkt dx_J)&=( \vec v_1 \we  \vec v_2) \mkt dx_I\wedge dx_J.
         \end{aligned}
\end{align}
\item 
Unless explicitly stated otherwise, for an open set $U\subset \R^n$, we identify $T_x^*U$ with $\R^n$ for $x\in U$, and write $D$, $|\cdot|$ for the flat Euclidean derivative and
pointwise tensor norms. We reserve $\g$, $\lan\cdot,\cdot \ran_g$, and $|\cdot|_g$ for the metric $g$. 
For a tensor field $A=(A_{I})$, we say $A\in W^{k,p}(U)$ if $A_{I}\in W^{k,p}(U)$ for each $I=(i_1,\dots,i_s)\in \{1,\dots,n\}^s$ and we set
\begin{equation}\label{eq:convsonab}
    |DA|\coloneq\bigg(\sum_{i=1}^n\sum_{I} |\partial_i\mko A_{I}|^2\bigg)^{\frac 12},\qquad 
    \|A\|_{W^{k,p}(U)}\coloneq\bigg(\sum_{I} \|A_I\|^2_{W^{k,p}(U)}\bigg)^{\frac 12}.
\end{equation}
The same convention applies to any Banach distribution space, e.g., $W^{k,(p,q)}$.
\end{enumerate}
\subsection{Iterated covariant derivatives}\label{sec:cov}
\vskip5pt
\begin{Dfi}\label{def:cov}
    Let $U\subset \R^n$ be an open set and let $\bP\colon U\to \R^m$ be an immersion. We denote by $\g$ the Levi--Civita connection of $g=g_{\bP}$. Let $A=A_{j_1\dots j_s} \mko dx^{j_1}\ot\cdots \ot dx^{j_s}$ be a $(0,s)$-tensor field on $U$. For $k\in \N$, we define the $(0,k+s)$-tensor field $\g^{(k)}A$ inductively by:
    \begin{align}\begin{aligned}\label{eq:defitecov}
       & \g^{(0)}A=A,\\[0.3ex]
       &\big(\g^{(k)}A\big)_{ i_1\dots i_k j_1\dots j_s}\coloneq \big(\g_{i_1}(\g^{(k-1)}A)\mko\big)_{i_2\dots i_kj_1\dots  j_{s}}, \quad k\ge 1.
        \end{aligned}
    \end{align}
    We also use the notation
    \begin{equation}\label{eq:notitecov}
        \g_{i_1\dots i_k}A_{j_1\dots j_s}\coloneq \big(\g^{(k)}A\big)_{ i_1\dots i_k j_1\dots j_s}.
    \end{equation}
 We define the $\R^m$-valued $(0,k+2)$-tensor field $\g^{(k)} \mko\bII$ by the same inductive formula, where, by abuse of notation, $\g$ denotes the connection on $\R^m$-valued tensor fields induced by the Levi--Civita connection of $g_{\bP}$ on $T^* U$ and the flat connection $D$ on $\R^m$.
Similarly, the \textit{iterated normal covariant derivatives} of the second fundamental form are defined by
\begin{equation}
\begin{aligned}\label{defnorcov}
        (\g^\perp)^{(0)}\mko\bII&\coeq\bII,\\ \big(\mko(\g^\perp)^{(k)}\mko  \bII\big)_{ i_1\dots i_{k+2}}&\coloneq \pi_N \Big(\g_{i_1}\big(\mko(\g^\perp)^{(k-1)}\mko\bII\big)\mko\Big)_{i_2\dots i_{k+2}}, \quad k\ge 1.
        \end{aligned}
\end{equation}
\end{Dfi}

In the notation~\eqref{eq:defitecov}--\eqref{eq:notitecov}, the covariant derivative on tensor fields satisfies the usual Leibniz rule:
\begin{align}\label{eq:Leibru}
    \g_i(A_{i_1\dots i_k} B_{j_1\dots j_s})=B_{j_1\dots j_s}\g_i A_{i_1\dots i_k} +A_{i_1\dots i_k}\g_i B_{j_1\dots j_s}.
\end{align}
Expanding the identity $\g^{(k)}(\g\bP\cdot \bII)=0$ yields the following lemma.
\begin{Lm}\label{lm:gpdot2ff}
    For each $k\in \N^+$, the $(0,k+3)$-tensor $\g\bP\cdot \g^{(k)}\mko\bII$ is a linear combination, up to permutations of indices, of the tensors 
    $$ \g^{(\ell)}\mko\bII\cdot \g^{(k-1-\ell)}\mko\bII,\qquad 0\le\ell \le k-1. $$
\end{Lm}
We next compare $\g^{(k)}\mko\bII$ and $(\g^\perp)^{(k)}\mko\bII$ for $k\in\N^+$. Since $(\g^\perp)^{(k-1)} \mko\bII$ is normal-valued, we compute 
$$
    \g_i \big(\mko(\g^\perp)^{(k-1)}\mko \bII\big) \cdot \p_j\bP=-\bII_{ij}\cdot (\g^\perp)^{(k-1)}\mko \bII. $$
It follows that
$$
    (\g^\perp)^{(k)}\mko \bII= \g  \big(\mko(\g^\perp)^{(k-1)}\mko \bII\big)+  dx^i\ot \Big(\big(\bII_i^j \cdot  (\g^\perp)^{(k-1)}\mko \bII\big) \mko\p_j\bP\Big).$$
Since $\g^{(2)}\bP= \bII$, an induction argument together with~\eqref{eq:Leibru} implies that for every $k\ge 1$, we can write 
\begin{equation}\label{naperpdif}
     (\g^\perp)^{(k)}\mko \bII=\g^{(k)}\mko \bII +\sum_{r\in R_k} a_{r,k}\mko \vec P_{r,k}.   
\end{equation}
Here $R_k$ is a finite index set, $a_{r,k}\in \R$, and each $\vec P_{r,k}$ is a partial contraction (with $k+2$ free indices) in the form (see~\cite{Alexakis12,mondino2018} for notation on contractions):
\begin{equation}\label{parrepdif}
    \text{pcontr}\Big(\big(\g^{(k_1)}\mko\bII\cdot \g^{(k_2)}\mko\bII\big)\ot\cdots \ot \big(\g^{(k_{2s-1})}\mko\bII\cdot \g^{(k_{2s})}\mko\bII\big)\ot \g^{(k_0)}\bP \Big),
\end{equation}
for suitable nonnegative integers $s$ and $k_i$ ($0\le i\le 2s$) depending on $r$. Moreover, we have 
\begin{align}\label{condsksgp-g}
    \begin{dcases}
    1\le k_0\le k,\\[1ex]
    \forall 1\le i\le 2s,\qquad 0\le k_i\le k-1,\\
    \sum_{i=0}^{2s} (k_i+1)=k+3.
    \end{dcases}
\end{align} 
This shows that the class of geometric functionals in~\eqref{eq:defE} is unchanged when $\g$ in~\eqref{eq:comcon} and~\eqref{eq:defE} is replaced by $\g^\perp$.

Now let $x\in U$. In $g$-geodesic normal coordinates centered at $x$, each $D^\al (g_{ij})(x)$ with $1\le i,j\le n$ and $|\al|\ge 2$ can be expressed as a universal polynomial in the components at $x$ of the Riemann curvature tensor of $g$ and its iterated covariant derivatives. We then expand the iterated covariant derivatives using the Christoffel symbols and~\eqref{Gaseq}. Thus, for any $k\ge 3$, indices $1\le i_1,\dots,i_k\le n$, and any $\R^m$-valued function $\vec f$ defined on $U$, we can write
\begin{equation}\label{difcovcan}
\p_{i_1}\cdots \p_{i_k}\vec f(x)=\g_{i_1\dots i_k}\vec f(x)+\vec P_{i_1\dots i_k}\Big(\bII(x),\dots, \g^{(k-3)}\mko \bII(x),\g \vec f(x),\dots,\g^{(k-1)}\vec f(x)\Big),
\end{equation}
where the ambient components of $\vec P_{i_1\dots i_k}$ are universal polynomials in the components of the iterated covariant derivatives of $\vec f$ and $\bII$ at $x$. 
Setting $\vec f=\bP$ in~\eqref{difcovcan} and combining this with Lemma~\ref{lm:gpdot2ff} and the invariance-theory arguments in~\cite{Atiyah73,grahamkuo26}, we obtain the classification of geometric quantities appearing in~\eqref{eq:firdefE}.
\subsection{Hodge star and contraction operators}\label{sec:uselem}
\
\vskip5pt
We define the \textit{interior product} and \textit{first order contraction} between multivectors (see also \cite[Sec.~1.5]{Federer96} and \cite[Sec.~I]{Riv08}).
\begin{Dfi}\label{defresbul}
    Let $(V,g)$ be a finite-dimensional inner product space. For $\alpha\in \bwe^pV$ and $\beta\in \bwe^q V$ with $q\le p$, we define the \textit{interior product} $\alpha\,\resg \beta\in \bwe^{p-q}V$ satisfying 
    \begin{align}\label{eq:defres}
    \langle \alpha\,\resg \beta,  \gamma\rangle_g=\langle \alpha, \beta\wedge \gamma\rangle_g,\qquad \text{for all }\gamma\in \bwe^{p-q}V. 
    \end{align}
For $q>p$, set $\al\, \resg \beta\coloneq0$. We also define the \textit{first order contraction} $\bulg \mko \colon \bwe^{p} V\times \bwe^{q} V\to \bwe^{p+q-2} V$ as follows. For $\al\in \bwe^p V$ and $\beta\in V$, set $\al \bulg \beta\coloneq\al\mko\resg \beta$; and for $\beta\in \bwe^{q_1} V$, $\ga\in \bwe^{q_2} V$, it holds that
\begin{align}\label{eq:defbul}
    \al\bulg(\beta \we \ga)=(\al\bulg \beta) \we \ga+(-1)^{q_1q_2}  (\al \bulg \ga)\we \beta.
\end{align}
\end{Dfi}
    In particular, if $\alpha,\beta \in \bwe^pV$, then we have $\alpha\,\resg\,\beta=\langle \alpha,\beta \rangle_g$. In addition, for $\al\in \bwe^p V$, $\beta\in\bwe V$, $v\in V$, we have the following identities:
   \begin{numcases}{}   *
   _g\,(\alpha\wedge \beta )=(*_g\,\alpha)\,\resg \beta \label{*_gcommwed},\\[0.4ex]
   (\al\we \beta)\,\resg v=(\al\,\resg v) \we \beta+(-1)^p\, \alpha\we (\beta\,\resg v),\label{prodrule}
   \end{numcases}\smallskip
  A consequence of \eqref{*_gcommwed} is that, for $\al\in \bwe^p V$ and $\beta\in \bwe^q V$ we have
  \begin{align}\label{*_gcomres}
      *_g\, (\al\, \resg \beta)=\begin{dcases}
          (*_g\, \al) \we \beta,& \text{if }\dim(V) \text{ is odd},\\[0.4ex]
          (-1)^q (*_g\, \al) \we \beta, \quad& \text{if }\dim(V) \text{ is even}.
      \end{dcases}
  \end{align} 
  Combining~\eqref{eq:defbul} and \eqref{prodrule}, for $u_1,u_2,v_1,v_2\in V$, we also obtain
\begin{align}\label{bul2vecs}
\begin{aligned}
    (u_1\we u_2) \bulg (v_1\we v_2)
    &=\lan u_1,v_1\ran_g \,u_2\we v_2+\lan u_2,v_2\ran_g \,u_1\we v_1\\
    &\quad-\big(\lan u_1,v_2\ran_g \mkt u_2\we v_1+\lan u_2,v_1\ran_g \,u_1\we v_2\big).
    \end{aligned}
\end{align}

When $V=\R^m$ with the Euclidean metric, we write $\,\res,\sbul$ instead of $\,\resg,\bulg$, and denote by $\cdot$ the inner product on $\bwe \R^m$ induced by the Euclidean metric. 

\subsection{Sobolev--Lorentz spaces and product inequalities}\label{sec:soblor}
\
\vskip5pt
We equip sums and intersections of compatible Banach spaces with their standard norms; see for instance~\cite[Ch.~5]{Bennett88}. We write $L^p+L^q(U)\coloneq L^p(U)+L^q(U)$. The same convention applies to Sobolev--Lorentz spaces.

\begin{Dfi}[Lorentz spaces]\label{def-Lor}
Let $U\subset\R^{n}$ be a measurable set. Given a measurable function $f\colon U\to\R$, we define the distribution function and the decreasing rearrangement of $f$ as
\[
d_{f}(\lambda)\coloneq\mathcal{L}^{n}\bigl\{x\in U:|f(x)|>\lambda\bigr\},
\qquad
f^{*}(t)\coloneq\inf\bigl\{\lambda\ge 0 : d_{f}(\lambda)\le t\bigr\}.
\]
For $1\le p<\infty$ and $1\le q\le\infty$, we define the Lorentz quasinorm
\[
|f|_{L^{p,q}(U)}
\coloneq\bigl\|t^{\frac1p}f^{*}(t)\bigr\|_{L^{q}(\R_{+},\,dt/t)}
    =p^{\frac1q}\bigl\|\lambda\,d_{f}(\lambda)^{\frac1p}\bigr\|_{L^{q}(\R_{+},\,d\lambda/\lambda)}.
\]
The Lorentz space $L^{p,q}(U)$ consists of all measurable $f$ with
$|f|_{L^{p,q}(U)}<\infty$. 
When $p>1$, the Lorentz quasinorm is equivalent to a norm (see for instance \cite[Ch.~4, Thm.~4.6]{Bennett88}), which we denote by
$\|\cdot\|_{L^{p,q}(U)}$.
\end{Dfi}
\begin{Dfi}
    \label{dfi-So-Lor}
    Let $k\in \N^+$ and $U\subset \R^n$ be a bounded open set. For $1< p< \infty$, $1\le q\le \infty$, set 
    \[
    W^{k,(p,q)}(U)\coloneq\Big\{f\in L^{p,q}(U)\colon \p^\al f\in L^{p,q}(U) \text{ for each } 0\le |\al|\le k\Big\}.
    \]
   We define the scale-homogeneous Sobolev--Lorentz norms
    \begin{equation}\label{eq:uniSobnm}
        \|f\|_{W^{k,(p,q)}(U)}\coloneq 
              \sum_{i=0}^k \mca L^n(U)^{-\frac in}\mkt\|D^{k-i}f\|_{L^{p,q}(U)}.
    \end{equation}
    The same convention applies to Sobolev norms. In particular, if $1<p<\frac nk$, then H\"older's inequality for Lorentz spaces~\cite[Thm.~4.5]{Hunt} and the Sobolev--Lorentz embeddings (see e.g.~\cite[Sec.~2.1]{deLonGas21}) imply that for all $r>0$ and $f\in W^{k,(p,q)}(B_r)$,
    \begin{align}\label{equivSobnorm}
    C_1\|f\|_{W^{k,(p,q)}(B_r)}\le \sum_{i=0}^k \|D^{k-i}f\|_{L^{\frac{pn}{n-pi},q}(B_r)}\le C_2\|f\|_{W^{k,(p,q)}(B_r)},
    \end{align}
    where $C_1,C_2>0$ are constants independent of $r$ and $f$.
    We also define the negative-order Sobolev--Lorentz space and its norm by
\begin{align}\label{negSobnm}
\begin{aligned}
   &W^{-k,(p,q)}(U)\coloneq\bigg\{f\in\mathcal D'(U)\colon f=\sum_{|\alpha|= k} D^\alpha f_{\alpha} \mbox{ for some }\{f_{\al}\}\subset L^{p,q}(U)\bigg\},\\
       &\|f\|_{W^{-k,(p,q)}(U)}\coloneq\inf \bigg\{\sum_{|\alpha|= k}\|f_\alpha\|_{L^{p,q}(U)}\colon f=\sum_{|\alpha|= k} D^\alpha f_{\alpha}\bigg\}.
       \end{aligned}
\end{align}
For $1\le p\le \nf$, the spaces $W^{-k,p}(U)$ and $W^{k,p}(U)$ are defined analogously. \\
When $U$ is bounded, we denote by $W_0^{k,p}(U)$ the closure of $C_c^\infty(U)$ in $W^{k,p}(U)$, equipped with the norm $\|f\|_{W^{k,p}_0(U)}\coloneq\|D^k f\|_{L^p(U)}$.
     \end{Dfi}

Let $U\subset \R^n$ be an open set, and let $k\in \N_0$ with $k\le n-1$. For any $p\in [\frac n{n-k}, \nf)$, $q\in (1,\frac nk]$ (we set $\frac n0\coloneq \nf$), and $r\in [1,\nf]$, we define 
\begin{align}
\begin{aligned}\label{defX1X2}
  X^{-k,(p,r)}(U)&\coloneq \begin{cases}W^{-k,(p,r)}(U) &\text{if }p>\frac n{n-k},\\
   L^1+W^{-k,p}(U) &\text{if }p=\frac n{n-k},
   \end{cases}\\[0.3ex] 
   X^{k,(q,r)}(U)&\coloneq \begin{cases}
        W^{k,(q,r)}(U) &\text{if }q<\frac nk,\\
        L^\nf \cap W^{k,q}(U) &\text{if }q=\frac nk.
    \end{cases}
    \end{aligned}
\end{align}
We record the following estimates for products in Sobolev--Lorentz spaces, which follow by the Leibniz-rule arguments in~\cite[Thms.~2.37--2.38 and~2.54--2.55]{deLongueville19} together with H\"older's inequality for Lorentz spaces~\cite[Thm.~4.5]{Hunt} and the Sobolev--Lorentz embeddings (see e.g.~\cite[Sec.~2.1]{deLonGas21}).
\begin{Lm}\label{lm:proinene}
Let $U\subset \R^n$ be a bounded Lipschitz domain, and let $k\in \N_0$ with $k\le n-1$. 
Let
$$
\frac n{n-k}\le p_1,p_2<\nf,\qquad 1< q_0,q_1,q_2\le \frac nk,\qquad 1\le r_0,r_1,r_2\le \nf.
$$ 
We assume that
\[\frac 1{p_1}+\frac 1{q_0}=\frac 1{p_2}+\frac kn, 
\qquad \frac 1{q_1}+\frac 1{q_0}=\frac 1{q_2}+\frac kn,
\qquad \frac 1{r_2}\le \frac 1{r_0}+\frac 1{r_1}.\]
For the endpoint cases, we further assume that for any $j\in\{0,1,2\}$ and $i\in \{1,2\}$,
\[q_j=\frac nk \;\;\Longrightarrow\;\; r_j=\nf, \qquad \qquad  p_i=\frac n{n-k} \;\;\Longrightarrow \;\;r_i=1.
\] Then for all $T\in X^{-k,(p_1,r_1)}(U)$, $f_0\in X^{k,(q_0,r_0)}(U)$, and $f_1\in X^{k,(q_1,r_1)}(U)$, the following inequalities hold:
\begin{align}\label{eq:fTproinene}
      \|f_0T\|_{X^{-k,(p_2,r_2)}(U)}&\le C(U,p_1,r_1, q_0,r_0) \mko\|T\|_{X^{-k,(p_1,r_1)}(U)} \|f_0\|_{X^{k,(q_0,r_0)}(U)},\\
\label{eq:ffproinepo}
       \|f_0f_1\|_{X^{k,(q_2,r_2)}(U)}&\le C(U,q_0,r_0,q_1,r_1)\mko \|f_1\|_{X^{k,(q_1,r_1)}(U)} \|f_0\|_{X^{k,(q_0,r_0)}(U)}.
\end{align}
\end{Lm}

For $0\le k\le n-1$, $p\in [\frac n{n-k}, \nf)$, and $q\in (1,\frac nk]$, we denote \[
    X^{k,q}(U)\coloneq X^{k,(q,q)}(U),\qquad X^{-k,p}(U)\coloneq X^{-k,(p,p)}(U).
\]
Combining~\eqref{eq:fTproinene}--\eqref{eq:ffproinepo} with the Sobolev--Lorentz embeddings, we obtain the following multiproduct estimates.

\begin{Co}\label{cor:multiproinene}
Let $n=2h\ge4 $ with $h\in \N^+$, and let $U\subset\R^n$ be a bounded Lipschitz domain.
Let $N\in\N^+$ and $a_1,\ldots,a_N\in\N_0$ satisfy
\[
    a\coloneq\sum_{i=1}^N a_i\le n.
\]
Then for all $f_i\in X^{h-a_i,2}(U)$, $1\le i\le N$, their product belongs to $X^{h-a,2}(U)$, with
\begin{equation}\label{eq:multiposine}
    \bigg\|\prod_{i=1}^N f_i\bigg\|_{X^{h-a,2}(U)}\le C(U,N)\prod_{i=1}^N\|f_i\|_{X^{h-a_i,2}(U)}.
\end{equation}
If we assume in addition that $a_i\le a-1$ for all $1\le i\le N$, then we have
\begin{equation}\label{eq:multicritgain}
    \bigg\|\prod_{i=1}^N f_i\bigg\|_{X^{h+1-a,\lf(\frac{2h}{h+1},2\rg)}(U)}\le C(U,N)\prod_{i=1}^N\|f_i\|_{X^{h-a_i,2}(U)}.
\end{equation}
\end{Co}

\begin{proof}
If $a\le h$, under the corresponding assumptions, successive applications of~\eqref{eq:ffproinepo} with $k=h-a$ and $k=h+1-a$ respectively, combined with the Sobolev--Lorentz embeddings, yield the estimates~\eqref{eq:multiposine}--\eqref{eq:multicritgain}.

Thus we assume $a>h$. If $a_i\le h$ for all $1\le i\le N$, then by Sobolev--Lorentz embeddings and H\"older's inequality, we obtain
\begin{equation}\label{proLna}
     \bigg\|\prod_{i=1}^N f_i\bigg\|_{L^{\frac na}(U)}\le C(U,N)\prod_{i=1}^N\|f_i\|_{X^{h-a_i,2}(U)}.
\end{equation}
Since $h<a\le n$, the estimates~\eqref{eq:multiposine}--\eqref{eq:multicritgain} then follow from~\eqref{proLna} and Sobolev--Lorentz embeddings.

It remains to consider the case where there exists exactly one $a_i>h$.
After relabeling, we write $a_1=h+\ell$ and set $F\coloneq\prod_{i=2}^N f_i$.
Since $\sum_{i=2}^N a_i=a-h-\ell\le h-\ell$, the positive-order case proved above gives
\begin{equation}\label{estprodf2fN}
    \|F\|_{X^{n-a+\ell,2}(U)}\le C(U,N)\prod_{i=2}^N\|f_i\|_{X^{h-a_i,2}(U)}.
\end{equation}
Since $a\le n$, combining~\eqref{eq:fTproinene} and~\eqref{estprodf2fN} with the embedding $X^{n-a+\ell,2}(U)\hookrightarrow X^{\ell,\frac{2n}{2a-n}}(U)$ yields
\[
    \|f_1F\|_{X^{-\ell,(\frac n{a-\ell},2)}(U)}\le C(U) \|f_1\|_{X^{-\ell,2}(U)}\|F\|_{X^{\ell,\frac{2n}{2a-n}}(U)} \le C(U,N)\prod_{i=1}^N\|f_i\|_{X^{h-a_i,2}(U)}.
\]
Since $a\ge a_1=h+\ell$ and $f_1 F=\prod_{i=1}^N f_i$, the embedding $X^{-\ell,(\frac n{a-\ell},2)}(U)\hookrightarrow X^{h-a,2}(U)$ then implies~\eqref{eq:multiposine}.
Under the additional assumption $a_1\le a-1$, we have $h+1-a\le -\ell$, hence the estimate~\eqref{eq:multicritgain} follows from the embedding $X^{-\ell,(\frac n{a-\ell},2)}(U)\hookrightarrow X^{h+1-a,\lf(\frac{2h}{h+1},2\rg)}(U)$.
\end{proof}
In contrast to the preceding product estimates involving $W^{k,q}$ with $q\le n/k$, the following estimates are based on $W^{k,q}$ with $q>n/k$ and will be used in the integrability and differentiability bootstraps in Section~\ref{sec:ndreg}.
\begin{Lm}\label{lm:multinonexact}
Let $U\subset\R^n$ be a bounded Lipschitz domain, $\sigma\in\N^+$, and let $2<q<\nf$ with $q>n/\si$.
Let $N\in\N^+$ and $a_1,\ldots,a_N\in\N_0$ satisfy
\[
    a\coloneq\sum_{i=1}^N a_i\le2\sigma.
\]
Then for all $f_i\in W^{\sigma-a_i,q}(U)$, $1\le i\le N$, we have
\begin{equation}\label{eq:multinoex}
    \bigg\|\prod_{i=1}^N f_i\bigg\|_{W^{\sigma-a,q}(U)}\le C(U,\sigma,N,q)\prod_{i=1}^N\|f_i\|_{W^{\sigma-a_i,q}(U)}.
\end{equation}
If we assume in addition that $a_i\le a-1$ for all $1\le i\le N$, then for any $1<p\le q$ satisfying
\begin{equation}\label{condmultinonexact}
    \frac1p>\frac2q-\frac{\sigma-1}{n},
\end{equation}
we have
\begin{equation}\label{eq:multinonexact}
    \bigg\|\prod_{i=1}^N f_i\bigg\|_{W^{\sigma+1-a,p}(U)}\le C(U,\sigma,N,p,q)\prod_{i=1}^N\|f_i\|_{W^{\sigma-a_i,q}(U)}.
\end{equation}
\end{Lm}

\begin{proof}
To prove~\eqref{eq:multinoex}, by induction, it suffices to consider the case $N=2$.
Using the conditions $q>2$, $\sigma>n/q$, and $a\le 2\sigma$, we obtain
\begin{numcases}{}
    \sigma-a_1+\sigma-a_2-\big(\sigma-(a_1+a_2)\mko\big)=\sigma>\frac nq,\\
    \sigma-a_1+\sigma-a_2\ge 0>n\Big(\frac 2q-1\Big).\label{conp1+p2>n2q}
\end{numcases}
The inequality~\eqref{eq:multinoex} then follows from~\cite[Thm.~A.1]{Behzadan21} combined with the Sobolev extension; see for instance~\cite{Rych99}.

We now prove~\eqref{eq:multinonexact}.
If $\sigma=1$, then the conditions $\sum_{i=1}^N a_i=a\le 2\sigma=2$ and $a_i\le a-1$ imply that $a=2$ and exactly two of $(a_i)_{1\le i\le N}$ equal $1$.
In this case, the inequality~\eqref{eq:multinonexact} follows from H\"older's inequality, the embedding $W^{1,q}(U)\hookrightarrow L^\nf(U)$, and $p<\frac q2$.

We thus assume $\sigma\ge2$.
It suffices to consider $p\ge q/2$, since the remaining cases follow from the estimate with $p=q/2$.
As in the proof of~\eqref{eq:multinoex}, we first assume $N=2$.
Since $a_i\le a-1$ for $i=1,2$, we have $\sigma+1-a\le \min\{\sigma-a_1,\sigma-a_2\}$.
By the assumptions~\eqref{condmultinonexact} and $p\ge q/2$, we also obtain
\begin{equation}\label{k-1>n(2q-1p}
    \sigma-a_1+\sigma-a_2-\big(\sigma+1-a\big)=\sigma-1>n\Big(\frac 2q-\frac 1p \Big)\ge 0.
\end{equation}
Using the conditions $1<p\le q$ and~\eqref{conp1+p2>n2q}--\eqref{k-1>n(2q-1p}, the estimate~\eqref{eq:multinonexact} for $N=2$ then follows from~\cite[Thm.~A.1]{Behzadan21}.
For the case of general $N\in \N^+$, after relabeling, we assume $1\le a_1\le a-1$ and set $F\coloneq\prod_{i=2}^N f_i$.
Then by~\eqref{eq:multinoex}, we have
\begin{equation}\label{eq:prodF}
    \|F\|_{W^{\sigma-a+a_1,q}(U)}\le C(U,\sigma,N,q)\prod_{i=2}^N\|f_i\|_{W^{\sigma-a_i,q}(U)}.
\end{equation}
The inequality~\eqref{eq:multinonexact} for general $N$ follows from~\eqref{eq:prodF} and the proved case for $N=2$.
\end{proof}

\section{Elliptic estimates with coefficients in critical Sobolev spaces}\label{sec:ellestcri}  
In this section we establish the elliptic estimates needed for the Noether potentials and the Morrey iteration.
We impose the smallness assumptions~\eqref{smaassaij} and~\eqref{eq:dg_bound} for direct applications in Section~\ref{sec:pfmainThm}; these assumptions may be replaced by suitable local moduli control as in~\cite[Sec.~III]{BerLanMarRiv26}.
Related results concerning $W^{1,n}$ coefficients can be found in~\cite{Miranda63,CruzUribe16,laMan20}.

\subsection{Elliptic estimates for divergence-form equations}\label{SobMorest}
\
\vskip5pt
This subsection extends~\cite[Lem.~II.9]{BerLanMarRiv26} to higher positive and negative Sobolev orders.
We begin with the corresponding global solvability and right-inverse estimates.

\begin{Lm}\label{lm:exisol}
    Let $n\in \N^+$, $k\in \N_0$ with $k\le n-1$, and let $1< p<\frac nk$ (set $\frac nk\coloneq \nf$ if $k=0$). Let $U\subset \R^n$ be a bounded $C^{k+1}$ domain. Suppose $\{a^{ij}\}_{i,j=1}^n\subset L^{\infty}\cap W^{1,n}(U)$ satisfies, for some constant
    $\La\ge 1$,
    \begin{equation}\label{elli}
   |a^{ij}(x)|\le \La\quad \text{ and }\quad   a^{ij}(x)\,\xi_i\mko \xi_j \ge \La^{-1}|\xi|^2,\qquad \text{for a.e. } x\in U \text{ and all } \,\xi\in \R^n.
\end{equation}
Then there exists $\vae=\vae(\La,U,p)>0$ such that the following holds. Suppose that
\begin{equation}\label{smaassaij}
 \sum_{i,j=1}^n \|D\mko a^{ij}\|_{L^n(U)}\le \vae.
\end{equation}
Suppose $a^{ij}\in W^{k,\frac nk}(U)$ and $\eta>0$ is a constant such that
\begin{equation}\label{ashidera}
       \sum_{i,j=1}^n \|D\mko a^{ij}\|_{W^{k-1,\frac nk}(U)}\le \eta,\quad \text{if }k\ge 2.     
\end{equation} 
Then for any $f\in W^{k-1,p}(U)$, there exists a unique $u\in W^{k+1,p}\cap W_0^{1,p}(U)$ solving
\begin{equation}\label{eqsuf}
    \begin{dcases}
    \begin{aligned}
        \p_i(a^{ij} \p_j u)&= f&& \text{ in }\;U,\\
        u&=0&& \text{ on }\,\p U,
        \end{aligned}
    \end{dcases}
\end{equation}
with the estimate 
\begin{equation}\label{estuf}
    \|u\|_{W^{k+1,p}(U)}\le C(\La,U,p,\eta)\mko  \|f\|_{W^{k-1,p}(U)}.
\end{equation}
Moreover, there exists a bounded linear operator $\msc R_k:W^{-k-1,p'}(U)\to W^{-k+1,p'}(U)$ ($p'\coloneq \frac p{p-1}$) such that for all $T\in W^{-k-1,p'}(U)$, there holds\footnote{Here $a^{ij}\mko \p_j ( \msc R_kT)\in \mca D'(U)$ is understood in the sense of Lemma~\ref{lm:proinene}.}
\begin{equation}\label{eqSTdis}
    \p_i\big(a^{ij}\mko \p_j ( \msc R_kT)\mko\big)= T\quad \text{in }\mca D'(U),
\end{equation}
with the estimate
\begin{equation}\label{estST}
    \|\msc R_k T\|_{W^{-k+1,p'}(U)}\le C(\La,U,p,\eta)\mko  \|T\|_{W^{-k-1,p'}(U)}.
\end{equation}
\end{Lm}
\begin{proof}
For $u\in W^{1,p}(U)$ with $1<p<\nf$, set
\begin{equation}\label{defLu}
    Lu\coloneq \p_i(a^{ij}\p_j u)\in W^{-1,p}(U).
\end{equation}
Since $a^{ij}\in W^{1,n}(U)\hookrightarrow \text{BMO}(U)$, by~\cite[Thm.~1.5]{Byun04} and~\eqref{smaassaij}, for small enough $\vae=\vae(\La,U,p)>0$ and any $f\in W^{-1,p}(U)$, there exists a unique
$u\in W^{1,p}_0(U)$ solving~\eqref{eqsuf}
and we have
\begin{equation}\label{VMO-W1p}
    \|u\|_{W^{1,p}(U)}
    \le C(\La,U,p)\mko \|f\|_{W^{-1,p}(U)}.
\end{equation} 
 Now we prove~\eqref{estuf} by induction on $k\ge 0$. The case $k=0$ has been proved. Let $1\le k\le n-1$, and assume the estimate has
been proved at order $k-1$.

Let $1<p<\frac n {k}$, $f\in W^{k-1,p}(U)$, and let $u\in W^{1,p}_0(U)$ solve~\eqref{eqsuf}. We aim to show that $u\in W^{k+1,p}(U)$. Set $\ti p\coloneq \frac {pn}{n-p}$. Since $f\in W^{k-1,p}(U)\hookrightarrow W^{k-2,\ti p}(U)$ and $\ti p<\frac n{k-1}$, by the induction hypothesis, we have
\begin{equation}\label{ind-est}
    \|u\|_{W^{k,\ti p}(U)}
    \le C(\La,U,p,\eta)\mko \|f\|_{W^{k-2,\ti p}(U)}\le  C(\La,U,p,\eta)\mko \|f\|_{W^{k-1,p}(U)}.
\end{equation}
Differentiating the equation $ Lu= f$, we obtain for any $\ell\in \{1,\dots,n\}$, 
\begin{equation}\label{reppiaijpjlu}
L(\p_\ell u)=\p_\ell f-\p_i(\p_\ell a^{ij}\p_j u) \quad \text{in }\mca D'(U).
\end{equation}
Combining~\eqref{smaassaij},~\eqref{ashidera}, and~\eqref{ind-est} with the inequality~\eqref{eq:ffproinepo} gives
\begin{align}\label{piplaijpu}
\begin{aligned}
    \|\p_i(\p_\ell a^{ij}\p_j u)\|_{W^{k-2,p}(U)}
    &\le \|\p_\ell a^{ij}\p_j u\|_{W^{k-1,p}(U)}\\
    &\le C(U,p)\mko\|\p_\ell a^{ij}\|_{W^{k-1,\frac n{k}}(U)}\|\p_j u\|_{W^{k-1,\ti p}(U)}\\
    &\le C(\La, U,p,\eta)\mko  \|f\|_{W^{k-1,p}(U)}.
\end{aligned}
\end{align}
Let $U'\Subset U$ be a subdomain, and let $\chi\in C_c^\nf(U)$ such that $\chi=1$ on $U'$. Equation~\eqref{reppiaijpjlu} implies
\begin{align}\label{exppiapjchipu}
  L(\chi\mko\mko \p_\ell u)=\chi\big(\p_\ell f-\p_i(\p_\ell a^{ij}\p_j u) \mko\big)+\p_i(a^{ij}\p_j\chi \mko \p_\ell u)+\p_i \chi \mko a^{ij}\p_j\p_\ell u.
\end{align}
Similar to~\eqref{piplaijpu}, we obtain 
\begin{equation}\label{estpchiertm}
   \|\p_i(a^{ij}\p_j\chi \mko \p_\ell u)\|_{W^{k-2,p}(U)} +\|\p_i \chi \mko a^{ij}\p_j\p_\ell u\|_{W^{k-2,p}(U)}\le C(\La,U,U',p,\eta) \|f\|_{W^{k-1,p}(U)}.
\end{equation}
Therefore, combining~\eqref{piplaijpu}--\eqref{estpchiertm} yields
\begin{equation*}
    \| L(\chi\mko\mko \p_\ell u)\|_{W^{k-2,p}(U)}\le C(\La,U,U',p,\eta)\mko \|f\|_{W^{k-1,p}(U)}.
\end{equation*}
By~\eqref{reppiaijpjlu} and arguing as in the proof of~\cite[Thm.~4.1]{laMan20}, we obtain $u\in W^{2,p}_{\loc}(U)$, hence $\chi\mko \p_\ell u\in W^{1,p}(U)$. Then by the induction hypothesis, since $p<\frac n{k}<\frac n{k-1}$, we have 
\begin{align*}
  \|u\|_{W^{k+1,p}(U')} 
  &\le C(U)\Big(\|u\|_{L^p(U)} +\sum_{\ell=1}^n\|\chi\mko \p_\ell u\|_{W^{k,p}(U)}\Big)\\
  &\le  C(\La,U,p,\eta)\Big( \|f\|_{W^{k-1,p}(U)}+\sum_{\ell=1}^n
    \| L(\chi\mko\mko \p_\ell u)\|_{W^{k-2,p}(U)}\Big)\\
    &\le C(\La,U,U',p,\eta)\mko \|f\|_{W^{k-1,p}(U)}.
\end{align*}
The boundary regularity of $u$ follows by a similar argument combined with the flattening procedure as in~\cite[Sec.~5.1]{Byun05}. Since $\partial U$ is of class $C^{k+1}$, upon choosing the boundary charts sufficiently small, 
the transformed coefficients $\ti a_{ij}$ belong to $L^\infty\cap W^{k,n/k}$ with small enough $\|D\ti a_{ij}\|_{L^n}$. The same induction may then be applied to tangential difference
quotients, which preserve the Dirichlet condition, and the
remaining highest-order pure normal derivative is recovered from the
transformed equation using uniform ellipticity; see also~\cite[Sec.~6.3.2]{evans}.

A finite covering of $\partial U$ and a partition of unity give $u\in W^{k+1,p}(U)$ and
\begin{equation*}
    \|u\|_{W^{k+1,p}(U)} \le C(\La,U,p,\eta)\mko \|f\|_{W^{k-1,p}(U)}.
\end{equation*}
By induction, the estimate~\eqref{estuf} is proved.

It remains to prove the existence of a solution to~\eqref{eqSTdis} with the estimate~\eqref{estST}. Suppose $k\ge 1$. We define the adjoint operator of $L$ for all $v\in W^{k+1,p}(U)\cap W^{1,p}_0(U)$:
\begin{equation}\label{defL*u}
    L^*v\coloneq \p_i(a^{ji}\p_jv).
\end{equation}
For
$\vp\in W^{k-1,p}_0(U)$, let $v_\vp\in W^{k+1,p}(U)\cap
W^{1,p}_0(U)$ be the unique solution of
\begin{equation}\label{defvphi}
    L^*v_\vp=\vp.
\end{equation}
Applying the estimate~\eqref{estuf} to $L^*$, we obtain
\begin{equation}\label{adj-est}
    \|v_\vp\|_{W^{k+1,p}(U)}\le  C(\La,U,p,\eta)\mko \|\vp\|_{W^{k-1,p}(U)}.
\end{equation}
We fix a bounded linear projection $P\colon W^{k+1,p}(U)\to W_0^{k+1,p}(U)$ (see for instance~\cite[Sec.~1.5.1]{grisvard11}). The projection is chosen independently of $p$, although its operator norm may depend on $p$. Let $T\in W^{-k-1,p'}(U)=\big(W^{k+1,p}_0(U)\big)^*$. Now for any $\vp\in  W^{k-1,p}_0(U)$, we define
\[
   \lan \msc R_k T,\vp\rangle
    \coloneq
    \langle  T, Pv_\vp\rangle .
\]
By~\eqref{adj-est}, we have
\[
    |\langle \msc R_k T,\vp\rangle|
    \le
    C(\La,U,p)\mko \|T\|_{W^{-k-1,p'}(U)}
    \|\vp\|_{W^{k-1,p}_0(U)} .
\]
Thus $\msc R_k T\in W^{-k+1,p'}(U)$, and
\[
    \|\msc R_k T\|_{W^{-k+1,p'}(U)}
    \le
    C(\La,U,p,\eta)\mko \|T\|_{W^{-k-1,p'}(U)} .
\]
Finally, Lemma~\ref{lm:proinene} implies that $L\msc R_k T\in W^{-k-1,p'}(U)$, and we have for any $\psi\in C_c^\infty(U)$,
\[
    \langle L\msc R_k T,\psi\rangle
    =
    \langle \msc R_k T,L^*\psi\rangle
    =
    \langle T,\psi\rangle.
\]
Hence $L\msc R_k T=T$ in $\mca D'(U)$,
which proves the claim.    
\end{proof}
Under the setting of Lemma~\ref{lm:exisol}, we now derive an inhomogeneous Morrey-type estimate for solutions of the equations $\p_i(a^{ij}\mko\p_jS)=T$. The (scale-homogeneous) Sobolev--Lorentz norms are defined in~\eqref{eq:uniSobnm}--\eqref{negSobnm}.
\begin{Lm}\label{lm:ellcacciolp}
  Let $n\in \N^+$, $\ell, k\in \N_0$ with $\ell\le k\le n-1$. Assume that
 \[1\le s,s_0\le \nf,\qquad \frac n{n-k}< q_0< \nf,\qquad 1<p<\frac n{\ell}.
 \] 
Here we set $\frac n0\coloneq \nf$. Let  $\al>0$. Then there exists $\vae_0=\vae_0(\La,n,p,q_0,\al)>0$ such that the following holds. 
Suppose $\{a^{ij}\}_{i,j=1}^n$ satisfies~\eqref{elli}--\eqref{ashidera} for $\vae=\vae_0$ and constants $\La,\eta>0$, with $U=B^n$. 
 Suppose $S\in W^{-k+1,(q_0,s_0)}(B^n)$ satisfies
 \begin{equation}\label{relST}
     \p_i(a^{ij}\mko \p_j S)= T\quad \text{in }\mca D'(B^n).
 \end{equation}
 We obtain the following estimates:
\begin{enumerate}[label=(\roman*),leftmargin=2.5em]
\item \label{posmorest}Assume $\al<\frac np-\ell$. If $T\in W^{\ell-1,(p,s)}(B^n)$, then $S\in W_{\loc}^{\ell+1,(p,s)}(B^n)$, and for all $r\in (0,\frac 12]$, we have
\begin{align}
  \|D S\|_{W^{\ell,(p,s)}(B_r)} &\le C(\La,\al,p,q_0,n,\eta) \Big(\|T\|_{W^{\ell -1,(p,s)}(B_1)}+r^{\al} \|S\|_{W^{-k+1,(q_0,s_0)}(B_1)}\Big),\label{eq:ellcacciop1}\\
   \|D S\|_{W^{\ell,(p,s)}(B_r)} &\le C(\La,\al,p,n,\eta) \Big(\| T\|_{W^{\ell-1,(p,s)}(B_1)}+r^{\al} \|D S\|_{L^{p,s}(B_{3/4})}\Big)\label{eq:ellcacciopg}.
  \end{align}
\item \label{negmorest}
Assume $\al<\frac n{p'}+\ell$. If $T\in W^{-\ell-1,(p',s)}(B^n)$, then
$S\in W_{\loc}^{-\ell+1,(p',s)}(B^n)$, and for all $r\in(0,\frac12]$,
we have
\begin{equation}\label{eq:ellcaccioneggrad}
\begin{aligned}
    \|D S\|_{W^{-\ell,(p',s)}(B_r)}
    \le C(\La,p,q_0,n,\eta,\al)\Big(
    &\|T\|_{W^{-\ell-1,(p',s)}(B_1)}
    +r^\al\|S\|_{W^{-k+1,(q_0,s_0)}(B_1)}\Big).
\end{aligned}
\end{equation}
\end{enumerate}
\end{Lm}
\begin{proof}
   We define the operators $L$ and $L^*$ as in~\eqref{defLu} and~\eqref{defL*u} respectively. Suppose $S\in W^{-k+1,(q_0,s_0)}(B_1)$ satisfies~\eqref{relST}. 
   We first prove part~\ref{posmorest}. Assume that $T\in W^{\ell-1,(p,s)}(B_1)$. By Lemma~\ref{lm:exisol} and interpolation (see~\cite[Lem.~II.8]{BerLanMarRiv26}), there exists a unique $S_0\in W^{\ell+1,(p,s)}(B_1)\cap W^{1,1}_0(B_1)$ solving
\[
    LS_0=T.
\]
Moreover, we have 
\begin{equation}\label{eq:R-pos-est}
    \|DS_0\|_{W^{\ell,(p,s)}(B_1)}
    \le C(p,n)\mko \|S_0\|_{W^{\ell+1,(p,s)}(B_1)}\le 
    C(\Lambda,p,n,\eta)\mko
    \|T\|_{W^{\ell-1,(p,s)}(B_1)} .
\end{equation}
Set $S_1\coloneq S-S_0$, then $LS_1=0$ in $\mca D'(B_1)$. Fix $q_1\in(\frac n{n-k},\nf)$ such that 
\[\frac 1{q_1}>\frac 1p-\frac {\ell+k}{n}.\]
We also fix $q_2\in \big(\frac n{n-k},\min(q_1,q_0)\big)$. By the Sobolev embeddings and~\eqref{eq:R-pos-est}, we have
\begin{equation}\label{S1W-k+1est}
    \|S_1\|_{W^{-k+1,q_2}(B_1)}
    \le
    C(\La,p,q_0,n,\eta)\Big(
        \|S\|_{W^{-k+1,(q_0,s_0)}(B_1)}
        +
        \|T\|_{W^{\ell-1,(p,s)}(B_1)}
    \Big).
\end{equation}
Now we prove that $S_1\in L^{q_2}_{\loc}(B_1)$ by a bootstrap argument. 
Without loss of generality here we assume $k\ge 2$. 
Fix $\chi\in C_c^\infty(B_{1})$ with $\chi\equiv1$ on
$B_{7/8}$. As in~\eqref{defvphi}, for $\varphi\in C_c^\infty(B_{7/8})$, let
$v_\varphi\in W^{k+1,q_2'}(B_{1})\cap W^{1,q_2'}_0(B_{1})$ be the solution of
\[
    L^*v_\varphi=\varphi .
\]
Moreover, since $q_2'<\frac nk<\frac n{k-1}$, Lemma~\ref{lm:exisol} implies that 
\begin{equation}\label{vphiest2}
    \|v_\vp\|_{W^{k,q_2'}(B_{1})}\le C(\La,q_2, n,\eta)\|\vp\|_{W^{k-2,q_2'}_0(B_{7/8})}.
\end{equation}
Using $LS_1=0$, we have
\begin{align}\label{S1vpexp}
    \langle S_1,\varphi\rangle
    = \langle S_1,\chi\mko \varphi\rangle=
    \langle S_1,\chi L^*v_\varphi\rangle                              =
    \langle S_1,\chi L^*v_\varphi-L^*(\chi\mko v_\varphi)\rangle .
\end{align}
The commutator on the right-hand side of~\eqref{S1vpexp} is supported in $B_1$. By~\eqref{vphiest2} and~\eqref{eq:ffproinepo}, we estimate
\begin{align}\label{difchiest}
\begin{aligned}
    \|\chi L^*v_\varphi-L^*(\chi v_\varphi)\|_{W^{k-1,q_2'}_0(B_{1})}
    &=\big\|\p_i(a^{ji}\p_j\chi \mko v_\vp)+\p_i \chi \mko a^{ji}\p_j v_\vp\big\|_{W^{k-1,q_2'}(B_{1})}\\
    &\le
    C(q_2,n)\mko\|a^{ij}\|_{L^\nf\cap W^{k,\frac nk}(B_1)}\|v_\varphi\|_{W^{k,q_2'}(B_{1})}\\
    &\le
    C(\La,q_2,n,\eta)\mko \|\varphi\|_{W^{k-2,q_2'}_0(B_{7/8})}.
    \end{aligned}
\end{align}
Therefore, combining~\eqref{S1vpexp}--\eqref{difchiest} gives that for any $\vp\in C_c^\nf(B_{7/8})$,
\[
    |\lan S_1,\varphi\rangle|
    \le
    C(\La,q_2,n,\eta)\mko \|S_1\|_{W^{-k+1,q_2}(B_{1})}
      \|\varphi\|_{W^{k-2,q_2'}_0(B_{7/8})}.
\]
It follows that 
\[
    \|S_1\|_{W^{-k+2,q_2}(B_{7/8})}\le  C(\La,q_2,n,\eta)\mko \|S_1\|_{W^{-k+1,q_2}(B_{1})}.
\]
Repeating the above and using a covering and rescaling argument, we obtain that $S_1\in L^{q_2}_{\loc}(B_1)$ with
\[
    \|S_1\|_{L^{q_2}(B_{7/8})}
    \le
     C(\La,q_2,n,\eta)\mko\|S_1\|_{W^{-k+1,q_2}(B_1)}.
\]
If $k=0$, then $S_1\in W^{1,q_2}(B_1)$; otherwise, we have $q_2>\frac n{n-k}\ge \frac n{n-1}$. In both cases, the proof of~\cite[Thm.~4.1]{laMan20} implies that $S_1\in W^{1,t}_{\loc}(B_1)$ for every finite $t>1$. For each fixed $t$, after decreasing $\vae_0$ accordingly, we have the uniform estimate 
\begin{align}\label{W1testhom}
    \|S_1\|_{W^{1,t}(B_{3/4})} &\le C(\La,q_2,t,n,\eta)\mko \|S_1\|_{{W^{-k+1,q_2}(B_1)}}.
\end{align}
Fix $\al\in(0,\frac np-\ell)$ and choose $t$ such that
\[
    p<t<\frac n\ell,\qquad \frac np-\frac nt>\al.
\]
By~\eqref{W1testhom} and the inductive
argument for interior estimates in the proof of Lemma~\ref{lm:exisol},
we obtain
\[
    \|D S_1\|_{W^{\ell,t}(B_{5/8})}
    \le C(\La,q_2,t,n,\eta)\mko
       \|S_1\|_{W^{-k+1,q_2}(B_1)}.
\]
Then by Sobolev embeddings and H\"older's inequality, using the
scale-homogeneous norms~\eqref{eq:uniSobnm}, for all
$r\in(0,\frac12]$, we have
\begin{align}\label{eq:DS1finest}
\begin{aligned}
    \|D S_1\|_{W^{\ell,(p,s)}(B_r)}
    &\le C(p,t,n)\mkt r^{\frac np-\frac nt}
       \|D S_1\|_{W^{\ell,t}(B_{r})}\\
    &\le C(p,t,n,\al) r^\al \|D S_1\|_{W^{\ell,t}(B_{5/8})}\\
    &\le C(\La,\al,q_2,p,n,\eta)\mkt r^\al
       \|S_1\|_{W^{-k+1,q_2}(B_1)}.
\end{aligned}
\end{align}
Combining~\eqref{eq:R-pos-est},~\eqref{S1W-k+1est}, and~\eqref{eq:DS1finest} with~\eqref{equivSobnorm}
proves~\eqref{eq:ellcacciop1}.

To prove~\eqref{eq:ellcacciopg}, we fix $\ov q_0>n/(n-k)$ such that
\[
    \frac1{\ov q_0}>\frac1p-\frac{k}{n}.
\]
Such a choice is possible since $p>1$. Applying~\eqref{eq:ellcacciop1} on $B_{3/4}$ to
$S-\fint_{B_{3/4}}S$, and using Poincar\'e's inequality and
Sobolev embeddings, for all $r\in(0,\frac38]$, we obtain
\begin{align*}
    \|D S\|_{W^{\ell,(p,s)}(B_r)}
    &\le C(\La,\al,p,n,\eta)\bigg(
       \|T\|_{W^{\ell-1,(p,s)}(B_1)}
       +r^\al\Big\|S-\fint_{B_{3/4}}S
       \Big\|_{W^{-k+1,(\ov q_0,s)}(B_{3/4})}\bigg)\\
        &\le C(\La,\al,p,n,\eta)\bigg(
       \|T\|_{W^{\ell-1,(p,s)}(B_1)}
       +r^\al\Big\|S-\fint_{B_{3/4}}S
       \Big\|_{W^{1,(p,s)}(B_{3/4})}\bigg)\\
    &\le C(\La,\al,p,n,\eta)\Big(
       \|T\|_{W^{\ell-1,(p,s)}(B_1)}
       +r^\al\|D S\|_{L^{p,s}(B_{3/4})}\Big).
\end{align*}
For $r\in(\frac38,\frac12]$, the same estimate follows by the preceding local estimate combined with a covering and rescaling argument.
This proves~\eqref{eq:ellcacciopg}.

Finally, we prove part~\ref{negmorest}. Suppose
$T\in W^{-\ell-1,(p',s)}(B_1)$, then by
Lemma~\ref{lm:exisol} and interpolation, there exists
$\wti S_0\in W^{-\ell+1,(p',s)}(B_1)$ solving
\[
    L\wti S_0=T.
\]
Moreover, we have
\begin{equation}\label{eq:tS0-pos-est}
    \|\wti S_0\|_{W^{-\ell+1,(p',s)}(B_1)}
    \le C(\La,p,n,\eta)\mko
       \|T\|_{W^{-\ell-1,(p',s)}(B_1)}.
\end{equation}
Set $\wti S_1\coloneq S-\wti S_0$. Then $L\wti S_1=0$ in
$\mca D'(B_1)$. Since $p'>\frac n{n-\ell}$, we fix
$q_3\in(\frac n{n-k},\nf)$ such that
\[
    \frac1{q_3}>\frac1{p'}-\frac{k-\ell}{n}.
\]
We also fix $q_4\in(\frac n{n-k},\min(q_3,q_0))$. By the
Sobolev embeddings and~\eqref{eq:tS0-pos-est}, we have
\begin{equation}\label{eq:tS1W-k+1est}
    \|\wti S_1\|_{W^{-k+1,q_4}(B_1)}
    \le C(\La,p,q_0,n,\eta)\Big(
       \|S\|_{W^{-k+1,(q_0,s_0)}(B_1)}
       +\|T\|_{W^{-\ell-1,(p',s)}(B_1)}\Big).
\end{equation}
The preceding bootstrap argument for proving~\eqref{W1testhom}, applied with $q_2$ replaced by $q_4$, implies that
$\wti S_1\in L^\nf_{\loc}(B_1)$. Thus
$S\in W^{-\ell+1,(p',s)}_{\loc}(B_1)$.

Fix $\al\in(0,\frac n{p'}+\ell)$ and choose $t>p'$ such that
\[
    \frac n{p'}+\ell-\frac nt>\al.
\]
By the proof of the homogeneous-solution estimate~\eqref{W1testhom}, after possibly
decreasing $\vae_0$ depending also on $\al$, we have
\[
    \|D\wti S_1\|_{L^t(B_{3/4})}
    \le C(\La,p,q_0,n,\eta,\al)\mko
       \|\wti S_1\|_{W^{-k+1,q_4}(B_1)}.
\]
By the Sobolev embedding and scaling, for all
$r\in(0,\frac12]$, we obtain
\begin{align*}
    \|D\wti S_1\|_{W^{-\ell,(p',s)}(B_r)}
    &\le C(p,n,t)\mkt r^{\frac n{p'}+\ell-\frac nt}
       \|D\wti S_1\|_{L^t(B_{r})}\\
    &\le C(\La,p,q_0,n,\eta,\al)\mkt r^\al
       \|\wti S_1\|_{W^{-k+1,q_4}(B_1)}.
\end{align*}
In addition, by~\eqref{eq:tS0-pos-est}, we have
\begin{align*}
    \|D\wti S_0\|_{W^{-\ell,(p',s)}(B_r)}
    &\le C(p,n)\mko
       \|\wti S_0\|_{W^{-\ell+1,(p',s)}(B_1)}
    \le C(\La,p,n,\eta)\mko
       \|T\|_{W^{-\ell-1,(p',s)}(B_1)}.
\end{align*}
Combining these estimates with~\eqref{eq:tS1W-k+1est}
proves~\eqref{eq:ellcaccioneggrad}. This completes the proof.
\end{proof}
For use in the bootstrap argument in Section~\ref{sec:ndreg}, we also record a local elliptic estimate for coefficients in $W^{\si,q}$ with $q>n/\si$.

\begin{Lm}\label{lm:ellboot}
Let $n\ge2$, $\sigma\in\N^+$ with $\si\ge n/2$, and let $2<q<\nf$ with $q>n/\si$.
Suppose $a^{ij}\in W^{\sigma,q}(B^n)$ satisfies~\eqref{elli}. Let $1<p\le q$ satisfy
\begin{equation}\label{condellboot}
    \frac2q-\frac{\sigma-1}{n}<\frac1p<\frac{n+2}{2n}.
\end{equation}
Let $\ell\in\Z$ with $1-n/2\le\ell\le \si$, and let $f \in W^{\ell-1,p}(B^n)$. Suppose that $u\in W^{\ell,q}(B^n)$ satisfies 
\[
    Lu\coloneq\p_i(a^{ij}\p_j u)=f \qquad \text{in }\mca D'(B^n).
\]
Then $u\in W^{\ell+1,p}_{\loc}(B^n)$.
\end{Lm}

\begin{proof}
We first consider the case $\ell<0$. By the conditions $q>2$ and~\eqref{condellboot}, we obtain for $1-n/2\le \ell <0$ that
\[
    p> \frac {2n}{n+2}\ge \frac{n}{n+\ell},\qquad q>2\ge \frac{n}{n+\ell-1}. 
\]
The estimate then follows from Lemma~\ref{lm:ellcacciolp}~\ref{negmorest}. 

For $\ell\ge 0$, we now prove the estimate by induction on $\ell$. The case $\ell=0$ has been proved in Lemma~\ref{lm:ellcacciolp}. Let $\ell\in \{1,\dots,\si\}$, and we assume the estimate has been proved at order $\ell-1$.
Differentiating the equation $ Lu= f$, we obtain for any $\mu\in \{1,\dots,n\}$, 
\[
    L(\p_\mu u)=\p_\mu f-\p_i(\p_\mu a^{ij}\p_j u) \qquad \text{in }\mca D'(B^n).
\]
Since $\p_\mu a^{ij}\in W^{\si-1,q}(B^n)$ and $\p_j u\in W^{\ell-1,q}(B^n)$, applying~\eqref{eq:multinonexact} with $a=\si-\ell+2$ gives $\p_\mu a^{ij}\p_j u\in W^{\ell-1,p}(B^n)$. It follows that
\[
    L(\p_\mu u)=\p_\mu f-\p_i(\p_\mu a^{ij}\p_j u)\in W^{\ell-2,p}(B^n).
\]
Since $\p_\mu u\in W^{\ell-1,q}(B^n)$, by induction hypothesis, we obtain $\p_\mu u\in W^{\ell,p}_{\loc}(B^n)$ and hence $u\in W^{\ell+1,p}_{\loc}(B^n)$. This completes the proof.
\end{proof}

\subsection{Right-inverse estimates for the Hodge--Dirac operator}
\
\vskip5pt
\label{sec:hod-dec}

To solve \eqref{eq:d*gL=*gV}--\eqref{eq:dL0=0int} and the equations associated with the conservation laws in Section~\ref{sec:conlaws}, we prove some existence results for the Hodge--Dirac operator $d+\delta$, which extend the results of~\cite[Sec.~III]{BerLanMarRiv26} to higher positive and negative Sobolev orders. 

Let $n\ge3$, $k\in\N_0$ with $k\le n-3$. Let $U\subset \R^n$ be a bounded $C^{k+1}$ domain and let $g$ be a metric on $U$ satisfying
\[
    g=(g_{ij})_{i,j=1}^n\in L^\nf\cap W^{k+1,\frac n{k+1}}(U,\R^{n\times n}_{\sym}).\] 
Let $\vae\in (0,1)$ be a small enough constant to be determined later, and let $\La>0$ be a constant. Throughout this subsection, we assume
\begin{numcases}{}
    \|D g\|_{W^{k,\frac n{k+1}}(U)} \le \vae, \label{eq:dg_bound} \\[0.5ex]
    \Lambda^{-1}\mkt  |\xi|^{2}\le g_{ij}(x)\mkern2mu \xi^{i}\xi^{j}\le \Lambda\mkt  |\xi|^{2},
    \quad\text{for a.e. } x\in U\text{ and all }\mkt \xi\in\R^{n}. \label{eq:ellipticity}
\end{numcases}
Writing $\al=\sum_{|I|=\ell}\al_I\mko dx^I$, by~\cite[Eq.~(4.11)]{mitrea01}, there exist coefficient tensors $b_I^{iJ}$, $c_I^{iJ}$, and $d_I^J$, depending only on $g$ and its first derivatives, such that
\begin{equation}\label{eq:HodLapcomp}
\begin{aligned}
    -(\lap_g\al)_I
    &=\p_i\big(g^{ij}\p_j\al_I\big)
      +\p_i\big(b_I^{iJ}\al_J\big)
      +c_I^{iJ}\p_i\al_J
      +d_I^J\al_J.
\end{aligned}
\end{equation}
Moreover, we have
\begin{equation}\label{eq:Hodcoeffpt}
    |b_I^{iJ}|+|c_I^{iJ}|\le C(\La)|Dg|,
    \qquad |d_I^J|\le C(\La)|Dg|^2
    \qquad \text{a.e. in }U.
\end{equation}
Consequently, Lemma~\ref{lm:proinene} together with~\eqref{eq:dg_bound}--\eqref{eq:ellipticity} implies
\begin{equation}\label{eq:Hodcoeffhigh}
    \begin{aligned}
   \sum_{i,I,J}\Big(
     \|b_I^{iJ}\|_{W^{k,\frac n{k+1}}(U)}
     +\|c_I^{iJ}\|_{W^{k,\frac n{k+1}}(U)}
     \Big)
    +\sum_{I,J}\|d_I^J\|_{W^{k,\frac n{k+2}}(U)}
    \le C(\La,U,n)\, \vae.
\end{aligned}
\end{equation}

Building on the solvability result~\cite[Thm.~III.3]{BerLanMarRiv26}, we first establish the higher-order Sobolev estimates for the Dirichlet problem.

\begin{Lm}\label{lm:exisolHod}
Let $n\ge 3$, $k\in\N_0$ with $k\le n-3$, $0\le\ell\le n$, and
\[
        \frac n{n-1}<p<\frac n{k+1}.  
\]
Let $U\subset\R^n$ be a bounded $C^{k+1}$ domain, and $\La\ge 1$ be a constant.  Then there exists $\vae=\vae(\La,U,p)>0$ such that the following holds. Suppose $g$ satisfies~\eqref{eq:dg_bound}--\eqref{eq:ellipticity}. 
Then for any $\beta\in W^{k-1,p}\big(U,\bwe^\ell\R^n\big)$, there exists a unique
    $\al\in W^{k+1,p}\cap
    W^{1,p}_0\big(U,\bwe^\ell\R^n\big)$
solving
\begin{equation}\label{eq:HodDirpos}
    \lap_g\al=\beta\qquad \text{in }U,
\end{equation}
with the estimate
\begin{equation}\label{est:HodDirpos}
    \|\al\|_{W^{k+1,p}(U)}
    \le C(\La,U,p)\mko
    \|\beta\|_{W^{k-1,p}(U)}.
\end{equation}
Moreover, there exists a bounded linear operator
\[
    \msc G_{k}\colon
    W^{-k-1,p'}\big(U,\bwe^\ell\R^n\big)
    \longrightarrow
    W^{-k+1,p'}\big(U,\bwe^\ell\R^n\big)
\]
such that, for every $T\in W^{-k-1,p'}\big(U,\bwe^\ell\R^n\big)$,
\begin{equation}\label{eq:HodDirneg}
       \lap_g(\msc G_kT)=T\qquad\text{in }\mca D'\big(U,\bwe^\ell\R^n\big),
\end{equation}
with the estimate
\begin{equation}\label{est:HodDirneg}
    \|\msc G_kT\|_{W^{-k+1,p'}(U)}
    \le C(\La,U,p)\mko
    \|T\|_{W^{-k-1,p'}(U)}.
\end{equation}
\end{Lm}

\begin{proof}
When $k=0$, the assertion follows from~\cite[Thm.~III.3 and Rem.~III.9]{BerLanMarRiv26}. We prove~\eqref{est:HodDirpos} by induction on $k$. Let $1\le k\le n-3$ and assume that the result has been proved at order $k-1$. Suppose
\[
    \frac n{n-1}<p<\frac n{k+1},\qquad
    \beta\in W^{k-1,p}\big(U,\bwe^\ell\R^n\big).
\]
Let $\al\in W^{1,p}_0\big(U,\bwe^\ell\R^n\big)$ be the solution of~\eqref{eq:HodDirpos}, and we set $\ti p\coloneq np/(n-p)$. Since $n/(n-1)<\ti p<n/{k}$, the induction hypothesis together with Sobolev embeddings yields
\begin{equation}\label{eq:Hodindalpha}
    \|\al\|_{W^{k,\ti p}(U)}
    \le C(\La,U,p)\mko \|\beta\|_{W^{k-2,\ti p}(U)}\le C(\La,U,p)\mko
    \|\beta\|_{W^{k-1,p}(U)}.
\end{equation}
By~\eqref{eq:Hodcoeffhigh} and Lemma~\ref{lm:proinene}, we obtain
\begin{equation}\label{eq:Hodlowterms}
\begin{aligned}
 &\sum_{i,I,J}\Big(
 \|\p_i(b_I^{iJ}\al_J)\|_{W^{k-1,p}(U)}
 +\|c_I^{iJ}\p_i\al_J\|_{W^{k-1,p}(U)}
 +\|d_I^J\al_J\|_{W^{k-1,p}(U)}
 \Big)\\
 & \le C(\La,U,p)\mko
 \|\al\|_{W^{k,\ti p}(U)}.
\end{aligned}
\end{equation}
Applying Lemma~\ref{lm:exisol} at order $k$ to each component of~\eqref{eq:HodLapcomp}, and then using~\eqref{eq:Hodindalpha}--\eqref{eq:Hodlowterms}, we have
\[
\begin{aligned}
    \|\al\|_{W^{k+1,p}(U)}
    &\le C(\La,U,p)\Big(
       \|\beta\|_{W^{k-1,p}(U)}
       +\|\al\|_{W^{k,\ti p}(U)}\Big)\\
    &\le C(\La,U,p)\mko
       \|\beta\|_{W^{k-1,p}(U)}.
\end{aligned}
\]
This proves~\eqref{est:HodDirpos}.

It remains to prove~\eqref{eq:HodDirneg}--\eqref{est:HodDirneg}. The case $k=0$ follows from the first part applied with exponent $p'$. Assume $k\ge1$. Let $M_g$ be the matrix-valued coefficient determined by
\[
   \int_U\langle\al,\vp\rangle_g\,\dvol_g
   =\int_U (M_g\mko \al)\cdot\vp\,d\mca L^n.
\]
The dual operator of $\lap_g$ with respect to the standard Euclidean pairing is
\[
       \wti\lap_g\coloneq M_g\lap_gM_g^{-1}.
\]
For $\vp\in W^{k-1,p}_0\big(U,\bwe^\ell\R^n\big)$, let
$v_\vp\in W^{k+1,p}\cap W^{1,p}_0\big(U,\bwe^\ell\R^n\big)$ be the unique solution of
\[
       \wti \lap_g v_\vp=\vp.
\]
Then by Lemma~\ref{lm:proinene}, the estimate~\eqref{est:HodDirpos} also holds for $\wti\lap_g$:
\begin{equation}\label{eq:Hodadjointest}
       \|v_\vp\|_{W^{k+1,p}(U)}
       \le C(\La,U,p)\mko
       \|\vp\|_{W^{k-1,p}(U)}.
\end{equation}
We fix, independently of $p$, a bounded linear projection
$P\colon W^{k+1,p}\big(U,\bwe^\ell \R^n\big)\to W^{k+1,p}_0\big(U,\bwe^\ell \R^n\big)$ as in~\cite[Sec.~1.5.1]{grisvard11}. For $T\in W^{-k-1,p'}\big(U,\bwe^\ell\R^n\big)$, define
\[
       \langle \msc G_{k}T,\vp\rangle
       \coloneq\langle T,Pv_\vp\rangle.
\]
The estimate~\eqref{eq:Hodadjointest} gives~\eqref{est:HodDirneg}. Finally, the coefficient bounds~\eqref{eq:Hodcoeffhigh} and Lemma~\ref{lm:proinene} imply that
$\lap_g\msc G_kT\in W^{-k-1,p'}(U)$, and for any
$\psi\in C_c^\nf\big(U,\bwe^\ell\R^n\big)$, we have
\[
   \langle\lap_g\msc G_kT,\psi\rangle
   =\langle\msc G_k T,\wti \lap_g\psi\rangle
     =\langle T,\psi\rangle.
\]
The proof is complete.
\end{proof}
In the remainder of the paper, we shall frequently use the following direct consequence of Lemma~\ref{lm:proinene}. Under the assumptions~\eqref{eq:dg_bound}--\eqref{eq:ellipticity}, if $k\in \N$ with $k\le n-1$, and $n/(n-k)<p<\nf$, then for a current $T$ and a form $\alpha$ on $U$, we have
\begin{gather}\label{eq:Hodstarmult}
\begin{aligned}
    \|*_g T\|_{W^{-k,p}(U)}
    &\le C(\La,U,p)\mko
       \|T\|_{W^{-k,p}(U)},\\
        \|*_g \al\|_{W^{k,p'}(U)}
    &\le C(\La,U,p)\mko
       \|\al\|_{W^{k,p'}(U)}. 
       \end{aligned}
\end{gather}
As in~\eqref{eq:defd*glap}, we denote by $\delta$ the codifferential with respect to $g$. Then it follows from~\eqref{eq:Hodstarmult} that
\begin{equation}\label{est:delTsob}
\begin{aligned}
    \|\de T\|_{W^{-k,p}(U)}&\le C(\La, U,p) \|T\|_{W^{1-k,p}(U)},\\
    \|\de \al\|_{W^{k-1,p'}(U)}&\le C(\La, U,p) \|\al\|_{W^{k,p'}(U)}.
    \end{aligned}
\end{equation}
\begin{Prop}[cf.~{\cite[Cor.~III.6]{BerLanMarRiv26}}]\label{prop:Hod-coclosed}
Let $n\ge3$, $1\le k\le n-2$, and let
\[
      \frac n{n-k}<p<n,\qquad q\in[1,\nf].
\]
Suppose $g$ satisfies~\eqref{eq:dg_bound}--\eqref{eq:ellipticity} for $U=B^n$, with $\vae=\vae(\La,n,p)>0$ sufficiently small. Let
$\ga\in W^{-k-1,(p,q)}\big(B^n,\bwe^{\ell+1}\R^n\big)$ with $0\le\ell\le n$, and assume $d\ga=0$ in $\mca D'\big(B^n,\bwe\R^n\big)$. Then there exists
$\si\in W^{-k,(p,q)}\big(B^n,\bwe^\ell\R^n\big)$ such that\footnote{We write $d*_g\si=0$ instead of $\delta\mko\si=0$ to avoid applying $*_g$ to $W^{-k-1,(p,q)}$-forms, on which it is not a priori well-defined.}
\[
       d\si=\ga\qquad\text{and}\qquad d*_g\si=0
       \qquad\text{in }\;B^n,
\]
with the estimate
\begin{equation}\label{est:Hodcoclosed}
    \|\si\|_{W^{-k,(p,q)}(B^n)}
    \le C(\La,n,p)\mko
       \|\ga\|_{W^{-k-1,(p,q)}(B^n)}.
\end{equation}
\end{Prop}

\begin{proof}
By the weak Poincar\'e lemma~\cite[Prop.~4.1]{Costa10} and interpolation (see e.g.~\cite[Lem.~II.8]{BerLanMarRiv26}), there exists
$\beta\in W^{-k,(p,q)}\big(B^n,\bwe^\ell\R^n\big)$ satisfying $d\beta=\ga$, with
\begin{equation}\label{est:Hodprimitive1}
    \|\beta\|_{W^{-k,(p,q)}(B^n)}
    \le C(n,p)\mko
       \|\ga\|_{W^{-k-1,(p,q)}(B^n)}.
\end{equation}
Since $n/(n-1)<p'<n/k$, Lemma~\ref{lm:exisolHod} at order $k-1$ and interpolation give a solution
$\al\in W^{2-k,(p,q)}\big(B^n,\bwe^\ell\R^n\big)$ of $\lap_g\al=\beta$, with
\begin{equation}\label{est:Hodalpha1}
    \|\al\|_{W^{2-k,(p,q)}(B^n)}
    \le C(\La,n,p)\mko
       \|\beta\|_{W^{-k,(p,q)}(B^n)}.
\end{equation}
Set $\si\coloneq\delta d\al$. Applying~\eqref{est:delTsob} and using~\eqref{est:Hodprimitive1}--\eqref{est:Hodalpha1} gives
\[
    \|\si\|_{W^{-k,(p,q)}(B^n)}
    \le C(\La,n,p)\mko
       \|\al\|_{W^{2-k,(p,q)}(B^n)}
    \le C(\La,n,p)\mko
       \|\ga\|_{W^{-k-1,(p,q)}(B^n)}.
\]
Moreover, the same product inequalities give $\delta\mko\al\in W^{1-k,(p,q)}(B^n)$. Hence, since $\lap_g=d\delta+\delta d$, we have in $\mca D'(B^n)$,
\[
       d\si=d\big(\lap_g\al-d\mko\delta\al\big)
             =d\beta=\ga,
       \qquad
       d*_g\si=(-1)^{\ell+1}d^2*_g d\al=0.
\]
This completes the proof.
\end{proof}

Combining Proposition~\ref{prop:Hod-coclosed} with the inequalities~\eqref{eq:Hodstarmult} yields the following corollary.

\begin{Co}[cf.~{\cite[Cor.~III.7]{BerLanMarRiv26}}]\label{co-Hod-decw-1p}
Let $n$, $k$, $p$, $q$, and $g$ be as in
Proposition~\ref{prop:Hod-coclosed}, and suppose in addition $k\ge 2$.
\[
    \ga_1\in W^{-k,(p,q)}\big(B^n,\bwe^{\ell+1}\R^n\big),
    \qquad
    \ga_2\in W^{-k,(p,q)}\big(B^n,\bwe^{\ell-1}\R^n\big),
\]where $0\le\ell\le n$. Assume
$d\ga_1=0$ and $d*_g\ga_2=0$ in $\mca D'\big(B^n,\bwe\R^n\big)$. Then there exists
$\si\in W^{1-k,(p,q)}\big(B^n,\bwe^\ell\R^n\big)$ such that
\[
       d\si=\ga_1\qquad\text{and}\qquad \delta\mko\si=\ga_2,
       \qquad\text{in }\;B^n,
\]
with the estimate
\begin{equation}\label{est:Hodfirst}
    \|\si\|_{W^{1-k,(p,q)}(B^n)}
    \le C(\La,n,p)\Big(
       \|\ga_1\|_{W^{-k,(p,q)}(B^n)}
       +\|\ga_2\|_{W^{-k,(p,q)}(B^n)}\Big).
\end{equation}
\end{Co}
\begin{proof}
The inequality~\eqref{eq:Hodstarmult} implies that $*_g\,\ga_2\in W^{-k,(p,q)}$. Since $p>n/(n-k)>n/(n+1-k)$, applying Proposition~\ref{prop:Hod-coclosed} at order $k-1$ to
$\ga_1$ and $(-1)^\ell*_g\ga_2$ respectively, we obtain an $\ell$-form $\si_1$
and an $(n-\ell)$-form $\tau$, both in $W^{1-k,(p,q)}(B^n)$, such that
\[
    d\si_1=\ga_1,\quad d*_g\si_1=0,
    \qquad
    d\tau=(-1)^\ell*_g\ga_2,\quad d*_g\tau=0.
\]
Set $\si\coloneq\si_1+*_g^{-1}\tau$. Then $d\si=\ga_1$ and, by
\eqref{eq:defd*glap},
\[
    \delta\mko\si=(-1)^\ell*_g^{-1}d\tau=\ga_2.
\]
The estimate~\eqref{est:Hodfirst} follows from
\eqref{est:Hodcoclosed} and the inequalities
\eqref{eq:Hodstarmult} at orders $k-1$ and $k$.
\end{proof}

\section{Structural identities and estimates for the Noether system}\label{sec:strid}
In this section, we provide a higher-dimensional generalization of the structural identities proved in~\cite[Sec.~IV]{BerLanMarRiv26}. Let $n=2h\ge 6$ with $h\in \N^+$, $m>n$, and let
$\bP\in\mathcal I_{h-1,2}(B^n,\R^m)$. We write
$g=\bP^*g_{\std}$, and throughout this section, we fix $\La\ge 1$ such that, for a.e. $x\in B^n$
and every $v\in T_xB^n$,
\begin{align}\label{eq:immcon4d}
 \La^{-1}|v|_{\R^n}^2\le |d\bP_x(v)|_{\R^m}^2\le \La |v|_{\R^n}^2.
\end{align}
We follow the notation in Section~\ref{sec:cov}, and adopt the convention~\eqref{eq:convsonab} for Sobolev spaces of tensor fields. Using the contraction operators introduced in Section~\ref{sec:uselem}, we define the pointwise bilinear operators $\dres$, $\wres$,
$\ovs{\sbul}{\res}_g$, and $\ovs{\sbul}{\sbul}_g$ as in~\eqref{defbiopes}: the upper operators act on the
$\bwe\R^m$-factors, while the lower operators act on the differential-form
factors. Using the metric $g$ to raise indices, we write $\g^i\bP\coloneq g^{ij}\g_j\bP=g^{ij}\p_j\bP$, and we define
\begin{equation}\label{eq:defeta}
 \vet\coloneq\frac12\,d\bP\overset{\ovwe}{\we}d\bP,\qquad \vet_{ij}=\g_i\bP\we \g_j\bP.
\end{equation}
In Section~\ref{sec:pfmainThm}, we will restrict to small enough balls and perform a rescaling. Thus throughout this section, we assume
\begin{align}\label{StruEpassu}
 \varepsilon_{\bP}\coloneq
 \|D^2\bP\|_{W^{h-1,2}(B^n)}\le 1.
\end{align}
In particular, applying~\eqref{eq:ffproinepo} to $\g_{ij}\bP=\p_i\p_j\bP-\Ga_{ij}^k \p_k\bP$ yields
\begin{equation}\label{estg2Phi}
    \|\g^2 \bP\|_{W^{h-1,2}(B^n)}^2\coloneq  \sum_{i,j=1}^n  \|\bII_{ij} \|_{W^{h-1,2}(B^n)}^2\le C(n,\La)\mkt \varepsilon_{\bP}^2.
\end{equation}
For a domain $U\subset B^n$, and a finite-dimensional Euclidean space $E$ regarded as a vector bundle over $U$, we set
\begin{align}\label{defcritspacesIV}
\begin{aligned}
 \msc E_n(U,E)&\coloneq
 W^{2-h,\lf(\frac{2h}{h+1},2\rg)}(U,E),\\
 \msc N_{n,1}(U,E)&\coloneq
 W^{1-h,(2,\nf)}(U,E),\\
 \msc N_{n,2}(U,E)&\coloneq
 W^{2-h,(2,\nf)}(U,E). 
\end{aligned}
\end{align}
In this section, we set $U=B^n$ and thus $U$ in the notation of the above spaces or other distribution spaces will be omitted. For instance, we write 
\[
    \mca D'(\R^m)\coloneq\mca D'(B^n,\R^m).
\]
By~\eqref{eq:fTproinene}, for each space $\msc A\in\{\msc E_n,\msc N_{n,2},\msc N_{n,1}\}$, multiplication by an arbitrary $a\in L^\nf\cap W^{h,2}(B^n)$ is bounded on $\msc A$:
\begin{align}\label{prolorgen}
 \|aT\|_{\msc A}
 \le C(n)\|a\|_{L^\nf\cap W^{h,2}}
       \|T\|_{\msc A}.
\end{align}
 The structural identities below build on~\cite{Bernard25,BerLanMarRiv26}.
 We derive their higher-dimensional counterparts using covariant derivatives and identify the remainder terms explicitly.
This allows us to precisely track the dimension-dependent coefficients and establish the Sobolev--Lorentz estimates for the remainders in negative-order Sobolev spaces.

\subsection{Algebraic and differential identities}

\begin{Lm}\label{ber-prop:II.2}
For
$\bL\in\msc N_{n,1}\big(\R^m\ot\bwe^2\R^n\big)$, define
\begin{equation*}
\begin{dcases}
 A\coloneq\bL\mkt\dres d\bP,
 &B\coloneq2\bL\dwe d\bP,\\
 \bC\coloneq\bL\wres d\bP,
 &\bD\coloneq2\bL\ovs{\ovwe}{\we}d\bP.
\end{dcases}
\end{equation*}
Then the following holds in $\msc N_{n,1}\big(\bwe^2\R^m\ot \bwe^1 \R^n\big)$:
\begin{align}\label{eq:berIII.1.a}
 (5-2n)\mko\bC
 =\vet\mkt\ovs{\sbul}{\res}_g\bC
  +\bD\mkt\ovs{\sbul}{\res}_g\vet
  +\vet\mkt\resg A-B\mkt\resg\vet.
\end{align}
\end{Lm}

\begin{proof}
By~\eqref{prolorgen}, all the terms involved are well-defined in
$\msc N_{n,1}$. Hence by approximation, it suffices to prove the identity
pointwise for smooth $\bL$. In covariant components, with all indices raised
and lowered using $g$, we have
\begin{align}
    &A_i=\bL_{ji}\cdot\g^j\bP,
\qquad
    &&B_{ijk}=2\big(\bL_{ij}\cdot\g_k \bP-\bL_{ik}\cdot\g_j\bP         +\bL_{jk}\cdot\g_i\bP\big),\label{repABcov}\\
    &\bC_i=\bL_{ji}\we\g^j\bP,
    \qquad
    &&\bD_{ijk}=2\big(\bL_{ij}\we\g_k \bP-\bL_{ik}\we\g_j\bP
             +\bL_{jk}\we\g_i\bP\big).\label{repCDcov}              
\end{align}
By~\eqref{repCDcov}, the antisymmetry of $\vet^{jk}$, and an interchange of
$j$ and $k$ in the second summand, we have
\begin{align}
\begin{aligned}\label{Dcontetaexp}
 \big(\bD\mkt\ovs{\sbul}{\res}_g\vet\big)_i
 &=\frac12\bD_{ijk}\sbul\vet^{jk}\\
 &=2\mko
      (\bL_{ij}\we\g_k\bP)
    \sbul(\g^j\bP\we\g^k\bP)+
      (\bL_{jk}\we\g_i\bP)
       \sbul(\g^j\bP\we\g^k\bP).
\end{aligned}
\end{align}
For the first sum, since $\g_i\bP\we \g^i\bP=0$ and
$\g_i\bP\cdot \g^j\bP=\de_i^j$ (which implies
$\g_i\bP\cdot \g^i\bP=n$ under summation), formula~\eqref{bul2vecs} gives
\begin{align}
\begin{aligned}\label{firstDcontraction}
     &(\bL_{ij}\we\g_k\bP)\sbul
       (\g^j\bP\we\g^k\bP)\\
       &=n\mko \bL_{ij}\we\g^j\bP
     -(\bL_{ij}\cdot\g^k\bP)\mko\g_k\bP\we\g^j\bP
     -\de_k^j\mko (\bL_{ij} \we \g^k\bP)\\
    &=-(n-1)\bC_i
  +(\bL_{ij}\cdot\g_k\bP)\mko\vet^{jk}.
\end{aligned}
\end{align}
For the second sum in~\eqref{Dcontetaexp}, another application
of~\eqref{bul2vecs} yields
\begin{align*}
(\bL_{jk}\we\g_i\bP)\sbul
       (\g^j\bP\we\g^k\bP)
&= (\bL_{jk}\cdot\g^j\bP)\g_i\bP\we\g^k\bP
 +\de_i^k(\bL_{jk}\we\g^j\bP)\\
&\quad -(\bL_{jk}\cdot\g^k\bP)\g_i\bP\we\g^j\bP
 -\de_i^j(\bL_{jk}\we\g^k\bP).
\end{align*}
Interchanging $j$ and $k$ in the third summand and using
$\bL_{kj}=-\bL_{jk}$, it follows that
\begin{align}\label{secondDcontraction}
    (\bL_{jk}\we\g_i\bP)\sbul
       (\g^j\bP\we\g^k\bP)=2\bC_i+2A^j\mko\vet_{ij}.
\end{align}
Combining~\eqref{Dcontetaexp}--\eqref{secondDcontraction}, we obtain
\begin{align}
 (\bD\mkt\ovs{\sbul}{\res}_g\vet)_i
 &=(4-2n)\bC_i+2A^j\mko\vet_{ij}
   +2(\bL_{ij}\cdot\g_k\bP)\vet^{jk}.
 \label{Detaindex}
\end{align}
We next compute $\vet\mkt\ovs{\sbul}{\res}_g\bC$. By~\eqref{bul2vecs}, we have
\begin{align*}
 &(\vet\mkt\ovs{\sbul}{\res}_g\bC)_i\\
 &=(\g_j\bP\we\g_i\bP)
    \sbul(\bL_k{}^j\we\g^k\bP)\\
 &=(\g_j\bP\cdot\bL_k{}^j)\g_i\bP\we\g^k\bP
   +\de_i^k\mko\g_j\bP\we\bL_k{}^j
 -\de_j^k\mko\g_i\bP\we\bL_k{}^j
   -(\g_i\bP\cdot\bL_k{}^j)\g_j\bP\we\g^k\bP.
\end{align*}
Here
$\g_j\bP\cdot\bL_k{}^j=-A_k$,
$\bL_j{}^j=0$, and
$\g_j\bP\we\bL_i{}^j=\bC_i$. It follows that
\begin{align}
 (\vet\mkt\ovs{\sbul}{\res}_g\bC)_i
 &=-A^j\mko\vet_{ij}+\bC_i
   +(\bL_{jk}\cdot\g_i\bP)\vet^{jk}.
 \label{etaCindex}
\end{align}
Finally, by~\eqref{repABcov} and an interchange of $j$ and $k$ in the
second summand,
\begin{align}
 (B\mkt\resg\vet)_i
 &=\frac12B_{ijk}\vet^{jk}=2(\bL_{ij}\cdot\g_k\bP)\vet^{jk}
   +(\bL_{jk}\cdot\g_i\bP)\vet^{jk}.
 \label{Betaindex}
\end{align}
Since $(\vet\mkt\resg A)_i=\vet_{ji}A^j=-A^j\vet_{ij}$,
combining~\eqref{Detaindex}--\eqref{Betaindex} yields
\begin{align*}
\big(\vet\mkt\ovs{\sbul}{\res}_g\bC
  +\bD\mkt\ovs{\sbul}{\res}_g\vet
  +\vet\mkt\resg A-B\mkt\resg\vet\big)_i=(5-2n)\bC_i.
\end{align*}
This proves~\eqref{eq:berIII.1.a}.
\end{proof}

For later use, we record the expressions of the exterior derivative and codifferential using covariant derivatives; see for instance~\cite[Sec.~9.4.1]{Petersen16}. Since $n$ is even, the codifferential defined in~\eqref{eq:defd*glap} is given by $\de=-*_g d\,*_g$.
For any smooth $\ell$-form $\omega$ and the Levi-Civita connection $\g$ of any smooth metric, we have
\begin{equation}\label{defd*g4d}
 (d\om)_{j_0\dots j_{\ell}}=\sum_{k=0}^{\ell} (-1)^k\g_{j_k} \om_{j_0\dots \wih{j_k}\dots j_\ell},\qquad (\delta\omega)_{i_2\dots i_\ell}
 =-\g^j\omega_{j i_2\dots i_\ell}.
\end{equation}
In addition, each component of $\g\om$ can be written as a finite sum of
expressions $a_1\mko\p_i(a_2\mkt\om_J)$ where $a_1,a_2\in
L^\nf\cap W^{h,2}$ with norms bounded by $C(n,\La)$. Since taking one derivative maps $\msc N_{n,2}$ continuously into $\msc N_{n,1}$, under the assumptions~\eqref{eq:immcon4d} and~\eqref{StruEpassu} with $g=g_{\bP}$, by~\eqref{prolorgen}, we have
\begin{equation}\label{estdwdew}
    \|d\om\|_{\msc N_{n,1}}+\|\de \om\|_{\msc N_{n,1}}+  \|\g\om\|_{\msc N_{n,1}}\le C(\La,n)\|\om\|_{\msc N_{n,2}}.
\end{equation}
Hence by approximation in both the form and metric, the identities in~\eqref{defd*g4d} remain valid for $\om\in \msc N_{n,2}$ and $g=g_{\bP}$.
\begin{Lm}\label{prop:III2ber}
Let
$\al\in L^\nf\cap W^{h,2}\big(\bwe^2\R^n\big)$ and
$\beta\in\msc N_{n,2}\big(\bwe^2\R^n\big)$. Then the following holds in $\mca D'(B^n,\R^n)$:
\begin{align}\label{eq:III2-scalar}
 \al\resg\delta\beta+d\beta\resg\al
 =\delta(\al\bulg\beta)+d(\al\resg\beta)
   +\mca R_1[\al,\beta;g],
\end{align}
where the remainder satisfies
\begin{align}\label{eq:III2-bd}
 \|\mca R_1[\al,\beta;g]\|_{\msc E_n}
 \le C(n,\La)\|\g\al\|_{W^{h-1,2}}
                  \|\beta\|_{\msc N_{n,2}}.
\end{align}
Similarly, if
$\vec\al\in L^\nf\cap W^{h,2}
\big(\bwe^2\R^m\ot\bwe^2\R^n\big)$, $\vec\beta\in\msc N_{n,2}\big(\bwe^2\R^m\ot\bwe^2\R^n\big)$, and
$\beta\in\msc N_{n,2}\big(\bwe^2\R^n\big)$, then we have
\begin{align}
 \vec\al\resg\delta\beta+d\beta\resg\vec\al
 &=\delta(\vec\al\bulg\beta)+d(\vec\al\resg\beta)
   +\vec{\mca R}_1[\vec\al,\beta;g],\label{eq:III2-vector1}\\
 \vec\al\ovs{\sbul}{\res}_g\delta\vec\beta
 -d\vec\beta\ovs{\sbul}{\res}_g\vec\al
 &=\delta(\vec\al\ovs{\sbul}{\sbul}_g\vec\beta)
   +d(\vec\al\ovs{\sbul}{\res}_g\vec\beta)
   +\vec{\mca R}_1[\vec\al,\vec\beta;g],\label{eq:III2-vector2}
\end{align}
with
\begin{align}\label{eq:III2-bdvec}
\begin{aligned}
 \|\vec{\mca R}_1[\vec\al,\beta;g]\|_{\msc E_n}
 &\le C(n,\La)\|\g\vec\al\|_{W^{h-1,2}}
                    \|\beta\|_{\msc N_{n,2}},\\
 \|\vec{\mca R}_1[\vec\al,\vec\beta;g]\|_{\msc E_n}
 &\le C(n,\La)\|\g\vec\al\|_{W^{h-1,2}}
                    \|\vec\beta\|_{\msc N_{n,2}}.
\end{aligned}
\end{align}
\end{Lm}

\begin{proof}
By weak-$*$ approximation of $\beta$ in $\msc N_{n,2}$, together with the inequalities~\eqref{prolorgen} and~\eqref{estdwdew}, it suffices to consider smooth $\beta$. Since $\g g=0$,
raising and lowering indices commute with $\g$. We first prove
\eqref{eq:III2-scalar}. From Definition~\ref{defresbul} and~\eqref{defd*g4d}, we have
\begin{equation}
\begin{dcases}
\begin{aligned}
     &(\delta\beta)_i=-\g^j\beta_{ji},\qquad
    &&(d\beta)_{ijk}=\g_i\beta_{jk}+\g_j\beta_{ki}+\g_k\beta_{ij},\\
     &\al\resg\beta=\frac12\al_{jk}\mko\beta^{jk}, \qquad &&(\al\bulg\beta)_{ij}
     =\al_i{}^k\mko\beta_{jk}-\al_j{}^k\beta_{ik}.\label{expalbeta}
     \end{aligned}
 \end{dcases}
\end{equation}
Using the antisymmetry of $\al$ and $\beta$, and interchanging $j$ and $k$
in the last summand, we obtain
\begin{align}
\begin{aligned}\label{lhsIII2scalar}
 \big(\al\resg\delta\beta+d\beta\resg\al\big)_i
 &=\al_i{}^k\g^j\beta_{jk}
   +\frac12\al^{jk}\big(
       \g_i\beta_{jk}+\g_j\beta_{ki}+\g_k\beta_{ij}\big)\\
 &=\al_i{}^k\g^j\beta_{jk}
   +\al^{jk}\g_j\beta_{ki}
   +\frac12\al^{jk}\g_i\beta_{jk}.
\end{aligned}
\end{align}
On the other hand, the identities~\eqref{expalbeta} imply that
\begin{align*}
\begin{dcases}
     \big(\delta(\al\bulg\beta)\big)_i
     =-(\g^j\al_j{}^k)\beta_{ik}
   -\al_j{}^k\g^j\beta_{ik}
   +(\g^j\al_i{}^k)\beta_{jk}
   +\al_i{}^k\g^j\beta_{jk},\\
 \big(d(\al\resg\beta)\big)_i
     =\frac12(\g_i\al_{jk})\beta^{jk}
   +\frac12\al^{jk}\g_i\beta_{jk}.
   \end{dcases}
\end{align*}
Since $-\al_j{}^k\g^j\beta_{ik}=\al^{jk}\g_j\beta_{ki}$, adding these
two identities and comparing with~\eqref{lhsIII2scalar} proves
\eqref{eq:III2-scalar}, where
\begin{align}\label{covR1ab}
 \mca R_1[\al,\beta;g]_i
 =(\g^j\al_j{}^k)\beta_{ik}-(\g^j\al_i{}^k)\beta_{jk}
   -\frac12(\g_i\al_{jk})\beta^{jk}.
\end{align}
Now the Sobolev--Lorentz embedding implies that $W^{h-1,2}\hookrightarrow W^{h-2,(\frac{2h}{h-1},2)}$. Then by~\eqref{eq:fTproinene}, we obtain the following inequality for any $f\in W^{h-1,2}$ and $T\in \msc N_{n,2}$:
\begin{equation}\label{ineproENn}
    \|fT\|_{\msc E_n}\le C(n) \|f\|_{W^{h-1,2}}\|T\|_{\msc N_{n,2}}.
\end{equation}
The estimate~\eqref{eq:III2-bd} then follows by combining~\eqref{prolorgen} with~\eqref{ineproENn}.

Replacing $\al$ by $\vec\al$ in the preceding component computation gives
\eqref{eq:III2-vector1}, with
\begin{align}\label{covR1abvec}
 \vec{\mca R}_1[\vec\al,\beta;g]_i
 =(\g^j\vec\al_j{}^k)\beta_{ik}-(\g^j\vec\al_i{}^k)\beta_{jk}
   -\frac12(\g_i\vec\al_{jk})\beta^{jk}.
\end{align}
For~\eqref{eq:III2-vector2}, we use that
$\vec u\sbul\vec v=-\vec v\sbul\vec u$ for
$\vec u,\vec v\in\bwe^2\R^m$. Thus
\begin{align*}
\big(\vec\al\mkt\ovs{\sbul}{\res}_g\delta\vec\beta
      -d\vec\beta\mkt\ovs{\sbul}{\res}_g\vec\al\big)_i=\vec\al_i{}^k\sbul\g^j\vec\beta_{jk}
  +\frac12\mko\vec\al^{jk}\sbul(d\vec \beta)_{ijk}.
\end{align*}
The remaining argument is the same as that for scalar-valued differential forms. This completes the proof.
\end{proof}
\subsection{Derivation of the structural equations}
\
\vskip5pt
Since $\bP\in \mathcal{I}_{h-1,2}(B^n, \R^m)$, inductively applying Lemma~\ref{lm:proinene} (or using~\eqref{difcov}--\eqref{congpdif} together with~\eqref{eq:fTproinene} and Corollary~\ref{cor:multiproinene}), for any $2\le k \le n$ and $1\le i_1,\dots,i_k\le n$, we obtain
\begin{equation}\label{nabbPspa}
    \|\g_{i_1\dots i_k}\bP\|_{W^{h+1-k,2}}\le  C(n,\La)\mko \vae_{\bP}.
\end{equation}
As in~\eqref{not-met}, let $\bn$ be the Gauss map and
$\pi_N$ the orthogonal projection onto the normal bundle. Since $n$ is
even, for every $\vec v\in\R^m$,
\begin{equation}\label{pinvres}
 \pi_N\mko\vec v=(-1)^{m-1}\bn\res(\bn\res\vec v).
\end{equation}
Under~\eqref{eq:immcon4d} and~\eqref{StruEpassu}, the coefficients of
$\pi_N$ and $\pi_T\coloneq\mathrm{id}-\pi_N$ are multipliers on
$\msc E_n$, $\msc N_{n,2}$, and $\msc N_{n,1}$ by~\eqref{prolorgen}. In the remainder of this section, we set 
\begin{equation}\label{defvecK}
    \vec K\coloneq \lap_g^{h}\bP.
\end{equation}
By~\eqref{eq:fTproinene} and~\eqref{nabbPspa}, we have
\begin{align}
 \vec K&\in W^{1-h,2}(B^n,\R^m)\hookrightarrow \msc N_{n,1}(B^n,\R^m).\label{gpHdotpp}
\end{align}
Moreover, by Lemma~\ref{lm:gpdot2ff}, $\g\bP\cdot \vec K$ is a linear combination of partial contractions of the tensors 
\[
    \big\{\g^{(\ell)}\bP\cdot\g^{(n+1-\ell)}\bP\big\}_{2\le \ell\le h}.
\]
Combining~\eqref{eq:immcon4d},~\eqref{StruEpassu} and~\eqref{nabbPspa} with Lemma~\ref{lm:proinene}, we then obtain 
\begin{equation}\label{estpiTK}
    \|\pi_T\vec K\|_{\msc E_n}\le C(n,\La)  \mko \varepsilon_{\bP}^2.
\end{equation}
We now apply Lemmas~\ref{ber-prop:II.2} and~\ref{prop:III2ber} to the Noether system that will be constructed in Section~\ref{sec:pfmainThm}.

\begin{Prop}\label{prop:LSRsys}
Let
\begin{align*}\begin{dcases}\begin{aligned}
    &\bL\in\msc N_{n,1}\big(\R^m\ot\bwe^2\R^n\big), &&\\
    &S\in\msc N_{n,2}\big(\bwe^2\R^n\big),
     &&\bR\in\msc N_{n,2}\big(\bwe^2\R^m\ot\bwe^2\R^n\big),\\&
     \vt_{\dil}\in\msc E_n\big(\bwe^1\R^n\big),
    &&\bvt_{\rot}\in\msc E_n\big(\bwe^2\R^m\ot\bwe^1\R^n\big).
    \end{aligned}
\end{dcases}
\end{align*}
Let $\vec K$ be defined as in~\eqref{defvecK}, and suppose
\begin{align}\label{sys:d*dSR}
\begin{dcases}
 \delta S=\bL\mkt\dres d\bP+\vt_{\dil},\quad
 &dS=-2\bL\dwe d\bP,\\
 \delta\bR=\bL\wres d\bP+d\bP\we\vec K+\bvt_{\rot},\quad
 &d\bR=-2\bL\ovs{\ovwe}{\we}d\bP.
\end{dcases}
\end{align}
Then the following holds in $\mca D'\big(B^n,\bwe^2\R^m\ot \bwe^1\R^n\big)$:
\begin{align}\label{eq:RSKsys}\begin{aligned}
    &\delta\big(\mko(2n-5)\bR+\vet\mkt\ovs{\sbul}{\sbul}_g\bR+\vet\mko\bulg S\big)
     +d\big(\vet\mkt\ovs{\sbul}{\res}_g\bR+\vet\mkt\resg S\big)\\
    &=(3n-6)\mko d\bP\we\vec K+\vec{\mca R}_0,
\end{aligned}
\end{align}
where $\vec{\mca R}_0\in\msc E_n\big(\bwe^2\R^m\ot\bwe^1\R^n\big)$ and satisfies
\begin{align}\label{R3ptbd}
 \|\vec{\mca R}_0\|_{\msc E_n}
 \le C(n,\La)\Big(&\|\vt_{\dil}\|_{\msc E_n}
     +\|\bvt_{\rot}\|_{\msc E_n}
     +\varepsilon_{\bP}
       \big(\|S\|_{\msc N_{n,2}}+\|\bR\|_{\msc N_{n,2}}+\vae_{\bP}\big)\Big).
\end{align}
\end{Prop}

\begin{proof}
Set
\begin{equation*}
 A=\delta S-\vt_{\dil},\qquad B=-dS,
 \qquad
 \bC=\delta\bR-d\bP\we\vec K-\bvt_{\rot},\qquad
 \bD=-d\bR.
\end{equation*}
Lemma~\ref{ber-prop:II.2} then gives
\begin{align*}
&(5-2n)\big(\delta\bR-d\bP\we\vec K-\bvt_{\rot}\big)\\
&=\vet\mkt\ovs{\sbul}{\res}_g\bC
  +\bD\mkt\ovs{\sbul}{\res}_g\vet
  +\vet\mkt\resg A-B\mkt\resg\vet\\
&=\vet\mkt\ovs{\sbul}{\res}_g\delta\bR
  -\vet\mkt\ovs{\sbul}{\res}_g(d\bP\we\vec K)
  -\vet\mkt\ovs{\sbul}{\res}_g\bvt_{\rot}
  -d\bR\mkt\ovs{\sbul}{\res}_g\vet+\vet\mkt\resg\delta S-\vet\mkt\resg\vt_{\dil}
  +dS\mkt\resg\vet.
\end{align*}
By~\eqref{eq:III2-vector1}--\eqref{eq:III2-vector2}, we have
\begin{align*}
 \vet\mkt\ovs{\sbul}{\res}_g\delta\bR
 -d\bR\mkt\ovs{\sbul}{\res}_g\vet
 &=\delta(\vet\mkt\ovs{\sbul}{\sbul}_g\bR)
   +d(\vet\mkt\ovs{\sbul}{\res}_g\bR)
   +\vec{\mca R}_1[\vet,\bR;g],\\
 \vet\mkt\resg\delta S+dS\mkt\resg\vet
 &=\delta(\vet\mko\bulg S)+d(\vet\mkt\resg S)
   +\vec{\mca R}_1[\vet,S;g].
\end{align*}
Substitution and rearrangement yield
\begin{align}\label{estRSsys1}
\begin{aligned}
&\delta\big((2n-5)\bR+\vet\mkt\ovs{\sbul}{\sbul}_g\bR
                         +\vet\mko\bulg S\big)
 +d\big(\vet\mkt\ovs{\sbul}{\res}_g\bR+\vet\mkt\resg S\big)\\
&=(2n-5) \mko d\bP\we\vec K
  +\vet\mkt\ovs{\sbul}{\res}_g(d\bP\we\vec K)
  +(2n-5)\mko\bvt_{\rot}
  +\vet\mkt\ovs{\sbul}{\res}_g\bvt_{\rot}
  \\&\qquad+\vet\mkt\resg\vt_{\dil}-\vec{\mca R}_1[\vet,\bR;g]
  -\vec{\mca R}_1[\vet,S;g].
\end{aligned}
\end{align}
It remains to compute the second term on the right. By~\eqref{bul2vecs}, we have
\begin{align}
\begin{aligned}\label{etabulresnwe}
&\big(\vet\mkt\ovs{\sbul}{\res}_g(d\bP\we\vec K)\mko\big)_i
 =(\g_j\bP\we\g_i\bP)
       \sbul(\g^j\bP\we\vec K)\\
&=n\mko \g_i\bP\we\vec K
 +0-(\g_j\bP\cdot\vec K)\g_i\bP\we\g^j\bP
 -\de_i^j(\g_j\bP\we\vec K)\\
&=(n-1)\g_i\bP\we\vec K-\g_i\bP\we\pi_T\vec K.
\end{aligned}
\end{align}
Here we used $\g_j\bP\cdot \g^j\bP=n$, $\g_j\bP\we\g^j\bP=0$, and
$\pi_T\vec K=(\vec K\cdot\g_j\bP)\g^j\bP$. Since
$2n-5+n-1=3n-6$, combining~\eqref{estRSsys1}--\eqref{etabulresnwe}
gives~\eqref{eq:RSKsys}, where
\begin{align}\label{defvecR3}
\begin{aligned}
    \vec{\mca R}_0=&-d\bP\we\pi_T\vec K
    +(2n-5)\bvt_{\rot}
    +\vet\mkt\ovs{\sbul}{\res}_g\bvt_{\rot}
    +\vet\mkt\resg\vt_{\dil}\\
    &-\vec{\mca R}_1[\vet,\bR;g]
  -\vec{\mca R}_1[\vet,S;g].
\end{aligned}
\end{align}
Finally, the definition~\eqref{eq:defeta} of $\vet$ gives
\[
 \g_\ell\mko \vet_{ij}
 =\bII_{\ell i}\we\g_j\bP
  +\g_i\bP\we\bII_{\ell j}.
\]
Thus by~\eqref{estg2Phi} and Lemma~\ref{lm:proinene}, we have
\begin{align}\label{nablaetaest}
 \|\g\vet\|_{W^{h-1,2}}
 \le C(n,\La)\mko\varepsilon_{\bP}.
\end{align}
The estimate~\eqref{R3ptbd} follows by combining~\eqref{prolorgen},~\eqref{eq:III2-bdvec},~\eqref{estpiTK}, and~\eqref{nablaetaest}.
\end{proof}

To handle the term $d\bP\we\vec K$ in~\eqref{eq:RSKsys}, we use the
following identity.

\begin{Lm}\label{lm:d*XwedPhi}
Let $\bX\in\msc N_{n,2}\big(\R^m\ot\R^n\big)$. Then the following holds in $\mca D'\big(\bwe^2\R^m\ot\bwe^1\R^n\big)$:
\begin{align}\label{d*XwedPhi}
    \delta\big(\bX\ovs{\ovwe}{\we}d\bP\big)
    =(\delta\bX)\we d\bP+d\big(\bX\wres d\bP\big)
   +d\bX\wres d\bP+\vec{\mca R}_1[\bX;g],
\end{align}
where the remainder $\vec{\mca R}_1[\bX;g]\in\msc E_n$ satisfies
\begin{align}\label{vecR4ptbd}
 \|\vec{\mca R}_1[\bX;g]\|_{\msc E_n}
 \le C(n,\La)\mko\varepsilon_{\bP}\mko\|\bX\|_{\msc N_{n,2}}.
\end{align}
\end{Lm}

\begin{proof}
By~\eqref{prolorgen} and~\eqref{estdwdew}, both sides of~\eqref{d*XwedPhi} belong to $\msc N_{n,1}$, hence by weak-$*$ approximation of $\bX$ in $\msc N_{n,2}$, it suffices to consider smooth $\bX$. Set
$\bY=\bX\ovs{\ovwe}{\we}d\bP$, then we have
\[
 \bY_{ij}=\bX_i\we\g_j\bP-\bX_j\we\g_i\bP.
\]
Using~\eqref{defd*g4d}, $\g_i\g_j\bP=\bII_{ij}$, and
$\g^j\g_j\bP=n\mko\bH$, we compute
\begin{align}
\begin{aligned}\label{deltaXwedcoord}
 (\delta\bY)_i
 &=-\g^j\big(\bX_j\we\g_i\bP-\bX_i\we\g_j\bP\big)\\
 &=(\delta\bX)\we\g_i\bP -\bX^j\we\bII_{ij}
   +(\g^j\bX_i)\we\g_j\bP
  +n\mko\bX_i\we\bH.
\end{aligned}
\end{align}
On the other hand, writing $\vec q=\bX\wres d\bP=\bX_j\we\g^j\bP$, we obtain
\begin{align*}
 (d\vec q)_i
 &=(\g_i\bX_j)\we\g^j\bP+\bX^j\we\bII_{ij},\\
 \big(d\bX\wres d\bP\big)_i
 &=(\g^j\bX_i)\we\g_j\bP-(\g_i\bX_j)\we\g^j\bP.
\end{align*}
Summing these two equations yields
\begin{align}\label{dXrescoord}
 \Big(d\big(\bX\wres d\bP\big)+d\bX\wres d\bP\Big)_i
 =(\g^j\bX_i)\we\g_j\bP+\bX^j\we\bII_{ij}.
\end{align}
Subtracting~\eqref{dXrescoord} from~\eqref{deltaXwedcoord} implies~\eqref{d*XwedPhi}, with
\begin{equation}\label{R1Xexplicit}
 \big(\vec{\mca R}_1[\bX;g]\big)_i
 =n\bX_i\we\bH-2\bX^j\we\bII_{ij}.
\end{equation}
The estimate~\eqref{vecR4ptbd} follows by combining~\eqref{eq:immcon4d} and~\eqref{estg2Phi} with~\eqref{ineproENn}.
\end{proof}

Before applying Proposition~\ref{prop:LSRsys}, we record the following normal
projection identity.

\begin{Lm}\label{lm:nd*RS}
Let
$S\in\msc N_{n,2}\big(\bwe^2\R^n\big)$ and
$\bR\in\msc N_{n,2}\big(\bwe^2\R^m\ot\bwe^2\R^n\big)$. Then the following holds in $\mca D'(\R^m)$:
\begin{align}\label{eq:norproRS}
 \pi_N\mko\delta\Big(
   \big(\vet\mkt\ovs{\sbul}{\res}_g\bR+\vet\mkt\resg S\big)\mko\res d\bP\Big)
 =-\pi_N\Big((\delta\bR)\mkt\ovs{\res}{\res}_g d\bP\Big)
   +\vec{\mca R}_2[\bR;g],
\end{align}
where
$\vec{\mca R}_2[\bR;g]\in\msc E_n(\R^m)$ and satisfies
\begin{align}\label{rmdbdnRS}
 \|\vec{\mca R}_2[\bR;g]\|_{\msc E_n}
 \le C(n,\La)\mko\vae_{\bP}\mko \|\bR\|_{\msc N_{n,2}}.
\end{align}
\end{Lm}

\begin{proof}
By~\eqref{prolorgen}, both sides of~\eqref{eq:norproRS} belong to $\msc N_{n,1}$. Hence by weak-$*$ approximation of $S$ and $\bR$ in $\msc N_{n,2}$, it suffices to compute for smooth $S$ and $\bR$. Set
\[
 \vec Q_S\coloneq \vet\mkt\resg S= \frac12\mko S^{ij}\mko \vet_{ij},
 \qquad
 \vec Q_R\coloneq \vet\mkt\ovs{\sbul}{\res}_g\bR=\frac12\mko\vet_{ij}\sbul\bR^{ij}.
\]
For the contribution of $S$, the identity
$(\vec a\we\vec b)\mkt\res\vec c
=(\vec a\cdot\vec c)\vec b-(\vec b\cdot\vec c)\vec a$ gives
\begin{align}
\begin{aligned}\label{QSresPhi}
 (\vec Q_S\mkt\res d\bP)_k
 &=\frac12S^{ij}\big(\mko
     (\g_i\bP\we\g_j\bP)\mkt \res\g_k\bP\big)\\
 &=\frac12S^{ij}\big(g_{ik}\g_j\bP-g_{jk}\g_i\bP\big)
  =S_k{}^j\g_j\bP.
\end{aligned}
\end{align}
It follows that
\begin{align*}
 \pi_N\mko\delta(\vec Q_S\mkt\res d\bP)
 &=-\pi_N\g^k(S_k{}^j\g_j\bP)\\
 &=-\pi_N\big(\mko(\g^kS_k{}^j)\g_j\bP
                    +S^{kj}\mko\bII_{kj}\big)=0.
\end{align*}
The last equality holds since the first term is tangent-valued and the second vanishes by the
symmetry of $\bII$ and antisymmetry of $S$.

We next compute the contribution of $\bR$. Differentiating $\vec Q_R\mkt\res d\bP$, we have
\begin{align}
\begin{aligned}\label{diffQRresdP}
    &\pi_N\delta(\vec Q_R\mkt\res d\bP)\\
    &=-\frac12\pi_N\Big(
       \big((\g^k\vet_{ij})\sbul\bR^{ij}\big)\mkt \res\g_k\bP
      +(\vet_{ij}\sbul\g^k\bR^{ij})\mkt \res\g_k\bP +(\vet_{ij}\sbul\bR^{ij})\mkt \res\g^k\g_k\bP\Big).\\
\end{aligned}
\end{align}
      Setting $\vec Z=\vec u\we\vec v$ with $\vec u,\vec v\in \R^m$, and applying~\eqref{bul2vecs}, we obtain
\begin{align}
\begin{aligned}\label{bulletnormalalg}
    &\pi_N\big( \mko (\vet_{ij}\sbul\vec Z)\mkt \res\g_k\bP\big)\\
    &= g_{jk}(\g_i\bP\cdot \vec u) \pi_N \vec v-  g_{ik}(\g_j\bP\cdot \vec u) \pi_N \vec v- g_{jk}(\g_i\bP\cdot \vec v) \pi_N \vec u+  g_{ik}(\g_j\bP\cdot \vec v) \pi_N \vec u\\
    &=g_{jk}\mko\pi_N(\vec Z\res\g_i\bP)
  -g_{ik}\mko\pi_N(\vec Z\res\g_j\bP). 
  \end{aligned}
\end{align}
By pointwise linearity in $\vec Z$, we may apply~\eqref{bulletnormalalg} with $\vec Z=\g^k\bR^{ij}$. It follows that
\begin{align}
\begin{aligned}\label{normalderRalg}
    \frac12\pi_N\Big(
     (\vet_{ij}\sbul\g^k\bR^{ij})\mkt \res\g_k\bP\Big)&=\frac12\pi_N\Big(
  (\g_j\bR^{ij})\mkt \res\g_i\bP
  -(\g_i\bR^{ij})\mkt \res\g_j\bP\Big)\\
    &=-\pi_N\Big((\g_i\bR^{ij})\mkt \res\g_j\bP\Big)
     =\pi_N\Big((\delta\bR)\mkt\ovs{\res}{\res}_g d\bP\Big),
\end{aligned}
\end{align}
where the second equality follows by interchanging $i$ and $j$ and using $\bR^{ji}=-\bR^{ij}$. Moreover, by~\eqref{bul2vecs}, for a.e. $x\in B^n$, the bivector $\vet_{ij}\sbul\bR^{ij}$ is a sum of wedge products, each containing a tangential factor. Hence its contraction with the normal
vector $\g^k\g_k\bP=n\bH$ is tangent-valued. It follows that
\begin{align}\label{etabulRtan}
    \pi_N \big(\mko (\vet_{ij}\sbul\bR^{ij})\mkt \res\g^k\g_k\bP\big)=0.
\end{align}
Now we define 
\begin{equation}\label{R2Rexplicit}
 \vec{\mca R}_2[\bR;g]
 =-\frac12\pi_N\Big(
       \big(\mko (\g^k\vet_{ij})\sbul\bR^{ij}\big)\mkt \res\g_k\bP\Big).
\end{equation}
The estimate~\eqref{rmdbdnRS} follows from~\eqref{eq:immcon4d} and~\eqref{estg2Phi}, together with the inequalities~\eqref{prolorgen}
and~\eqref{ineproENn}. Finally, substituting~\eqref{normalderRalg}--\eqref{R2Rexplicit} into~\eqref{diffQRresdP} yields
\begin{equation*}
    \pi_N\delta(\vec Q_R\mkt\res d\bP)
    =-\pi_N\Big((\delta\bR)\mkt\ovs{\res}{\res}_g d\bP\Big)
  +\vec{\mca R}_2[\bR;g].
\end{equation*}
This completes the proof.
\end{proof}

Combining Proposition~\ref{prop:LSRsys} and
Lemmas~\ref{lm:d*XwedPhi}--\ref{lm:nd*RS}, we obtain the structural equations that will be used in
Section~\ref{sec:pfmainThm}.

\begin{Co}\label{co:sysvecu}
Let $\bL$, $S$, $\bR$, $\vt_{\dil}$, $\bvt_{\rot}$, and $\vec K$ be as in
Proposition~\ref{prop:LSRsys}. Define
\begin{equation}\label{defuetabul}
    \bu\coloneq
    \vet\mkt\ovs{\sbul}{\res}_g\bR+\vet\mkt\resg S -(3n-6)\,d\big(\lap_g^{h-1}\bP\big)\wres d\bP.
\end{equation}
Then we have 
\begin{equation}\label{estvecu}
    \|\vec u\|_{\msc N_{n,2}}\le C(n,\La) \big(\|\bR\|_{\msc N_{n,2}}+\|S\|_{\msc N_{n,2}}+\vae_{\bP}\big).
\end{equation}
Moreover, the following hold in $\mca D'\big(\bwe^2\R^m\ot\bwe^n\R^n\big)$ and $\mca D'(\R^m)$ respectively:\footnote{We retain the divergence-form equation~\eqref{eq:d*du=d*R3+} rather than writing $\de\mko d\vec u=\de \vec{\mca R}_3$, since $d*_g \vec{\mca R}_3\in W^{1-h,\lf(\frac{2h}{h+1},2\rg)}$ and multiplication by $(\det g)^{-1/2}$ falls under the endpoint case of~\eqref{eq:fTproinene}. The resulting $L^1$ term on the right-hand side is less convenient for applying the elliptic estimates in Section~\ref{sec:ellestcri}, for which we would need to
return to the divergence form.}
\begin{align}
 d*_g d\bu&=d*_g\vec{\mca R}_3,
 \label{eq:d*du=d*R3+}\\
 \pi_N\delta(\bu\mkt\res d\bP)
 &=(2n-6)\lap_g^{h}\bP+\vec{\mca R}_4,
 \label{eq:ndotd*uresdp}
\end{align}
where the remainders satisfy
\begin{align}\label{remtotptbd}
\begin{aligned}
 \|\vec{\mca R}_3\|_{\msc E_n}
 +\|\vec{\mca R}_4\|_{\msc E_n}
 \le C(n,\La)\Big(&\|\vt_{\dil}\|_{\msc E_n}
 +\|\bvt_{\rot}\|_{\msc E_n}
 +\varepsilon_{\bP}\big(
   \|S\|_{\msc N_{n,2}}+\|\bR\|_{\msc N_{n,2}}
   +\vae_{\bP}\big)\Big).
\end{aligned}
\end{align}
\end{Co}

\begin{proof}
Set $\bX\coloneq d\big(\lap_g^{h-1}\bP\big)$. By~\eqref{nabbPspa} and Lemma~\ref{lm:proinene}, we have
\begin{equation}\label{estX}
    \|\bX\|_{\msc N_{n,2}}\le C(n)\|\lap_g^{h-1}\bP\|_{W^{3-h,2}}\le C(n,\La) \mko \vae_{\bP}.
\end{equation}
The estimate~\eqref{estvecu} then follows from~\eqref{prolorgen} and~\eqref{estX}. In addition, there holds
\begin{equation}\label{deltaXK}
    d\bX=0,\qquad  \delta\bX=\delta \mko d\big(\lap_g^{h-1}\bP\big)=\vec K.
\end{equation}
By Lemma~\ref{lm:d*XwedPhi}, and using
$\vec K\we d\bP=-d\bP\we\vec K$, we obtain
\begin{align}\label{lapHndPdd*sys}
 d\bP\we\vec K
 =-\delta\big(\bX\ovs{\ovwe}{\we}d\bP\big)
   +d\big(\bX\wres d\bP\big)+\vec{\mca R}_1[\bX;g].
\end{align}
Substituting~\eqref{lapHndPdd*sys} into~\eqref{eq:RSKsys} and using~\eqref{defuetabul}, it follows that
\begin{align}\label{d*dsysRSX}
&\delta\Big((2n-5)\bR+\vet\mkt\ovs{\sbul}{\sbul}_g\bR
 +\vet\mko\bulg S+(3n-6)\bX\ovs{\ovwe}{\we}d\bP\Big)
 +d\bu=\vec{\mca R}_3,
\end{align}
where
\begin{equation}\label{defvecR2}
 \vec{\mca R}_3\coloneq
 \vec{\mca R}_0+(3n-6)\vec{\mca R}_1[\bX;g].
\end{equation}
By~\eqref{eq:defd*glap}, we have $d*_g\de=-d*_g*_g \,d\,*_g=0$. The identity~\eqref{eq:d*du=d*R3+} then follows from applying $d\,*_g$ to~\eqref{d*dsysRSX}. Combining~\eqref{R3ptbd},~\eqref{vecR4ptbd}, and~\eqref{estX} gives the required estimate for $\vec{\mca R}_3$.

It remains to prove~\eqref{eq:ndotd*uresdp} with the required estimate for $\vec{\mca R}_4$. By Lemma~\ref{lm:nd*RS} and~\eqref{sys:d*dSR}, we have
\begin{align}
\begin{aligned}\label{normalBpart}
    &\pi_N\mko\delta\Big(
 \big(\vet\mkt\ovs{\sbul}{\res}_g\bR+\vet\mkt\resg S\big)\mkt \res d\bP\Big)\\
    &=-\pi_N\Big(
  \big(\bL\wres d\bP+d\bP\we\vec K+\bvt_{\rot}\big)
        \ovs{\res}{\res}_g d\bP\Big)
  +\vec{\mca R}_2[\bR;g].
  \end{aligned}
  \end{align}
  Now we compute
\begin{align}
\begin{dcases} \label{Ldbcontr}
     \Big(\big(\bL\wres d\bP\big)\mkt\ovs{\res}{\res}_g d\bP\Big)=(\bL_{ji}\we\g^j\bP)\mkt \res\g^i\bP=(\bL_{ji}\cdot\g^i\bP)\g^j\bP-g^{ij}\bL_{ji},\\
 \Big((d\bP\we\vec K)\mkt\ovs{\res}{\res}_g d\bP\Big)=(\g_i\bP\we\vec K)\mkt \res\g^i\bP
   =n\vec K-\pi_T\vec K.
 \end{dcases}
\end{align}
Since $g^{ij}\bL_{ji}=0$, it follows that
\begin{equation}\label{ndotLweresdpres}
 \pi_N\Big(\big(\bL\wres d\bP\big)
           \ovs{\res}{\res}_g d\bP\Big)=0.
\end{equation}
Combining~\eqref{normalBpart}--\eqref{ndotLweresdpres} then gives
\begin{align}\label{applyLmnd*RS}
    \pi_N\mko\delta\Big(
 \big(\vet\mkt\ovs{\sbul}{\res}_g\bR+\vet\mkt\resg S\big)\mkt \res d\bP\Big)=-n\mko\pi_N\vec K -\pi_N\big(\bvt_{\rot}\mkt\ovs{\res}{\res}_g d\bP\big)+\vec{\mca R}_2[\bR;g].
\end{align}
For the contribution of $\bX$, set
$\vec q\coloneq\bX\wres d\bP=\bX_i\we\g^i\bP$. Then we have
\begin{align}\label{qresPhiindex}
 (\vec q\mkt\res d\bP)_k
 =(\bX_i\we\g^i\bP)\mkt\res \g_k\bP = (\bX_i\cdot\g_k\bP)\g^i\bP-\bX_k.
\end{align}
Applying $-\g^k$ to~\eqref{qresPhiindex} and taking the normal projection, we obtain
\begin{align}
\begin{aligned}\label{proqresdP}
     \pi_N\mko\delta(\vec q\mkt\res d\bP)
     &=-\pi_N\Big(
   \g^k(\bX_i\cdot\g_k\bP)\g^i\bP
   +(\bX_i\cdot\g_k\bP)\mko\bII^{ki}
   -\g^k\bX_k\Big)\\
 &=-(\bX_i\cdot\g_k\bP)\mko\bII^{ki}-\pi_N(\de \vec X)\\
 &=-\pi_N \vec K-(\bX_i\cdot\g_k\bP)\mko\bII^{ki}.
 \end{aligned}
\end{align}
Finally, combining~\eqref{defuetabul},~\eqref{applyLmnd*RS}, and~\eqref{proqresdP} yields
\begin{align}\begin{aligned}
    &\pi_N\delta(\bu\mkt\res d\bP)\\
     &=\big(\mko(3n-6)-n\big)\pi_N\vec K
  -\pi_N\big(\bvt_{\rot}\mkt\ovs{\res}{\res}_g d\bP\big)
  +\vec{\mca R}_2[\bR;g]+(3n-6)(\bX_i\cdot\g_k\bP)\mko\bII^{ki}\\
 &=(2n-6)\vec K+\vec{\mca R}_4.
 \end{aligned}
\end{align}
Here we set
\begin{align}\label{defvecR4}
\begin{aligned}
 \vec{\mca R}_4\coloneq -(2n-6)\pi_T\vec K
 -\pi_N\big(\bvt_{\rot}\mkt\ovs{\res}{\res}_g d\bP\big)+\vec{\mca R}_2[\bR;g]+(3n-6)(\bX_i\cdot\g_k\bP)\mko\bII^{ki}.
\end{aligned}
\end{align}
This proves~\eqref{eq:ndotd*uresdp}. The required estimate for $\vec{\mca R}_4$ follows from~\eqref{estpiTK},~\eqref{rmdbdnRS} and~\eqref{estX}, together with the product inequalities~\eqref{prolorgen} and~\eqref{ineproENn}.
\end{proof}

\section{Proof of the main theorem}\label{sec:pfmainThm}
In this section, we complete the proof of Theorem~\ref{th-main} by deriving the conservation laws associated with $E_{\bc}$ and applying the results from previous sections. The cases $n=2$ and $n=4$ have been proved in~\cite{Riv08} and~\cite{BerLanMarRiv26} respectively. Henceforth, let $n=2h\ge6$ with $h\in \N^+$.

Using local coordinates, we assume $\bP\in \mathcal{I}_{h-1,2}(B^n, \R^m)$. Throughout Section~\ref{sec:pfmainThm}, we write $\p_i=\p_{x^i}$, and use the notation~\eqref{defbiopes} and the convention~\eqref{eq:convsonab}. We fix a constant $\La\ge 1$ such that
\begin{align}\label{weaimmconbP}
    \Lambda^{-1} \mko |v|^2_{\R^n} \le | d\bP_x(v)|^2_{\R^m}  \le \Lambda  \mko |v|^2_{\R^n},\qquad \text{for a.e. }x\in B^n \text{ and all }v\in T_x B^n.
\end{align}
As in~\eqref{eq:defE},  we fix $\vec c=(c_\la)_{\la\in\mca L}$ and define 
\[E_{\vec c}(\bP)=\int_{B^n} \bigg(\big| \g^{(h-1)}\mko\bII\big|_g^2+\sum_{\la\in \mca L} c_\la P_\la(\bP)\bigg) \, \dvol_g.\]
For any open set $U\subset B^n$, we say $\bP$ is a weak critical point of $E_{\bc}$ on $U$ if for any $\vec w\in C_c^\infty(U,\R^m)$, it holds that
\[
   \frac d{dt} \left. E_{\vec c}\left( \bP+t\mkt \vec w \right)\right|_{t=0}=0.
\]

\subsection{The Euler--Lagrange equation}\label{sec:EullagestV}
\
\vskip5pt
Before deriving the Euler--Lagrange equation, we carry out the rescaling needed for the estimates in Sections~\ref{sec:strid} and~\ref{sec:pfmainThm}.

\subsubsection*{Rescaling.}

For $\rho\in(0,1)$, we define 
\begin{equation}\label{phidil}
    \bP_\rho(x)\coloneq\rho^{-1}\bP(\rho\mko x).
\end{equation}
Then since $E_{\bc}$ is invariant under rescaling, $\bP_\rho$ is a weak critical point of $E_{\bc}$ on $B_1$ if and only if $\bP$ is a weak critical point of $E_{\bc}$ on $B_\rho$. In addition, the condition~\eqref{weaimmconbP} remains valid if $\bP$ is replaced by $\bP_\rho$. By change of variables and using the scale-homogeneous Sobolev norms~\eqref{eq:uniSobnm}, we have
\begin{align}\label{Dphinormscal}
    \|D^2\bP\|_{W^{h-1,2}(B_\rho)}=\|D^2 \bP_\rho\|_{W^{h-1,2}(B_1)}.
\end{align}
By Sobolev--Lorentz embeddings, for each $1\le k\le h$, we have $D^{k+1}\bP\in L^{n/k,2}(B^n)$. In addition, the inequality~\eqref{equivSobnorm} gives
\begin{equation}\label{comphomSobLnk}
     \|D^2\bP\|_{W^{h-1,2}(B_\rho)}\le C(n)  \sum_{k=1}^h \|D^{k+1}\bP\|_{L^{\frac nk,2}(B_\rho)}.
\end{equation}
In view of~\eqref{Dphinormscal}--\eqref{comphomSobLnk}, after restricting
to a sufficiently small ball and rescaling, we may relabel $\bP_\rho$
as $\bP$ and assume
\begin{equation}\label{defep_0}
    \varepsilon_{\bP}\coloneq \|D^2\bP\|_{W^{h-1,2}(B^n)}\le\vae_0<1.
\end{equation}
Here $\vae_0\in (0,1)$ is a constant to be determined later.

Since the $n$-form $\Big(\big| \g^{(h-1)}\mko\bII\big|_g^2+\sum_{\la\in \mca L} c_\la P_\la(\bP)\Big) \, \dvol_g$ is pointwise invariant under translations, dilations, and rotations in the ambient space $\R^m$, we apply Noether's theorem to derive the Euler--Lagrange equation in divergence form, together with some conservation laws for weak critical points of $E_{\vec c}$. These conservation laws form the main ingredients for Proposition~\ref{sysestLSR}, proved at the end of the next subsection.
\subsubsection*{The Noether current associated with translations.}
\begin{Lm}
\label{lm-critic-W^{1,1}} Let $\bP\in \mca I_{h-1,2}(B^n,\R^m)$. Then there exist a finite index set $I$, coefficients $\{a_\ga\}_{\ga\in I}\subset \R$ depending only on $\bc$, and covector fields $\big\{\vec Q_\ga(\bP) \big\}_{\ga\in I}\subset W^{-h,2}+L^1\big(B^n,\R^m\ot \R^n\big)$ such that the following hold:
\begin{enumerate}[label=(\roman*),leftmargin=2.5em]
    \item \label{expQgabP} Each $\vec Q_\ga(\bP)$ is a partial contraction (with one free index) in the form
    \begin{equation}
        \label{pconvecQga}\text{pcontr}\Big(\big(\g^{(\ell_1)}\mko\bII\cdot \g^{(\ell_2)}\mko\bII\big)\ot\cdots \ot \big(\g^{(\ell_{2s-1})}\mko\bII\cdot \g^{(\ell_{2s})}\mko\bII\big)\ot \g^{(\ell_0)}\bP \Big),
    \end{equation}
    for suitable $s\in \N^+$ and non-negative integers $\ell_i$ ($0\le i\le 2s$), depending on $\ga\in I$. Moreover, we have
\begin{align}\label{condskivQ}\begin{dcases}
    1\le \ell_0\le n,\\[1ex]
     \forall 1\le i\le 2s,\qquad 0\le \ell_{i}\le n-2,\\
    \sum_{i=0}^{2s} (\ell_i+1)=n+2.   
    \end{dcases}
\end{align}
    \item $\bP$ is a weak critical point of $E_{\bc}$ if and only if 
    \begin{equation}\label{ELequ}
        d*_g\bigg(d\Big(\lap_g^{h}\mkt\bP\Big)+ \sum_{\ga\in I} a_{\ga}\mko \vec Q_{\ga}(\bP)\bigg)=0.
        \end{equation}
\end{enumerate}
\end{Lm}
\begin{proof}
For the estimates of the Noether current and subsequent applications, we may assume without loss of generality that~\eqref{weaimmconbP} and~\eqref{defep_0} hold. 
By Lemma~\ref{lm:proinene} and~\eqref{nabbPspa}, for any $2\le k \le n+1$ and $1\le i_1,\dots,i_k\le n$, we obtain that
\begin{align}\label{nablaphispa}
  \big(\g^{(k-2)}\bII\big)_{i_1\dots i_k}\in\begin{dcases}
         W^{h+1-k,2}(B^n,\R^m)&\text{ if }\,2\le k\le n,\\
    L^1+W^{-h,2}(B^n,\R^m) \,&\text{ if }\,k=n+1.
    \end{dcases}
\end{align}
When $k\le h+1$, the Sobolev--Lorentz embedding gives
\begin{equation}\label{gIILp}
   \big(\g^{(k-2)}\bII\big)_{i_1\dots i_k}\in L^{\frac n{k-1},2}(B^n)\hookrightarrow L^{\frac n{k-1}}(B^n).
\end{equation} 
Applying H\"older's inequality with~\eqref{eq:condsks}--\eqref{defmcaL} then implies
\begin{equation}\label{integrandL1}
    \big| \g^{(h-1)}\mko\bII\big|_g^2+\sum_{\la\in \mca L} c_\la P_\la(\bP)\; \in L^1(B^n).
\end{equation}
Hence $E_{\bc}(\bP)$ is well-defined for $\bP\in \mca I_{h-1,2}(B^n,\R^m)$. Now we rewrite the integrand of $E_{\bc}$ in terms of partial derivatives of $\bP$. Expanding the iterated covariant derivatives using the Christoffel symbols, for any $k\ge 2$ and $1\le i_1,\dots,i_k\le n$, we can write
\begin{equation}\label{difcov}
    \big(\g^{(k-2)}\bII\big)_{i_1\dots i_k}=\g_{i_1\dots i_k}\bP=\p_{i_1}\cdots \p_{i_{k}}\bP+\sum_{r\in \wti R_{k}} a_{r,i_1\dots i_k}\mko\wti P_{r,i_1\dots i_k}(\bP).
\end{equation}
Here $\wti R_k$ is a finite index set, the coefficients $a_{r,i_1\dots i_k}\in \R$ are constants, and each $ \wti P_{r,i_1\dots i_k}(\bP)$ is of the form
 \begin{align}\label{expPr}
    (\det g_{ij})^{-q}\bigg( \prod_{l=1}^{s}\big(D^{\al_{2l-1}} \bP\cdot D^{\al_{2l}}\bP\big)\bigg) D^{\al_0}\bP,
 \end{align}
 for suitable $q,s\in \N$ and multi-indices $\al_i$ ($0\le i\le 2s$), depending on $r$.\footnote{Equation~\eqref{difcov} should be distinguished from~\eqref{difcovcan} with $\vec f=\bP$, which is written in geodesic normal coordinates. Here and below, the symbols $q$, $s$, and $\al_i$ are local and may vary from one occurrence to another.}
 Moreover, these multi-indices satisfy
 \begin{align}\label{congpdif}
    \begin{dcases}
    1\le |\al_0|\le k-1,\\[1ex]
    \forall\, 1\le i\le 2s,\qquad 1\le |\al_{i}|\le k,\\
    \sum_{i=0}^{2s} (|\al_{i}|-1)=k-1.
    \end{dcases}
\end{align}  
Consequently, we can write
 \begin{align}\label{cooexpgra}
 \begin{aligned}
     &\big| \g^{(h-1)}\mko\bII\big|_g^2+\sum_{\la\in \mca L} c_\la P_\la(\bP)\\
     &= g^{i_1j_1}\cdots g^{i_{h+1}j_{h+1}}\big(\p_{i_1}\cdots \p_{i_{h+1}}\bP\big)\cdot \big(\p_{j_1}\cdots \p_{j_{h+1}} \bP\big)+\sum_{r\in I_1} b_{r,1} \mko\wti Q_{r,1}(\bP).
     \end{aligned}
 \end{align}
 Here $I_1$ is a finite index set, the coefficients $b_{r,1}\in \R$ depend only on $\bc$, and each $\wti Q_{r,1}(\bP)$ is of the form 
 \begin{align}\label{expQr}
    (\det g_{ij})^{-q} \prod_{l=1}^s \big(D^{\al_{2l-1}} \bP\cdot D^{\al_{2l}}\bP\big),
 \end{align}
for suitable $q,s\in \N^+$ and multi-indices $\al_i$ ($1\le i\le 2s$), depending on $r$. Moreover, it follows from~\eqref{difcov}--\eqref{congpdif} and~\eqref{eq:condsks} that the multi-indices in~\eqref{expQr} satisfy
\begin{align}\label{condsksQ}
    \begin{dcases}
    \forall \,1\le i\le 2s,& 1\le |\al_{i}|\le h+1,\\[0.3ex]
     \forall \,1\le l\le s,& |\al_{2l-1}|+|\al_{2l}|\le n+1,\\
    \sum_{i=1}^{2s} (|\al_{i}|-1)=n.&
    \end{dcases}
\end{align}  

Next, we consider a variation $(\bP_t)_{|t|<1}$ in $\mathcal{I}_{h-1,2}(B^n, \R^m)$ that is $C^1$ in $t$, with $\bP_0=\bP$. Here $\mathcal{I}_{h-1,2}(B^n, \R^m)$ is regarded as an open subset of $ W^{1,\nf}\cap W^{h+1,2}(B^n,\R^m)$, equipped with the intersection norm.
We write $g=g_{\bP}$, $g_t=g_{\bP_t}$, $\bII=\bII_{\bP}$, $\bII_t=\bII_{\bP_t}$, and denote by $\g_t$ the covariant derivative with respect to $g_t$. Let
\begin{equation}\label{ptvardef}
    \bw\coloneq \frac d{dt}\mko \bP_t\Big|_{t=0}\;\in W^{1,\nf}\cap W^{h+1,2}(B^n,\R^m).
\end{equation} 
Applying~\cite[Prop.~2.5.2]{Cartan71} and H\"older's inequality as in the proof of~\eqref{integrandL1}, we obtain
\begin{equation}\label{ptvarinL1}
  t\mapsto \Big(\big| \g_t^{(h-1)}\mko\bII_t\big|_{g_t}^2+\sum_{\la\in \mca L} c_\la P_\la(\bP_t)\Big) \, \dvol_{g_t}\; \in C^1\Big((-1,1),L^1\big(B^n,\bwe^n \R^n\big)  \Big).
\end{equation}

We now compute the first variation of the energy density in~\eqref{ptvarinL1}. By standard computations, see for instance \cite[Sec.~5.1]{LaMaRi26}, we have
\begin{align}\label{varmet}\begin{dcases}
    \frac{d}{dt} \,g_t^{ij}\mko\Big|_{t=0}=-g^{ik}g^{\ell j}(\p_{k}\vec w \cdot\p_{\ell}\bP+\p_{\ell}\vec w \cdot\p_{k}\bP),\\
    \frac{d}{dt} \,\dvol_{g_t}\mko\Big|_{t=0}=\lan \g\bw ,\g\bP\ran_g\, \dvol_g.
    \end{dcases}
\end{align}
Combining~\eqref{cooexpgra}--\eqref{expQr} with~\eqref{varmet}, we obtain
\begin{align}\label{ptvarrou}\begin{aligned}
    &\frac d{dt}\bigg (\Big(\big| \g_t^{(h-1)}\mko\bII_t\big|_{g_t}^2+\sum_{\la\in \mca L} c_\la P_\la(\bP_t)\Big) \, \dvol_{g_t}\bigg)\bigg|_{t=0}\\
    &=\bigg(2\,g^{i_1j_1}\cdots g^{i_{h+1}j_{h+1}}\big(\p_{i_1}\cdots \p_{i_{h+1}}\bw\big)\cdot \big(\p_{j_1}\cdots \p_{j_{h+1}} \bP\big)+\sum_{r\in I_2} b_{r,2} \mkt\wti Q_{r,2}(\bw,\bP)\bigg)\mkt \dvol_g.
    \end{aligned}
\end{align}
Here $I_2$ is a finite index set, the coefficients $b_{r,2}\in \R$ depend only on $\bc$, and each $\wti Q_{r,2}(\bw,\bP)$ is of the form
 \begin{align}\label{expQr2}
    (\det g_{ij})^{-q} \big(D^{\al_1}\bw\cdot D^{\al_2}\bP\big)\prod_{l=2}^s \big(D^{\al_{2l-1}} \bP\cdot D^{\al_{2l}}\bP\big),
 \end{align}
for suitable $q,s\in \N^+$ and multi-indices $\al_i$ ($1\le i\le 2s$), depending on $r$. In addition, by~\eqref{condsksQ}, the multi-indices in~\eqref{expQr2} satisfy
\begin{align}\label{condskstQ}
    \begin{dcases}
    \forall\, 1\le i\le 2s,\qquad 1\le |\al_{i}|\le h+1,\\
    \sum_{i=1}^{2s} (|\al_{i}|-1)=n,\\[0.2ex]
    |\al_{1}|+|\al_{2}|\le n+1.
    \end{dcases}
\end{align}  
Applying~\eqref{difcovcan} with $\vec f=\bP$ and $\vec f=\bw$ respectively, we rewrite the right-hand side of~\eqref{ptvarrou} in terms of polynomials in the components of iterated covariant derivatives of $\bP$ and $\bw$.
Since the integrand of $E_{\bc}(\bP_t)$ is independent of reparametrizations, its pointwise first variation with respect to $t$ is invariant as well. 
By Lemma~\ref{lm:gpdot2ff} and the invariance-theory argument in~\cite[Sec.~2]{Atiyah73}, we then obtain\footnote{For smooth immersions, this follows by using geodesic normal coordinates as in~\eqref{difcovcan}. The approximation result~\cite[Thm.~4.6]{LaMaRi26} then implies~\eqref{ptvargwgp} for weak immersions $\bP\in \mca I_{h-1,2}(B^n,\R^m)$.}
\begin{align}\label{ptvargwgp}\begin{aligned}
     &\frac d{dt}\bigg (\Big(\big| \g_t^{(h-1)}\mko\bII_t\big|_{g_t}^2+\sum_{\la\in \mca L} c_\la P_\la(\bP_t)\Big) \, \dvol_{g_t}\bigg)\bigg|_{t=0}\\
     &=\bigg(2\mkt \Big\lan \g^{(h+1)}\mko\bw,\g^{(h+1)}\mko\bP \Big\ran_g+\sum_{r\in I_3} b_{r,3}\mkt Q_r(\bw,\bP)\bigg)\mkt \dvol_g.
     \end{aligned}
\end{align}
Here $I_3$ is a finite index set, the coefficients $b_{r,3}\in \R$ depend only on $\bc$, and each $ Q_{r}(\bw,\bP)$ is a complete contraction in the form
\begin{align}\label{contgwgp}
    \text{contr}\Big(\big(\g^{(k_1)}\bw\cdot \g^{(k_2)}\bP \big)\ot\big(\g^{(k_3)}\mko\bII\cdot \g^{(k_4)}\mko\bII\big)\ot\cdots \ot\big(\g^{(k_{2s-1})}\mko\bII\cdot \g^{(k_{2s})}\mko\bII\big) \Big),
\end{align}
for suitable $s\ge 2$ and non-negative integers $k_i$ ($1\le i\le 2s$), depending on $r$. 
Moreover, it follows from~\eqref{condskstQ} that the orders $k_i$ appearing in~\eqref{contgwgp} satisfy\footnote{The terms $Q_r(\bw,\bP)$ appearing in~\eqref{ptvargwgp} indeed satisfy $k_2\le h+1$ and $k_i\le h-1$ for $3\le i\le 2s$, which ensure their $L^1$ integrability by Sobolev embeddings and H\"older's inequality. 
We retain the weaker conditions~\eqref{condskiQ} to include the additional terms arising from the computations below.}
\begin{equation}\label{condskiQ}
\begin{dcases}
    1\le k_1\le h+1,\qquad 1\le k_2\le n,\\
    k_1+k_2\le n+1,\\
    0\le k_i\le n-1\qquad (3\le i\le 2s),\\
    \sum_{i=1}^{2s}(k_i+1)=n+4.
\end{dcases}
\end{equation}
In particular, since $k_2\ge 1$ and $k_i\ge 0$ for all $i\ge 3$, the last line of~\eqref{condskiQ} implies that
\begin{align}\label{k1+ki<n}
    \forall\, 3\le i\le 2s,\qquad k_1+k_i\le n-1.
\end{align}
We claim that the integrand on the right-hand side of~\eqref{ptvargwgp} admits the following decomposition in the sense of distributions:
\begin{align}\label{ptvardivgw}\begin{aligned}
    &2\mkt \Big\lan \g^{(h+1)}\mko\bw,\g^{(h+1)}\mko\bP \Big\ran_g+\sum_{r\in I_3} b_{r,3}\mkt Q_r(\bw,\bP)\\
    &=-2\, \text{div}_g\bigg(\g \bw\cdot \lap_g^{h}\mko \bP+\sum_{\ga'\in I'} \bar a_{\ga'}\mko \bar P_{\ga'}(\bw,\bP)\bigg)\\
    &\quad +2\mkt\bigg\lan\g \bw,\, \g \Big(\lap_{g}^{h}\mkt\bP\Big)+ \sum_{\ga\in I} a_{\ga}\mko \vec Q_{\ga}(\bP) \bigg\ran_g.
    \end{aligned}
\end{align}
Here $I, I'$ are finite index sets, the coefficients $a_\ga,\bar a_{\ga'}\in \R$ depend only on $\bc$, $\vec Q_\ga(\bP)$ can be expressed as in~\ref{expQgabP} and each $\bar P_{\ga'}(\bw,\bP)$ is a partial contraction (with one free index) in the form
\begin{align}\label{condvecPgaw}
    \text{pcontr}\Big(\big(\g^{(\ell'_1)}\bw\cdot \g^{(\ell'_2)}\bP \big)\ot\big(\g^{(\ell'_3)}\mko\bII\cdot \g^{(\ell'_4)}\mko\bII\big)\ot\cdots \ot\big(\g^{(\ell'_{2\bar s-1})}\mko\bII\cdot \g^{(\ell'_{2\bar s})}\mko\bII\big) \Big),
\end{align}
for suitable $\bar s\ge 1$ and non-negative integers $\ell'_i$ ($1\le i\le 2\bar s$), depending on $\ga'\in I'$. In addition, the integers $\ell'_i$ in~\eqref{condvecPgaw} satisfy \begin{align}\label{condsliP}\begin{dcases}
    1\le \ell'_1\le h,\\
    1\le \ell'_2\le n-1,\\[1ex]
     \forall\, 3\le i\le 2\bar s,\qquad 0\le \ell'_{i}\le n-3,\\
    \sum_{i=1}^{2\bar s} (\ell'_i+1)=n+3.   
    \end{dcases}
\end{align}
Before proving~\eqref{ptvardivgw}, we first estimate the norm of its right-hand side. We retain the spaces $\msc E_n$, $\msc N_{n,1}$, and $\msc N_{n,2}$ from~\eqref{defcritspacesIV}, and define 
\begin{equation}\label{defNn0}
    \msc N_{n,0}(U,E)\coloneq W^{-h,2}+L^1(U,E).
\end{equation}
Combining~\eqref{weaimmconbP},~\eqref{defep_0}, and~\eqref{nabbPspa} with~\eqref{eq:fTproinene} yields
\begin{align}\label{gwdotlapPhiest}
    \big\|\g \bw\cdot \lap_g^{h}\mko \bP\big\|_{W^{1-h,2}(B^n)}\le C(\La,n) \mkt\vae_{\bP} \mko\|D\bw\|_{L^\nf \cap W^{h,2}(B^n)}.
\end{align}
Similar to~\eqref{nabbPspa}, by induction and applying Lemma~\ref{lm:proinene} with~\eqref{defep_0}, for any $2\le k\le n$, and $1\le i_1,\dots,i_k\le n$, we obtain
\begin{equation}\label{estnablaw}
    \|\g_{i_1\dots i_k}\bw\|_{W^{h+1-k,2}(B^n)}\le C(n,\La)\Big(\|D^2\bw\|_{W^{h-1,2}}+\vae_{\bP}\|D\bw\|_{L^\nf\cap W^{h,2}}\Big).  
\end{equation}
For $\bar P_{\ga'}(\bw,\bP)$ in~\eqref{condvecPgaw}, set $a_1=\ell'_1-1$, $a_2=\ell'_2-1$, and $a_i=\ell'_i+1$ for $i\ge3$. By~\eqref{condsliP}, we have
\begin{equation}\label{condweig}
    \sum_{i=1}^{2\bar s}a_i=n-1,\qquad 0\le a_1\le h-1,\qquad 0\le a_i\le n-2\quad(i\ge2).
\end{equation}
Combining~\eqref{nabbPspa} and~\eqref{estnablaw} with the product inequality~\eqref{eq:multicritgain} yields $\bar P_{\ga'}(\bw,\bP)\in \msc E_n (B^n)$. 
To track the norms, by~\eqref{condweig}, there exists at least one $a_i>0$ with $i\ge 2$. Moreover, if $a_1=0$, then $\bar s\ge2$ and there are at least two $a_i>0$ with $i\ge 2$. Thus by~\eqref{defep_0}, we obtain in all cases
\begin{equation}\label{estPgamma'wp}
     \|\bar P_{\ga'}(\bw,\bP)\|_{ \msc E_n (B^n)}\le C(n,\La) \mko\vae_{\bP} \Big(\vae_{\bP}\|D\bw\|_{L^\nf\cap W^{h,2}(B^n)}+\|D^2\bw\|_{W^{h-1,2}(B^n)} \Big).
\end{equation}
To estimate $\vec Q_\ga(\bP)$, by~\eqref{pconvecQga}--\eqref{condskivQ}, we apply~\eqref{eq:multiposine} with $a=n$   and obtain
\begin{equation}\label{estQgammap}
    \|\vec Q_\ga(\bP)\|_{\msc N_{n,0}(B^n)}
    \le C(n,\La)\mko\vae_{\bP}^2.
\end{equation}
Here the presence of at least two
$\bII$-factors yields the factor $\vae_{\bP}^2$. Combining~\eqref{gwdotlapPhiest} and~\eqref{estPgamma'wp}--\eqref{estQgammap} with~\eqref{eq:fTproinene}, we obtain that the right-hand side of~\eqref{ptvardivgw} lies in $\msc N_{n,0}$. Thus by the approximation theorem for weak immersions~\cite[Thm.~4.6]{LaMaRi26}, it suffices to prove~\eqref{ptvardivgw} for smooth $\bP$ and $\bw$, with $g=g_{\bP}$.  

For the principal term on the left-hand side of~\eqref{ptvardivgw}, the Leibniz rule~\eqref{eq:Leibru} gives
\begin{align} \label{giwjbPint}
\begin{aligned}
    \g_{i_1\dots i_{h+1}}\bw\cdot \g_{j_1\dots j_{h+1}}\bP&=\sum_{k=1}^{h-1}(-1)^{k+1}\g_{i_k}\Big(\g_{i_{k+1}\dots i_{h+1}}\bw\cdot \g_{i_{k-1}\dots i_1j_1\dots j_{h+1}}\bP \Big)\\
    &\quad +(-1)^{h-1}\g_{i_{h}i_{h+1}}\bw\cdot \g_{i_{h-1}\dots i_1j_1\dots j_{h+1}}\bP,\\
    \end{aligned}
\end{align}
where for $k=1$ we set $ \g_{i_{k-1}\dots i_1j_1\dots j_{h+1}}:=  \g_{j_1\dots j_{h+1}}$. After contracting $i_s$ against $j_s$ for all $1\le s\le h+1$, each term under the summation symbol in~\eqref{giwjbPint} is of the form $\text{div}_g\big(\bar P_{\ga'}(\bw,\bP)\big)$ as in~\eqref{ptvardivgw}--\eqref{condsliP}. 

To handle the last term in~\eqref{giwjbPint}, we first permute the indices of $\g^{(n)}\bP$ up to lower-order terms. Indeed, by~\eqref{difcovcan}, we can write
\begin{equation}\label{tiPi1jn2+1}
    \g_{i_{h-1}\dots i_1j_1\dots j_{h+1}}\bP= \g_{j_{h+1}\mko i_{h-1}\dots i_1j_1\dots j_{h}}\bP+\ti P_{i_1\dots i_{h-1}\mko j_1\dots j_{h+1}}(\bP),
\end{equation}
where $\ti P_{i_1\dots i_{h-1}\mko j_1\dots j_{h+1}}(\bP)$ is a linear combination of partial contractions of tensor products of covariant derivatives $\g^{(k)}\bP$ ($1\le k\le n-1$). By the Leibniz rule, we have
\begin{align}\begin{aligned}\label{gjn2in2n21w}
    \g_{i_{h}i_{h+1}}\bw\cdot \g_{j_{h+1}\mko i_{h-1}\dots i_1j_1\dots j_{h}}\bP
    &=\g_{j_{h+1}}\Big(  \g_{i_{h}i_{h+1}}\bw\cdot \g_{ i_{h-1}\dots i_1j_1\dots j_{h}}\bP\Big)\\
    &\quad -\g_{j_{h+1}\mko i_{h}\mko i_{h+1}}\bw\cdot \g_{ i_{h-1}\dots i_1j_1\dots j_{h}}\bP.
    \end{aligned}
\end{align}
After contracting $i_s$ against $j_s$ for all $1\le s\le h+1$, the first term on the right-hand side of~\eqref{gjn2in2n21w} is of the form $\text{div}_g\big(\bar P_{\ga'}(\bw,\bP)\big)$ as in~\eqref{ptvardivgw}--\eqref{condsliP}. For the last term of~\eqref{gjn2in2n21w}, we permute the indices of $\g^{(3)}\bw$: by~\eqref{difcovcan} (or using the Ricci identity for covector fields together with~\eqref{Gaseq}), we can write
\begin{equation}
    \g_{j_{h+1}\mko i_{h}\mko i_{h+1}}\bw=\g_{i_{h}\mko i_{h+1}\mko j_{h+1}}\bw+\ti P_{i_{h}\mko i_{h+1}\mko j_{h+1}}(\bw,\bP),
\end{equation}
where $\ti P_{i_{h}\mko i_{h+1}\mko j_{h+1}}(\bw,\bP)$ is a linear combination of partial contractions of $(\bII \cdot \bII)\ot \g \bw$. By the Leibniz rule again, we have
\begin{align}\label{gin2i1j1jn2bP}
    \begin{aligned}
        \g_{i_{h}\mko i_{h+1}\mko j_{h+1}}\bw\cdot \g_{ i_{h-1}\dots i_1j_1\dots j_{h}}\bP&= \g_{i_{h}}\Big(\g_{i_{h+1}\mko j_{h+1}}\bw\cdot \g_{ i_{h-1}\dots i_1j_1\dots j_{h}}\bP\Big)\\
        &\quad -\g_{i_{h+1}\mko j_{h+1}}\bw\cdot \g_{ i_{h}\dots i_1j_1\dots j_{h}}\bP
    \end{aligned}
\end{align}
For the last term of~\eqref{gin2i1j1jn2bP}, we permute the indices of $\g^{(k)}\bP$ again. By~\eqref{difcovcan}, we can write
\begin{equation}\label{tiP'i1jn2}
    \g_{ i_{h}\dots i_1j_1\dots j_{h}}\bP=\g_{ i_1j_1\dots i_{h}\mko j_{h}}\bP+\ti P'_{i_1j_1\dots i_{h}\mko j_{h}}(\bP),
\end{equation}
where $\ti P'_{i_1j_1\dots i_{h}\mko j_{h}}(\bP)$ is a linear combination of partial contractions of tensor products of covariant derivatives $\g^{(k)}\bP$ ($1\le k\le n-1$).
Combining~\eqref{tiPi1jn2+1}--\eqref{tiP'i1jn2}, we obtain that
\begin{align}\label{intindwbP}
    \begin{aligned}
    &\g_{i_{h}i_{h+1}}\bw\cdot \g_{i_{h-1}\dots i_1j_1\dots j_{h+1}}\bP\\
          &=\g_{i_{h}i_{h+1}}\bw\cdot \ti P_{i_1\dots i_{h-1}\mko j_1\dots j_{h+1}}(\bP)-\ti P_{i_{h}\mko i_{h+1}\mko j_{h+1}}(\bw,\bP)\cdot \g_{ i_{h-1}\dots i_1j_1\dots j_{h}}\bP\\
          &\quad +\g_{i_{h+1}\mko j_{h+1}}\bw\cdot \ti P'_{i_1j_1\dots i_{h}\mko j_{h}}(\bP)+\g_{j_{h+1}}\Big(  \g_{i_{h}i_{h+1}}\bw\cdot \g_{ i_{h-1}\dots i_1j_1\dots j_{h}}\bP\Big)\\
          &\quad -\g_{i_{h}}\Big(\g_{i_{h+1}\mko j_{h+1}}\bw\cdot \g_{ i_{h-1}\dots i_1j_1\dots j_{h}}\bP\Big)+\g_{i_{h+1}\mko j_{h+1}}\bw\cdot \g_{ i_1j_1\dots i_{h}\mko j_{h}}\bP.
    \end{aligned}
\end{align}
After contracting $i_s$ against $j_s$ for all $1\le s\le h+1$, each of the first three terms on the right-hand side of~\eqref{intindwbP} is a finite linear combination of contractions of the form~\eqref{contgwgp} satisfying~\eqref{condskiQ}. 
Concerning the last term of~\eqref{intindwbP}, the Leibniz rule gives
\begin{align}\label{lapn2bP}\begin{aligned}
    \g_{i_{h+1}\mko j_{h+1}}\bw\cdot \g_{ i_1j_1\dots i_{h}\mko j_{h}}\bP&=\g_{i_{h+1}}\Big( \g_{j_{h+1}}\bw\cdot \g_{ i_1j_1\dots i_{h}\mko j_{h}}\bP\Big)\\
    &\quad -\g_{j_{h+1}}\bw\cdot \g_{ i_{h+1}\mko i_1j_1\dots i_{h}\mko j_{h}}\bP.
    \end{aligned}
\end{align}
Since $\g$ is compatible with the metric $g$, we have 
\begin{align}
\text{div}_g\Big(\g \bw\cdot \lap_g^{h}\mko \bP \Big)&=(-1)^{h}g^{i_1j_1}\cdots g^{i_{h+1}\mko j_{h+1}}\g_{i_{h+1}}\Big( \g_{j_{h+1}}\bw\cdot \g_{ i_1j_1\dots i_{h}\mko j_{h}}\bP\Big),\\
    \Big\lan \g\bw,\g \Big(\lap_g^{ h}\mkt\bP\Big)\Big\ran_g&=(-1)^{h}g^{i_1j_1}\cdots g^{i_{h+1}\mko j_{h+1}}\g_{j_{h+1}}\bw\cdot \g_{i_{h+1}\mko i_1j_1\dots i_{h}j_{h}}\bP.\label{gwdglapbp}
\end{align}
Combining~\eqref{giwjbPint} and~\eqref{intindwbP}--\eqref{gwdglapbp}, we obtain that, up to lower-order terms of the form~\eqref{contgwgp}--\eqref{condskiQ}, the principal term $2\mkt \Big\lan \g^{(h+1)}\mko\bw,\g^{(h+1)}\mko\bP \Big\ran_g$ can be expressed as in the right-hand side of~\eqref{ptvardivgw}. 

It remains to treat the lower-order terms of the form~\eqref{contgwgp}--\eqref{condskiQ}. The same iterated application of the Leibniz rule as in~\eqref{giwjbPint} shows that each of such terms can be expressed as a linear combination of terms of the form $\text{div}_g\big(\bar P_{\ga'}(\bw,\bP)\big)$ and $\big\lan \g\bw, \vec Q_\ga(\bP)\big\ran_g$, where $\vec Q_\ga(\bP)$ and $\bar P_{\ga'}(\bw,\bP)$ are respectively described by~\eqref{pconvecQga}--\eqref{condskivQ} and~\eqref{condvecPgaw}--\eqref{condsliP}. 
In particular, the bounds $\ell_0\le n$ in~\eqref{condskivQ} and $\ell_2'\le n-1$ in~\eqref{condsliP} follow from ~\eqref{k1+ki<n} and the third line of~\eqref{condskiQ}. Therefore, the expression~\eqref{ptvardivgw} is proved.

Finally, for a smooth immersion $\bP$ and $\bw\in C_c^\nf (B^n,\R^m)$, integrating~\eqref{ptvargwgp} and~\eqref{ptvardivgw} gives
\[
 \frac d{dt}E_{\bc}(\bP_t)\Big|_{t=0}
 =2\int_{B^n}\bigg\lan d\bw,\,
 d\Big(\lap_g^{h}\bP\Big)+\sum_{\ga\in I}a_\ga\vec Q_\ga(\bP)
 \bigg\ran_g\,\dvol_g.
\]
We set
\begin{equation}\label{defVmase}
 \vec V\coloneq d\Big(\lap_g^{h}\bP\Big)
       +\sum_{\ga\in I}a_\ga\vec Q_\ga(\bP).
\end{equation}
 Integration by parts then yields
\begin{equation}\label{varEfor}
 \frac d{dt}E_{\bc}(\bP_t)\Big|_{t=0}
       =-2\mko\big\lan d*_g\vec V,\bw\big\ran.
\end{equation}
 For $\bP\in\mathcal I_{h-1,2}(B^n,\R^m)$ satisfying~\eqref{weaimmconbP} and~\eqref{defep_0}, combining~\eqref{nabbPspa} and~\eqref{estQgammap} gives
\begin{equation}\label{estbV}
    \|\vec V\|_{\msc N_{n,0}(B^n)}\le C(n,\La,\bc)\mko \vae_{\bP}.
\end{equation}
Thus, by the approximation theorem~\cite[Thm.~4.6]{LaMaRi26}, the variation formula~\eqref{varEfor} remains valid for $\bP\in\mathcal I_{h-1,2}(B^n,\R^m)$, where $\langle\cdot,\cdot\rangle$ is the canonical distributional pairing with the test field $\bw\in C_c^\nf(B^n,\R^m)$.
This completes the proof.
\end{proof}

\subsection{The Noether currents associated with dilations and rotations}\label{sec:conlaws}
\
\vskip5pt
Let $\bP\in \mca I_{h-1,2}(B^n,\R^m)$ satisfying~\eqref{weaimmconbP} and~\eqref{defep_0} be a weak critical point of $E_{\bc}$. We define $\msc N_{n,0}$ and $\vec V\in \msc N_{n,0}(B^n,\R^m\ot \R^n)$ as in~\eqref{defNn0} and~\eqref{defVmase} respectively. Lemma~\ref{lm-critic-W^{1,1}} then implies that 
\[
    d*_g\vec V=0.
\]
Moreover, combining~\eqref{eq:fTproinene},~\eqref{defep_0},~\eqref{estbV}, and the Sobolev--Lorentz embedding yields that
\begin{equation}\label{soblorest*V}
    \|*_g\vec V\|_{W^{-h,(2,\nf)}(B^n)}\le C(n)\|*_g\vec V\|_{\msc N_{n,0}(B^n)}\le C(n,\La) \|\vec V\|_{\msc N_{n,0}(B^n)}\le C(n,\La,\bc)\mko \vae_{\bP}.
\end{equation}

In the remainder of this paper, we retain the notation of spaces $\msc E_n$, $\msc N_{n,1}$, and $\msc N_{n,2}$ from~\eqref{defcritspacesIV}, and use the subscript loc to denote the corresponding local Sobolev--Lorentz spaces. Applying Proposition~\ref{prop:Hod-coclosed} to $*_g\,\vec V$ with $k=h-1$, $p=2$, and using~\eqref{prolorgen}, we obtain $\bL\in \msc N_{n,1}\big(B^n,\R^m\ot \bwe^2\R^n\big)$ such that
\begin{subnumcases}{}
 d*_g\bL=*_g\,\vec V,\label{d*gL=*V}\\[0.3ex]
 d\bL=0.\label{dL0=0}
\end{subnumcases}
In addition, we have the estimate
\begin{align}\label{est-L_0}\begin{aligned}
     \|\bL\|_{\msc N_{n,1}(B^n)}
    &\le C(n,\La)\|*_g \bL\|_{\msc N_{n,1}(B^n)}\\
     &\le C(n,\La)\|*_g\vec V\|_{W^{-h,(2,\nf)}(B^n)}\\
     &\le C(n,\La,\bc)\mko\varepsilon_{\bP}.
 \end{aligned}
\end{align}

Since the $n$-form $\Big(\big| \g^{(h-1)}\mko\bII\big|_g^2+\sum_{\la\in \mca L} c_\la P_\la(\bP)\Big) \, \dvol_g$ is pointwise invariant under dilations and rotations, Noether's theorem yields the corresponding conservation laws, which we now establish.
\begin{Lm}
\label{lm-dilations}
Let $\bP$ satisfying~\eqref{weaimmconbP} be a weak critical point of $E_{\bc}$, and define $\bL$ as in~\eqref{d*gL=*V}--\eqref{est-L_0}. Then there exist $\vt_{\dil}\in\msc E_n(B^n,\R^n)$ and $\bvt_{\rot}\in\msc E_n\big(B^n,\bwe^2\R^m\ot\bwe^1\R^n\big)$ such that
\begin{subnumcases}{}
d*_g\big( \bL \,\dres  d\bP+\vt_{\dil}\mko\big)=0,\label{cons-law-dilation}\\[1ex]
d*_g \Big(\bL \wres d\bP +\mko d\bP\we  \lap_g^{h}\mko \bP+\bvt_{\rot} \Big)=0.\label{cons-law-rotation}
\end{subnumcases}
In addition, under~\eqref{defep_0}, we have
\begin{align}\label{ptbdvts}
 \|\vt_{\dil}\|_{\msc E_n(B^n)}
       +\|\bvt_{\rot}\|_{\msc E_n(B^n)}
       \le C(n,\La,\bc)\mko\varepsilon_{\bP}^{\,2}.
\end{align}
\end{Lm}
\begin{proof}
To prove~\eqref{cons-law-dilation}, we first consider the variation $\bP_t=(1+t)\bP$ with 
$$
    \bw=\frac d{dt}\mko \bP_t\Big|_{t=0}=\bP.
$$
Denote $\g_t$, $\bII_t$, and $g_t$ as in the proof of Lemma~\ref{lm-critic-W^{1,1}}. 
Since the $n$-form defining $E_{\bc}$ in~\eqref{eq:defE} is invariant under dilations, using~\eqref{ptvargwgp} and~\eqref{ptvardivgw}, we obtain that
\begin{align}\label{dilconcomp1}
\begin{aligned}
    0&=   \frac d{dt}\bigg(\Big(\big| \g_t^{(h-1)}\mko\bII_t\big|_{g_t}^2+\sum_{\la\in \mca L} c_\la P_\la(\bP_t)\Big)\dvol_{g_t}\bigg)\bigg|_{t=0}\\
    &=-2\mko d*_g\bigg(d \bw\cdot \lap_g^{h}\mko \bP+\sum_{\ga'\in I'} \bar a_{\ga'}\mko \bar P_{\ga'}(\bw,\bP)\bigg)+2\mkt d \bw\dwe *_g \vec V
    \end{aligned}
\end{align}
By~\eqref{d*gL=*V} and~\eqref{*_gcomres}, we have
\begin{align}\label{d*(PV)}
    d \bw\dwe *_g \vec V=d\bw\dwe d*_g\bL=-d\big( \mko(*_g\,\bL)\dwe d\bw\big)=d*_g(\bL \,\dres  d\bw).
\end{align}
Set 
\begin{equation}\label{defvtdil}
    \vt_{\dil}\coloneq -\bigg(d \bP\cdot \lap_g^{h}\mko \bP+\sum_{\ga'\in I'} \bar a_{\ga'}\mko \bar P_{\ga'}(\bP,\bP)\bigg).
\end{equation}
The conservation law~\eqref{cons-law-dilation} then follows from~\eqref{dilconcomp1}--\eqref{defvtdil}. Combining~\eqref{prolorgen},~\eqref{estpiTK}, and~\eqref{estPgamma'wp}, we obtain the estimate
\begin{align}\label{estvtdil}
 \|\vt_{\dil}\|_{\msc E_n(B^n)}
       \le C(n,\La,\bc)\mko\varepsilon_{\bP}^{\,2}.
\end{align}
Next, we prove~\eqref{cons-law-rotation}.  Fix $\vec a\in \bwe^2\R^m$, and we define $\bP_t$ by
\begin{align}\label{varrot}
  \begin{dcases} 
  \frac d{dt}\mko \bP_t=\vec a\,\res \bP_t, 
  \qquad t\in \R,\\[0.3ex]
  \bP_0=\bP.
  \end{dcases}
\end{align} 
For this variation, since $ \frac d{dt}\mko \bP_t\cdot \bP_t=\vec a \cdot (\bP_t\we \bP_t)=0$, we have $|\bP_t|^2=|\bP|^2$ for all $t\in \R$. In addition, $\bP_t$ depends linearly on $\bP$.
Hence there exists $Q_t\in \text{SO}(m)$, depending only on $\vec a$ and $t$, such that $\bP_t=Q_t\circ \bP$ on $B^n$. Consequently, the $n$-form $\Big(\big| \g_t^{(h-1)}\mko\bII_t\big|_{g_t}^2+\sum_{\la\in \mca L} c_\la P_\la(\bP_t)\Big)\dvol_{g_t}$ is independent of $t$, and the equations~\eqref{dilconcomp1}--\eqref{d*(PV)} remain valid if we take $\bw=\vec a\,\res \bP$. By~\eqref{eq:defres}, we have 
\begin{align}
    d (\vec a\,\res \bP)\cdot \lap_g^{h}\mko \bP&=\vec a\cdot \big(d\bP\we  \lap_g^{h}\mko \bP\big),\label{dwdlaprew}\\
     \bL\,\dres d(\vec a\,\res \bP)&=-\vec a\cdot \big(\bL \wres d\bP\big).\label{Lddwrew}
\end{align}
Similar to~\eqref{dwdlaprew}, we can write 
\begin{equation}\label{ParesPhirew}
    \sum_{\ga'\in I'} \bar a_{\ga'}\mko \bar P_{\ga'}(\vec a\,\res \bP,\bP)=\vec a\cdot \bvt_{\rot}. 
\end{equation}
The estimate~\eqref{estPgamma'wp} gives
\[
 \bigg\|\sum_{\ga'\in I'}\bar a_{\ga'}
       \bar P_{\ga'}(\vec a\,\res\bP,\bP)\bigg\|_{\msc E_n(B^n)}
       \le C(n,\La,\bc)\mko|\vec a|\mko\varepsilon_{\bP}^{\,2}.
\]
Hence~\eqref{ParesPhirew} defines $\bvt_{\rot}\in\msc E_n$ and proves the remaining estimate in~\eqref{ptbdvts}.
By~\eqref{d*(PV)} and~\eqref{Lddwrew}, we have 
\begin{equation}\label{dwd*Vrot}
    d(\vec a\,\res \bP)\dwe *_g \vec V=-\vec a \cdot d*_g \big(\bL \wres d\bP\big).
\end{equation}
Combining~\eqref{dilconcomp1},~\eqref{dwdlaprew},~\eqref{ParesPhirew} with~\eqref{dwd*Vrot}, we obtain that for $\bP_t$ defined in~\eqref{varrot}, there holds
\begin{align*}
    0&= \frac d{dt}\bigg(\Big(\big| \g_t^{(h-1)}\mko\bII_t\big|_{g_t}^2+\sum_{\la\in \mca L} c_\la P_\la(\bP_t)\Big)\dvol_{g_t}\bigg)\bigg|_{t=0}\\
    &=-2\mkt \vec a \cdot  d*_g\Big(\bL \wres d\bP +\mko d\bP\we  \lap_g^{h}\mko \bP+\bvt_{\rot} \Big).
\end{align*}
Since $\vec a\in \bwe^2 \R^m$ is arbitrary, the conservation law~\eqref{cons-law-rotation} follows.
\end{proof}
Define $\bL$ as in~\eqref{d*gL=*V}--\eqref{est-L_0}. Then by~\eqref{dL0=0}, we have
\begin{align}\label{dLwedbp}
 \begin{dcases}
 d(\bL\dwe d\bP)=d\bL\dwe d\bP=0,\\[0.5ex]
 d\big(\bL\ovs{\ovwe}{\we}d\bP\big)
       =d\bL\ovs{\ovwe}{\we}d\bP=0.
 \end{dcases}
\end{align}
Invoking~\eqref{prolorgen},~\eqref{nabbPspa}, and~\eqref{est-L_0}, we estimate
\begin{align}\label{estLwedplapH}\begin{aligned}
&\|\bL\mkt\dres d\bP\|_{\msc N_{n,1}(B^n)}
 +\|\bL\wres d\bP\|_{\msc N_{n,1}(B^n)}  +\|\bL\dwe d\bP\|_{\msc N_{n,1}(B^n)}\\
&\quad +\|\bL\ovs{\ovwe}{\we}d\bP\|_{\msc N_{n,1}(B^n)}
 +\|d\bP\we\lap_g^{h}\mko \bP\|_{\msc N_{n,1}(B^n)} \le C(n,\La,\bc)\mko\varepsilon_{\bP}.
 \end{aligned}
\end{align}
By~\eqref{ptbdvts} and $\msc E_n\hookrightarrow\msc N_{n,1}$, we also have
\begin{align}\label{estvts}
 \|\vt_{\dil}\|_{\msc N_{n,1}(B^n)}
       +\|\bvt_{\rot}\|_{\msc N_{n,1}(B^n)}
       \le C(n,\La,\bc)\mko\varepsilon_{\bP}^{\,2}.
\end{align}
Combining~\eqref{cons-law-dilation},~\eqref{cons-law-rotation}, and~\eqref{dLwedbp} with Corollary~\ref{co-Hod-decw-1p} ($k=h-1$, $p=2$ when $n\ge6$), we obtain $S\in\msc N_{n,2}\big(B^n,\bwe^2\R^n\big)$ and $\bR\in\msc N_{n,2}\big(B^n,\bwe^2\R^m\ot\bwe^2\R^n\big)$
such that
\begin{align*}
 \begin{dcases}
 \delta S=\bL\mkt\dres d\bP+\vt_{\dil},\quad
       &dS=-2\bL\dwe d\bP,\\[0.5ex]
 \delta\bR=\bL\wres d\bP+d\bP\we\lap_g^{h}\mko \bP+\bvt_{\rot},\quad
       &d\bR=-2\bL\ovs{\ovwe}{\we}d\bP.
 \end{dcases}
\end{align*}
Moreover, from the estimates~\eqref{est:Hodfirst},~\eqref{estLwedplapH}, and~\eqref{estvts}, we obtain that
\[
 \|S\|_{\msc N_{n,2}(B^n)}
       +\|\bR\|_{\msc N_{n,2}(B^n)}
       \le C(n,\La,\bc)\mko\varepsilon_{\bP}.
\]

We summarize our results of Sections~\ref{sec:EullagestV}--\ref{sec:conlaws} in the following proposition.
\begin{Prop}\label{sysestLSR}
Assume $\bP\in\mathcal I_{h-1,2}(B^n,\R^m)$ is a weak critical point of $E_{\bc}$ satisfying~\eqref{weaimmconbP} and~\eqref{defep_0}. Then there exist
$\bL\in\msc N_{n,1}$, $S,\bR\in\msc N_{n,2}$, and
$\vt_{\dil},\bvt_{\rot}\in\msc E_n$, with their respective target spaces as above, such that the following system holds in $\mathcal D'(B^n)$:
\begin{equation}\label{sys-LRS}
 \begin{dcases}
 d*_g\bL=*_g\bigg(d\Big(\lap_g^h\bP\Big)
       +\sum_{\ga\in I}a_\ga\vec Q_\ga(\bP)\bigg),\\
 \delta S=\bL\mkt\dres d\bP+\vt_{\dil},\\
 \delta\bR=\bL\wres d\bP+d\bP\we\lap_g^h\bP+\bvt_{\rot},\\[0.5ex]
 d\bL=0,\\[0.5ex]
 dS=-2\bL\dwe d\bP,\\[0.5ex]
 d\bR=-2\bL\ovs{\ovwe}{\we}d\bP.
 \end{dcases}
\end{equation}
Moreover, the following estimates hold:
\begin{align}\label{estLSR}
 \|\bL\|_{\msc N_{n,1}(B^n)}
 +\|S\|_{\msc N_{n,2}(B^n)}
 +\|\bR\|_{\msc N_{n,2}(B^n)}
 &\le C(n,\La,\bc)\mko\varepsilon_{\bP},\\[1ex]
 \|\vt_{\dil}\|_{\msc E_n(B^n)}
 +\|\bvt_{\rot}\|_{\msc E_n(B^n)}
 +\|\pi_T\lap_g^h\bP\|_{\msc E_n(B^n)}
 &\le C(n,\La,\bc)\mko\varepsilon_{\bP}^{\,2}.
 \label{ptbdvts2}
\end{align}
\end{Prop}
\subsection{Morrey decay estimates}\label{sec:finSoblormoest}
\
\vskip5pt
 Throughout the remainder of the paper, we use the scale-homogeneous Sobolev--Lorentz norms defined in~\eqref{eq:uniSobnm}--\eqref{negSobnm}. For every ball $B\subset B^n$, we write
$\varepsilon_{\bP}(B)\coloneq\|D^2\bP\|_{W^{h-1,2}(B)}$, and retain
$\varepsilon_{\bP}=\varepsilon_{\bP}(B^n)$. Combining Corollary~\ref{co:sysvecu} and Proposition~\ref{sysestLSR} with Lemma~\ref{lm:ellcacciolp} yields the following Sobolev--Lorentz--Morrey type estimate.

\begin{Th}\label{th:lapHL432}
Suppose $\bP\in\mathcal I_{h-1,2}(B^n,\R^m)$ is a weak critical point of
$E_{\bc}$ satisfying~\eqref{weaimmconbP}. For any $\al\in(0,n-1)$,
there exists $\vae_0=\vae_0(n,\La,\bc,\al)>0$ such that the following holds.
Assume that
\begin{align}\label{smaIIggvae}
 \varepsilon_{\bP}\le\vae_0.
\end{align}
Then $\vec K=\lap_g^h\bP\in\msc E_{n,\loc}(B^n,\R^m)$ and, for all
$r\in(0,\frac12]$,
\begin{align}\label{est:lapHL432}
 \|\vec K\|_{\msc E_n(B_r)}
 \le C(n,\La,\bc,\al)
       \big(\varepsilon_{\bP}+r^\al\big)\varepsilon_{\bP}.
\end{align}
\end{Th}
\begin{proof}
Assume~\eqref{smaIIggvae} holds with $\vae_0\in(0,1)$ to be determined below.
As in~\eqref{defuetabul}, set
\begin{align}\label{defuetabul2}
 \vec u\coloneq\vet\mkt\ovs{\sbul}{\res}_g\bR+\vet\mkt\resg S
       -(3n-6)\bX\wres d\bP,
 \qquad \bX=d\big(\lap_g^{h-1}\bP\big).
\end{align}
By Corollary~\ref{co:sysvecu} and~\eqref{sys-LRS}, we have the following identities in $\mca D'(B^n)$:
\begin{align}
 d*_gd\vec u&=d*_g\vec{\mca R}_3,
       \label{eq:d*du=d*R3+2}\\[0.3ex]
 \pi_N\delta(\vec u\mkt\res d\bP)
       &=(2n-6)\lap_g^h\bP+\vec{\mca R}_4.
       \label{eq:ndotd*uresdp2}
\end{align}
Moreover, combining the estimates~\eqref{estvecu},~\eqref{remtotptbd},~\eqref{estLSR}, and~\eqref{ptbdvts2} gives
\begin{align}\label{estRL432}
 \|\vec{\mca R}_3\|_{\msc E_n(B^n)}
       +\|\vec{\mca R}_4\|_{\msc E_n(B^n)}
       &\le C(n,\La,\bc)\mko\varepsilon_{\bP}^{\,2},\\
 \|\vec u\|_{\msc N_{n,2}(B^n)}
       &\le C(n,\La,\bc)\mko\varepsilon_{\bP}.
       \label{u2nfbd}
\end{align}
Set $a^{ij}\coloneq(\det g)^{1/2}g^{ij}$. In coordinates, the equation~\eqref{eq:d*du=d*R3+2} reads
\begin{equation}\label{eq:d*ducoor}
 \p_i(a^{ij}\p_j\vec u)
       =\p_i\big(a^{ij}(\vec{\mca R}_3)_j\big).
\end{equation}
The product inequality~\eqref{prolorgen} implies
\[
 \big\|\p_i\big(a^{ij}(\vec{\mca R}_3)_j\big)\big\|
       _{W^{1-h,\lf(\frac{2h}{h+1},2\rg)}(B^n)}
 \le C(n,\La)\|\vec{\mca R}_3\|_{\msc E_n(B^n)}.
\]
We apply~\eqref{eq:ellcaccioneggrad} in Lemma~\ref{lm:ellcacciolp} to~\eqref{eq:d*ducoor}, with
\[
 k=h-1,\qquad \ell=h-2,\qquad q_0=2,\qquad s_0=\nf,
 \qquad p'=\frac{2h}{h+1},\qquad s=2.
\]
These exponents are admissible, since
\[
 \frac n{n-k}=\frac{2h}{h+1}<2,
 \qquad p=\frac{2h}{h-1}<\frac n\ell,
 \qquad \frac n{p'}+\ell=n-1.
\]
Also, we have $\|Da^{ij}\|_{W^{h-2,n/(h-1)}(B^n)}\le C(n,\La)\mko\vae_{\bP}$.
Thus, after decreasing $\vae_0$ if necessary, for any $\al\in (0,n-1)$ and $r\in(0,\frac 12]$, the lemma gives
\begin{align}\label{guL432est}\begin{aligned}
    \|D\vec u\|_{\msc E_n(B_r)}
    &\le C(n,\La,\al)\Big(
       \|\vec{\mca R}_3\|_{\msc E_n(B^n)}
       +r^\al\|\vec u\|_{\msc N_{n,2}(B^n)}\Big)\\
 &\le C(n,\La,\bc,\al)
       \big(\varepsilon_{\bP}+r^\al\big)\mko\varepsilon_{\bP}.
       \end{aligned}
\end{align}

Expanding the codifferential and applying the inequality~\eqref{ineproENn}, we obtain
\begin{align}\label{estndd*ures}
 \big\|\pi_N\delta(\vec u\mkt\res d\bP)\big\|_{\msc E_n(B_r)}
 \le C(n,\La)\Big(
 &\|D\vec u\|_{\msc E_n(B_r)}+\varepsilon_{\bP}(B_r)
       \|\vec u\|_{\msc N_{n,2}(B_r)}\Big).
\end{align}
By~\eqref{equivSobnorm}, we have $\varepsilon_{\bP}(B_r)\le C(n)\mko\varepsilon_{\bP}$. Together with~\eqref{u2nfbd} and~\eqref{guL432est}, this implies
\begin{align}\label{estndeldpep}
    \big\|\pi_N\delta(\vec u\mkt\res d\bP)\big\|_{\msc E_n(B_r)}
 \le C(n,\La,\bc,\al)\big(\vae_{\bP}+r^\al\big)\mko\vae_{\bP}.
\end{align}
Combining~\eqref{eq:ndotd*uresdp2},~\eqref{estRL432}, and~\eqref{estndeldpep} proves~\eqref{est:lapHL432}.
The claim $\vec K\in\msc E_{n,\loc}(B^n,\R^m)$ follows by restricting to sufficiently small balls and rescaling.
\end{proof}

In the following, we will prove $\bH\in W^{h,\lf(\frac{2h}{h+1},2\rg)}_{\loc}(B^n,\R^m)$ from the fact that $\vec K\in\msc E_{n,\loc}(B^n,\R^m)$. Once this is established, we will carry out an iteration for the following critical norms
\begin{align}\label{defN12r}
 \mca N_1(r)&\coloneq
       \|\bH\|_{W^{h,\lf(\frac{2h}{h+1},2\rg)}(B_r)},\qquad \mca N_2(r)\coloneq \vae_{\bP}(B_r).
\end{align}

\begin{Co}\label{morestHgn}
Suppose $\bP\in\mathcal I_{h-1,2}(B^n,\R^m)$ is a weak critical point of
$E_{\bc}$ satisfying~\eqref{weaimmconbP}, and that the coordinates are
$g_{\bP}$-harmonic. Then we have $\bP\in W^{h+2,\lf(\frac{2h}{h+1},2\rg)}_{\loc}(B^n)$. Moreover, for any $\beta\in(0,1)$, there exists
$\vae_1=\vae_1(n,\La,\bc,\beta)>0$ such that the following holds. Assume
\begin{align}\label{smaIIggvae1}
\mca N_2(1)\le\vae_1.
\end{align}
Then for all $r\in(0,\frac 14]$, we have
\begin{align}\label{HgHmornor}
 \mca N_1(r)+\mca N_2(r)
       \le C(n,\La,\bc,\beta)\mko r^\beta\mca N_2(1).
\end{align}
\end{Co}
\begin{proof}
In view of~\eqref{equivSobnorm}, we choose $\vae_1$ sufficiently small such that the smallness assumptions needed below hold for any ball contained in $B^n$.

Thereom~\ref{th:lapHL432} implies that $\vec K\in\msc E_{n}(B_{1/2},\R^m)$. Using the rescaling in~\eqref{phidil}--\eqref{Dphinormscal}, we normalize $B_{1/2}$ to $B^n$, and thus assume without loss of generality that $\vec K\in \msc E_{n}(B^n,\R^m)$. Since $\lap_g\bP=-n\bH$, we have
\begin{equation}\label{laphphi=K}
    \lap_g^{h-1}\bH=-\frac1n\vec K.
\end{equation}
By~\eqref{eq:ffproinepo} and~\eqref{smaIIggvae1}, we obtain
\begin{equation}\label{estHfrombP}
    \|\bH\|_{W^{h-1,2}(B^n)}\le C(n,\La)\|D^2\bP\|_{W^{h-1,2}(B^n)}\le C(n,\La)\mko \mca N_2(1).
\end{equation}
To further estimate $\bH$ through~\eqref{laphphi=K}, we argue as in the proof
of Lemma~\ref{lm:ellcacciolp}. Set
$a^{ij}\coloneq(\det g)^{1/2}g^{ij}$ and
$F_{h-1}\coloneq-\vec K/n$. For $\nu=h-2,\ldots,0$, we successively apply Lemma~\ref{lm:exisol} and interpolation to construct solutions of
\begin{equation}\label{eq:iteratedparticularV}
\begin{dcases}
    \p_i(a^{ij}\p_jF_\nu)=-(\det g)^{1/2}F_{\nu+1}
    \qquad\text{in }\mca D'(B^n,\R^m),\\
    F_\nu\in W^{h-2\nu,\lf(\frac{2h}{h+1},2\rg)}(B^n,\R^m).
\end{dcases}
\end{equation}
Indeed, by Lemma~\ref{lm:proinene}, multiplication by
$(\det g)^{1/2}$ is bounded on $W^{h-2\nu-2,\lf(\frac{2h}{h+1},2\rg)}$ for $0\le \nu\le h-2$.
When $2\nu\ge h$, we apply~\eqref{estST}
with $k=2\nu+1-h$ and $p=2h/(h-1)$; otherwise, we apply~\eqref{estuf} with
$k=h-2\nu-1$ and $p=2h/(h+1)$. The resulting estimates give, by induction,
\begin{equation}\label{estHFK}
    \|F_\nu\|_{W^{h-2\nu,\lf(\frac{2h}{h+1},2\rg)}(B^n)}
    \le C(n,\La)\mko\|\vec K\|_{\msc E_n(B^n)},
    \qquad 0\le \nu\le h-1.
\end{equation}
Set $\bH_0\coloneq F_0$ and $\bH_1\coloneq\bH-\bH_0$. Then we have $\lap_g^{h-1}\bH_0=-\vec K/n$ and $\lap_g^{h-1}\bH_1=0$. By~\eqref{estHFK},~\eqref{estHfrombP}, and Sobolev--Lorentz embedding, we obtain
\begin{align}\label{estbH1}
\begin{aligned}
    \|\bH_1\|_{W^{h-1,2}(B^n)}&\le C(n,\La) \big(\|\bH_0\|_{W^{h-1,2}(B^n)}+\|\bH\|_{W^{h-1,2}(B^n)} \big)\\
    &\le C(n,\La) \big(\|\vec K\|_{\msc E_n(B^n)}+\mca N_2(1) \big).
    \end{aligned}
\end{align}
We now refine the interior estimate for the homogeneous solution. Set $Z_\nu=\lap_g^\nu\bH_1$ for $0\le \nu\le h-1$. 
Then we have $Z_{h-1}=0$. By Lemma~\ref{lm:proinene} and induction, we estimate
\begin{equation}\label{rouesthomsol}
 \sum_{\nu=0}^{h-2}\|Z_\nu\|_{W^{h-1-2\nu,2}(B^n)}
       \le C(n,\La)\|\bH_1\|_{W^{h-1,2}(B^n)}.
\end{equation}
Fix $2<q<2h/(h-1)$. For $\nu=h-2,\ldots,0$, we successively estimate $Z_\nu$ from the equation
\begin{equation}\label{eq:lapZj=Zj+1}
     \p_i(a^{ij}\p_jZ_\nu)=-(\det g)^{1/2}Z_{\nu+1}.
\end{equation}  
By Lemma~\ref{lm:proinene}, multiplication by $(\det g)^{1/2}$ is bounded on $W^{h-2\nu-2,q}$ for $0\le \nu\le h-2$. When $ \frac{h-1}2<\nu\le h-2$, since $h-1-2\nu \ge 3-h$, the estimate~\eqref{rouesthomsol} implies $Z_\nu\in W^{3-h,2}$. 
Then we apply~\eqref{eq:ellcaccioneggrad} with \[
    k=h-2,\qquad \ell=2\nu+1-h,\qquad p'=s=q,\qquad q_0=s_0=2.
\] 
In this case, since $n/(2\nu+1-h)>2>q'$ and $2>2h/(h+2)$, the exponents are admissible. By induction, it follows that $Z_\nu\in W^{h-2\nu,q}_{\loc}(B^n)$ for $\frac{h-1}2<\nu\le h-2$.
When $2\nu+1\le h$, the estimate~\eqref{rouesthomsol} gives $Z_\nu\in L^2\hookrightarrow W^{2-h,2}$. We apply~\eqref{eq:ellcacciop1} with
\[
    k=h-1,\qquad \ell=h-2\nu-1,\qquad p=s=q,\qquad q_0=s_0=2.
\]
The exponents are admissible since $n/(h-2\nu-1)\ge2h/(h-1)>q$. The resulting estimates on nested balls together with~\eqref{rouesthomsol} give, by induction, 
\begin{equation}\label{estZnuWh-2}
     \|Z_\nu\|_{W^{h-2\nu,q}(B_{1/2})}\le C(n,q,\La) \|\bH_1\|_{W^{h-1,2}(B^n)},\qquad 0\le \nu\le h-2.
  \end{equation}
  In particular, since $Z_0=\bH_1$ and $q>2$, Sobolev embedding then implies $\bH_1\in W^{h,q}(B_{1/2})\hookrightarrow C^0(\ov{B_{1/2}})$. By Sobolev embeddings and
H\"older's inequality, for all $0\le j \le h$ and $r\in(0,\frac 12]$, we have
\[
    \|D^{j}\bH_1\|_{L^{\frac n{1+j},2}(B_r)}\le  C(n)\mkt  r\mko\|D^{j}\bH_1\|_{L^{\frac nj}(B_r)}\le C(n,q)\mkt  r\mko \|\bH_1\|_{W^{h,q}(B_{1/2})}.
\]
Here we set $\frac nj\coloneq \nf$ when $j=0$. By~\eqref{equivSobnorm} and~\eqref{estZnuWh-2}, it follows that
\begin{equation}\label{eq:H1Wh2hh+1est}
    \|\bH_1\|_{W^{h,\lf(\frac{2h}{h+1},2\rg)}(B_r)}
    \le C(n,q)\mkt r\mko\|\bH_1\|_{W^{h,q}(B_{1/2})}
    \le C(n,\La,q)\mkt r\mko\|\bH_1\|_{W^{h-1,2}(B^n)}.
\end{equation}
Combining~\eqref{eq:H1Wh2hh+1est},~\eqref{estHFK} (with $\nu=0$), and~\eqref{estbH1} with the
embedding $W^{h,\lf(\frac{2h}{h+1},2\rg)}\hookrightarrow W^{h-1,2}$, we obtain that $\bH\in W^{h,\lf(\frac{2h}{h+1},2\rg)}_{\loc}(B^n,\R^m)$ and for all $r\in (0,\frac 12]$,
\begin{align}\label{HgHlapH}
 \mca N_1(r) \le C(n,\La)\Big(
       \|\vec K\|_{\msc E_n(B^n)}+r\mkt\mca N_2(1)\Big).
\end{align}
For the estimate of $\mca N_2$, the harmonic coordinate condition is
\begin{equation}\label{harcoreq}
 \forall\,1\le j\le n,\qquad \p_i\mko a^{ij}=0.
\end{equation}
It follows that
\begin{align}\label{ellgraf=H}
 n\bH=-\lap_g\bP=g^{ij}\p_i\p_j\bP.
\end{align}
We differentiate~\eqref{ellgraf=H} once and set $\vec v_k\coloneq \p_k\bP$. Rewriting the resulting equation in divergence form, we obtain
\begin{numcases}{}
 \p_i\big(g^{ij}\p_j\vec v_k\big)=n\mkt\p_k\bH+\vec B_k,\label{eq:DPhidivV}\\[0.3ex]
 \vec B_k\coloneq
       (\p_i g^{ij})\p_j\vec v_k
       -(\p_k g^{ij})\p_i\p_j\bP.
       \label{eq:BkV}
    \end{numcases}
By Lemma~\ref{lm:proinene} and~\eqref{smaIIggvae1}, we have
\begin{align}\label{eq:BkcriticalV}\begin{aligned}
 \sum_k\|\vec B_k\|_{W^{h-1,(\frac {2h}{h+1},2)}(B^n)}
 &\le C(n,\La)\|Dg\|_{W^{h-1,2}(B^n)}
                  \|D^2\bP\|_{W^{h-1,2}(B^n)}\\
 &\le C(n,\La)\mko \vae_{\bP}^2\le  C(n,\La)\mko\vae_1\mko\mca N_2(1).
 \end{aligned}
\end{align}
Now we apply~\eqref{eq:ellcacciopg} in Lemma~\ref{lm:ellcacciolp} to each
$\vec v_k$, with 
\[
    k=\ell=h,\qquad  p=\frac {2h}{h+1},\qquad s=s_0=2, \qquad  q_0=\frac {2h}{h-1}. 
\]
Fix $\beta_0\in (\beta,1)$, depending only on $\beta$. Since $2h/(h+1)<2=n/h$ and $2h/(h-1)>2=n/(n-h)$, the exponents are admissible. 
Since $\bv_k\in L^\nf\hookrightarrow W^{1-h,(\frac{2h}{h-1},2)}$  and $\beta_0<h+1-h=1$, the resulting estimate together with the embedding $W^{h,\lf(\frac{2h}{h+1},2\rg)}\hookrightarrow W^{h-1,2}$ implies that $\bP\in W^{h+2,\lf(\frac{2h}{h+1},2\rg)}_{\loc}$ and for all $r\in (0,\frac 12]$,
\begin{align}\label{bdN2N1}
 \mca N_2(r)\le C(n)\|D^2\bP\|_{W^{h,\lf(\frac{2h}{h+1},2\rg)}(B_r)}\le C(n,\La,\beta)\Big(
       \mca N_1(1)+(\vae_1+r^{\beta_0})\mca N_2(1)\Big).
\end{align}
By Theorem~\ref{th:lapHL432} (with $\al=1$) and
\eqref{smaIIggvae1}, under the condition $\vae_1\le \vae_0<1$, we also obtain for all $r\in (0,\frac 12]$ that
\begin{align}\label{bdlapHN2}
 \|\vec K\|_{\msc E_n(B_r)}
       \le C(n,\La,\bc)(\vae_1+r)\mko\mca N_2(1).
\end{align}
Under the dilation~\eqref{phidil}, we have
\begin{equation}\label{dilgnH}
 \bH_{\bP_\rho}(x)=\rho\mko \bH(\rho\mko x),\qquad
 \lap_{g_{\bP_\rho}}^h\bP_\rho(x)=\rho^{n-1}\vec K(\rho x).
\end{equation}
Consequently, all the three norms $\mca N_1$, $\mca N_2$, and
$\|\vec K\|_{\msc E_n}$ are scale-invariant:
\begin{equation}\label{gggndilinv}
 \mca N_{i,\bP_\rho}(t)=\mca N_i(\rho t)\quad(i=1,2),\qquad
 \big\|\lap_{g_{\bP_\rho}}^h\bP_\rho\big\|_{\msc E_n(B_t)}
       =\|\vec K\|_{\msc E_n(B_{\rho t})}.
\end{equation}
By normalizing an arbitrary interior ball to $B^n$, we assume without loss of generality that all the three norms above are finite.
Combining~\eqref{HgHlapH},~\eqref{bdN2N1}, and~\eqref{bdlapHN2} with the rescaling, for all $0<2r\le\rho\le1$ we obtain
\begin{subnumcases}{}
 \|\vec K\|_{\msc E_n(B_r)}
       \le C(n,\La,\bc)(\vae_1+r/\rho)\mca N_2(\rho),
       \label{lapHN2}\\[0.5ex]
 \mca N_1(r)\le C(n,\La)\Big(
       \|\vec K\|_{\msc E_n(B_\rho)}
       +(r/\rho)\mca N_2(\rho)\Big),
       \label{N1lapHN1}\\[0.5ex]
 \mca N_2(r)\le C(n,\La,\beta)\Big(
       \mca N_1(\rho)
       +(\vae_1+(r/\rho)^{\beta_0})\mca N_2(\rho)\Big).
       \label{N2N1+N2}
\end{subnumcases}
Now we fix integers $s_0\ge2$ and $s>s_0$, depending only on $\beta$, such that
\begin{equation}\label{condss0}
    \beta_0<\frac{s_0-1}{s_0},\qquad
    \beta<\beta_0\mko\frac{s-s_0}{s}.
\end{equation}
By~\eqref{equivSobnorm}, for all $r\le 1$, we have
\begin{equation}\label{N2normcomprad}
    \mca N_2(r)\le C(n)\mca N_2(1).
\end{equation}
Let $0<r\le \frac 12$. For $s_0+1\le\nu\le s$, applying~\eqref{N1lapHN1} with inner radius $r^{\nu-1}$ and outer radius $r^{\nu-s_0}$, and then
\eqref{lapHN2} with outer radius one, we obtain
\begin{align}\begin{aligned}\label{eq:appliabN}
    \mca N_1(r^{\nu-1})
    &\le C(n,\La)\Big(
       \|\vec K\|_{\msc E_n(B_{r^{\nu-s_0}})}
       +r^{s_0-1}\mca N_2(r^{\nu-s_0})\Big)\\
    &\le C(n,\La,\bc)\Big(
       (\vae_1+r^{\nu-s_0})\mca N_2(1)
       +r^{s_0-1}\mca N_2(r^{\nu-s_0})\Big).
       \end{aligned}
\end{align}
In addition, by~\eqref{N2N1+N2}, we have
\begin{equation}\label{eq:N2citera}
    \mca N_2(r^\nu)
    \le C(n,\La,\beta)\Big(
       \mca N_1(r^{\nu-1})
       +(\vae_1+r^{\beta_0})\mca N_2(r^{\nu-1})\Big).
\end{equation}
We prove by induction that for all $s_0+1\le\nu\le s$ and $0<r\le \frac 12$,
\begin{equation}\label{eq:finiteindN12}
    \mca N_2(r^\nu)+\mca N_1(r^{\nu-1})
    \le C(n,\La,\bc,\beta)\big(\vae_1+r^{(\nu-s_0)\beta_0}\big)\mca N_2(1).
\end{equation}
 Since $s_0-1\ge 1>\beta_0$, combining~\eqref{eq:appliabN}--\eqref{eq:N2citera} with~\eqref{N2normcomprad} gives
\begin{align*}
    \mca N_1(r^{s_0})+\mca N_2(r^{s_0+1})
    &\le C(n,\La,\bc,\beta)\big(\vae_1+r+r^{s_0-1}+r^{\beta_0}\big) \mca N_2(1)\\
    &\le C(n,\La,\bc,\beta)\big(\vae_1+r^{\beta_0}\big)\mca N_2(1).
\end{align*}
This proves the case $\nu=s_0+1$. 

Now let $s_0+2\le\nu\le s$, and suppose~\eqref{eq:finiteindN12} holds for each index from $s_0+1$ to $\nu-1$. 
We first estimate the term $r^{s_0-1}\mca N_2(r^{\nu-s_0})$. If $\nu\le2s_0$, then $1\le\nu-s_0\le s_0$. By~\eqref{condss0} and~\eqref{N2normcomprad}, we have
\[
    r^{s_0-1}\mca N_2(r^{\nu-s_0})
    \le C(n)\mko r^{s_0-1}\mca N_2(1)
    \le C(n)\mko r^{s_0\beta_0}\mca N_2(1)\le C(n)\mko r^{(\nu-s_0)\beta_0}\mca N_2(1).
\]
If $\nu>2s_0$, then $s_0+1\le\nu-s_0<\nu$, and the induction hypothesis at index $\nu-s_0$ gives
\begin{align*}
    r^{s_0-1}\mca N_2(r^{\nu-s_0})
    &\le C(n,\La,\bc,\beta)\mko r^{s_0-1}
       \big(\vae_1+r^{(\nu-2s_0)\beta_0}\big)\mca N_2(1)\\
    &\le C(n,\La,\bc,\beta)\big(\vae_1+r^{(\nu-s_0)\beta_0}\big)\mca N_2(1).
\end{align*}
Thus in both cases, we have
\begin{equation}\label{estrv-s0}
    r^{s_0-1}\mca N_2(r^{\nu-s_0})
    \le C(n,\La,\bc,\beta)\big(\vae_1+r^{(\nu-s_0)\beta_0}\big)\mca N_2(1).
\end{equation}
Combining~\eqref{eq:appliabN},~\eqref{eq:N2citera} and~\eqref{estrv-s0} with the induction hypothesis at index $\nu-1$, we obtain~\eqref{eq:finiteindN12} at index $\nu$. This completes the induction.
In particular, there exists $C_0=C_0(n,\La,\bc,\beta)>0$ such that
\begin{align}\label{HgHleer}
    \mca N_2(r^s)+\mca N_1(r^{s-1})
    \le C_0\big(\vae_1+r^{(s-s_0)\beta_0}\big)\mca N_2(1).
\end{align}
By~\eqref{condss0}, we can choose $r_0=r_0(n,\La,\bc,\beta)\in(0,\frac12]$ such that
\begin{align}\label{C0r0albeta}
    C_0r_0^{(s-s_0)\beta_0}\le\frac12r_0^{s\beta}.
\end{align}
Set $r_1\coloneq r_0^s$ and decrease $\vae_1$ such that $\vae_1\le r_1^\beta/(2C_0)$ and all the preceding smallness requirements for $\vae_1$ hold. Then~\eqref{HgHleer}--\eqref{C0r0albeta} give
\begin{align}\label{HgHlerbe}
    \mca N_2(r_1)\le r_1^\beta\mca N_2(1).
\end{align}
By dilation in~\eqref{gggndilinv}, we iterate it on
$B_{r_1},B_{r_1^2},\ldots$ to obtain
\[
    \mca N_2(r_1^\nu)\le r_1^{\nu\beta}\mca N_2(1),
    \qquad \nu\in\N_0.
\]
For $r_1^\nu<r\le r_1^{\nu-1}$ with $\nu\ge1$, by~\eqref{equivSobnorm}, we have
\[
    \mca N_2(r)
    \le C(n)\mca N_2(r_1^{\nu-1})
    \le C(n)\mko r_1^{(\nu-1)\beta}\mca N_2(1)
    \le C(n)r_1^{-\beta}r^\beta\mca N_2(1).
\]
Thus for all $r\in (0,1]$, there holds
\[
    \mca N_2(r)\le C(n,\La,\bc,\beta)r^\beta\mca N_2(1),
    \qquad 0<r\le1.
\]
Finally, for all $r\in (0,\frac 14]$, applying~\eqref{N1lapHN1} and~\eqref{lapHN2} yields
\begin{align*}
    \mca N_1(r)
    &\le C(n,\La)\Big(
       \|\vec K\|_{\msc E_n(B_{2r})}
       +\mca N_2(2r)\Big)\\
    &\le C(n,\La,\bc)\mca N_2(4r)\\
    &\le C(n,\La,\bc,\beta)\mko r^\beta\mca N_2(1).
\end{align*}
This completes the proof.
\end{proof}
\subsection{Proof of Theorem~\ref{th-main}}\label{sec:ndreg}
\
\vskip5pt
We are now ready to prove the main theorem.
\begin{hproof5}
Let $\bP_0\in\mathcal I_{h-1,2}(\Sigma,\R^m)$ be a weak critical point of $E_{\bc}$. 
By~\cite[Thm.~1.2]{MarRiv2}, there are local $g_{\bP_0}$-harmonic coordinates in which the metric belongs to $W^{2,(h,1)}\hookrightarrow C^0$. Let $\vec y$ be a local bi-Lipschitz harmonic coordinate map provided by~\cite[Sec.~5]{MarRiv2}. 
In the original local coordinates of $\Si$, we set $ a_{\bP_0}^{ij}\coloneq(\det g_{\bP_0})^{1/2}g_{\bP_0}^{ij}
    \in L^\nf\cap W^{h,2}_{\loc}$. Then the harmonic coordinate condition is
\begin{equation}\label{eq:harcoorconori}
    \p_i(a_{\bP_0}^{ij}\p_j\mko \vec y)=0.
\end{equation}
Since $D\vec y\in L^\nf_{\loc}$, differentiating~\eqref{eq:harcoorconori} once and successively applying Lemma~\ref{lm:ellcacciolp}, with $\ell=\nu-1$ and $p=n/\nu$ for $\nu=1,\ldots,h$, yields $D^{\nu+1}\vec y\in L^{n/\nu}_{\loc}$.
The bi-Lipschitz property and the Sobolev chain rule then imply that $\vec y\mko^{-1}\in W^{h+1,2}_{\loc}$ and $\bP_0\circ \vec y\mko^{-1}\in W^{1,\nf}_{\loc}\cap W^{h+1,2}_{\loc}$.

The proof of Lemma~\ref{lm-critic-W^{1,1}} shows that the first variation formula remains valid for compactly supported variations in $W^{1,\nf}\cap W^{h+1,2}$.
Since the integrand of $E_{\bc}$ is invariant under reparametrization, $\bP_0\circ \vec y\mko^{-1}$ is also a weak critical point of $E_{\bc}$ on its domain. 
Thus, it suffices to consider a weak critical point $\bP\in\mathcal I_{h-1,2}(B^n,\R^m)$ in harmonic coordinates, satisfying~\eqref{weaimmconbP}.

Let $\beta\in (0,1)$ be arbitrary. By normalizing a sufficiently small ball, we may assume in addition~\eqref{smaIIggvae1}. By Corollary~\ref{morestHgn} and rescaling, we obtain $\bP\in W^{h+2,\lf(\frac{2h}{h+1},2\rg)}_{\loc}(B^n,\R^m)$ and
\begin{align}\label{morHgHapp}
 \sup_{x\in B_{1/2},\,0<r\le1/4}r^{-\beta}
 \Big( \|\bH\|_{W^{h,\lf(\frac{2h}{h+1},2\rg)}(B_r(x))}
       +\|D^2\bP\|_{W^{h-1,2}(B_r(x))}\Big)<\nf.
\end{align}
Moreover, combining~\eqref{bdN2N1} and~\eqref{morHgHapp} with rescaling yields
\begin{align}\label{eq:PhihigherMorreyV}
\begin{aligned}
    & \sup_{x\in B_{1/2},\,0<r\le 1/8}r^{-\beta}
    \|D^2\bP\|_{W^{h,\lf(\frac{2h}{h+1},2\rg)}(B_{r}(x))}\\
    &\le C(n,\La)\sup_{x\in B_{1/2},\,0<r\le1/8}r^{-\beta}
    \Big( \|\bH\|_{W^{h,\lf(\frac{2h}{h+1},2\rg)}(B_{2r}(x))}
       +\|D^2\bP\|_{W^{h-1,2}(B_{2r}(x))}\Big)\\
       &<\nf.
\end{aligned}
\end{align}
We choose $p_0<2h/(h+1)$ sufficiently close to $2h/(h+1)$. By~\eqref{eq:PhihigherMorreyV} and H\"older's inequality, we obtain
\[
\sup_{x\in B_{1/2},\,0<r\le 1/8}r^{h+1-\frac n{p_0}-\beta}
    \|D^{h+2}\bP\|_{L^{p_0}(B_{r}(x))}
     <\nf.
\]
Applying the Riesz potential estimate~\cite[Thm.~3.1]{Adams75} for $I_1$ then gives
\begin{align}\label{phiW3s0}
 \forall\, 2<q_0<\frac{p_0(h+1-\beta)}{h-\beta},\qquad \bP\in W^{h+1,q_0}(B_{1/4},\R^m).
\end{align}
Fix $q_0$ as in~\eqref{phiW3s0}, and we assume $q_0<2h/(h-1)$. A covering and rescaling then gives $\bP\in W^{h+1,q_0}_{\loc}(B^n,\R^m)$. We now bootstrap the integrability exponent using the Euler--Lagrange equation~\eqref{ELequ}, which we rewrite as
\begin{equation}\label{equicondia}
    d*_gd\vec K
    =-d*_g\sum_{\ga\in I}a_\ga\vec Q_\ga(\bP).
\end{equation}
Here $\vec K=\lap_g^h \bP$, and we recall from Lemma~\ref{lm-critic-W^{1,1}}\ref{expQgabP} that each $\vec Q_\ga(\bP)$ is a partial contraction with one free index of the form
\begin{equation}\label{pconvecQgafin}
    \text{pcontr}\Big(\big(\g^{(\ell_1)}\mko\bII\cdot\g^{(\ell_2)}\mko\bII\big)\ot\cdots\ot\big(\g^{(\ell_{2s-1})}\mko\bII\cdot\g^{(\ell_{2s})}\mko\bII\big)\ot\g^{(\ell_0)}\bP\Big),
\end{equation}
where $s\ge1$ and the non-negative integers $\ell_i$ satisfy
\begin{equation}\label{condskivQfin}
    1\le\ell_0\le n,\qquad
    0\le\ell_i\le n-2\;\;\;(1\le i\le2s),\qquad
    \sum_{i=0}^{2s}(\ell_i+1)=n+2.
\end{equation}
In the following, all the Sobolev spaces are taken over $B^n$. Suppose $\bP\in W^{h+1,q}_{\loc}$ for some $2<q_0\le q\le 2h/(h-1)$, and we choose $p$ such that
\begin{equation}\label{defpave}
    \frac 1p=\frac 1q+\frac 1n+\frac 12\Big(\frac 1q-\frac 12\Big).
\end{equation}
Then we have
\begin{equation}\label{eq:pbootstrapV}
   \frac 1q\le\frac2q-\frac{h-1}{2h}
    <\frac1p<\frac1q+\frac1n<\frac{h+1}{2h}.
\end{equation}
In particular, we have $2h/(h+1)<p<q$. By the inequality~\eqref{eq:multinoex} with $a=0$ and $\si=h$, we have $g_{ij}\in W^{h,q}_{\loc}$. 
Repeatedly applying~\eqref{eq:multinoex} then yields $\det(g_{ij})\in W^{h,q}_{\loc}$. Since $\det(g_{ij})$ admits a uniform positive lower bound, we obtain $(\det g_{ij})^{-1/2}\in W^{h,q}_{\loc}$ and $(g^{ij})_{1\le i,j\le n}\subset W^{h,q}_{\loc}$. 
Let $2\le k\le n+1$ and $1\le i_1,\dots,i_k\le n$. Combining~\eqref{eq:multinoex} with $a=k-1$, $\si=h$, and the expression~\eqref{difcov}--\eqref{congpdif}, we obtain
\begin{equation}\label{itecovWh+1-kq}
   \g_{i_1\dots i_k}\bP\in  W^{h+1-k,q}_{\loc}(B^n,\R^m).
\end{equation}
By~\eqref{eq:multinoex}, it follows that  
\begin{equation}\label{KinW1-hq}
    \vec K\in W^{1-h,q}_{\loc}(B^n,\R^m).
\end{equation} 
To estimate $\vec Q_\ga(\bP)$, we apply~\eqref{eq:multinonexact} with $a=n$ and $\si=h$. By~\eqref{pconvecQgafin}--\eqref{eq:pbootstrapV}, the conditions in Lemma~\ref{lm:multinonexact} are satisfied, and we obtain
\begin{equation}\label{eq:QhigherbootstrapV}
   *_g \sum_{\ga\in I}a_\ga\vec Q_\ga(\bP)
    \in W^{1-h,p}_{\loc}\big(B^n,\R^m\ot\bwe^{n-1}\R^n\big).
\end{equation}
Set $a^{ij}\coloneq(\det g)^{1/2}g^{ij}$. By~\eqref{equicondia} and~\eqref{eq:QhigherbootstrapV}, we have
\begin{equation}\label{lapKinW-hp}
    \p_i(a^{ij}\p_j \vec K)
    \in W^{-h,p}_{\loc}(B^n,\R^m).
\end{equation} 
Since $p>2h/(h+1)$ and $q>n/(n-h)=2$, applying Lemma~\ref{lm:ellcacciolp}~\ref{negmorest} with $k=h$ and $\ell=h-1$ gives
\begin{equation}\label{lapHLs}
    \vec K\in W^{2-h,p}_{\loc}(B^n,\R^m).
\end{equation}

We next argue as in the proofs of~\eqref{HgHlapH} and~\eqref{bdN2N1}, using Lemma~\ref{lm:ellboot} in place of Lemma~\ref{lm:ellcacciolp}, to obtain that $\bP\in W^{h+2,p}_{\loc}$. For $0\le \nu\le h-1$, we denote 
\[
    \ti F_\nu\coloneq \lap_g^\nu \bH.
\]
Then we have 
\begin{equation}\label{tiFh-1inW2-hp}
    \ti F_{h-1}=-\vec K/n\in W^{2-h,p}_{\loc}.
\end{equation}
By~\eqref{itecovWh+1-kq} and~\eqref{eq:multinoex}, for all $0\le \nu\le h-2$, we also obtain
\begin{equation}\label{tiFvinWh-1-2nu}
     \ti F_\nu\in W^{h-1-2\nu,q}_{\loc}(B^n,\R^m).
\end{equation}
From~\eqref{tiFh-1inW2-hp}, we aim to show that $\ti F_\nu\in W^{h-2\nu,p}_{\loc}$ for all $0\le\nu\le h-1$. For $\nu=h-2,\ldots,0$, we successively apply Lemma~\ref{lm:ellboot} to the equation
\[
    \p_i(a^{ij}\p_j\ti F_\nu)=-(\det g)^{1/2}\ti F_{\nu+1}
    \qquad\text{in }\mca D'(B^n,\R^m).
\]
By~\eqref{eq:pbootstrapV}, we have $2h/(h+2)<p<2h/(h-2)$. Since $(\det g)^{1/2}\in L^\nf\cap W^{h,2}(B^n)\hookrightarrow L^\nf\cap W^{h-2,\frac{2h}{h-2}}(B^n)$, Lemma~\ref{lm:proinene} then implies that multiplication by $(\det g)^{1/2}$ is bounded on $W^{h-2\nu-2,p}_{\loc}$ for all $0\le \nu\le h-2$. 
By~\eqref{eq:pbootstrapV} and~\eqref{tiFvinWh-1-2nu}, the hypotheses for applying Lemma~\ref{lm:ellboot} are satisfied for all $0\le \nu\le h-2$. Hence, we obtain by induction that
\[
   \forall\, 0\le \nu\le h-1,\qquad  \ti F_\nu\in W^{h-2\nu,p}_{\loc}(B^n,\R^m).
\]
In particular, it holds that
\begin{equation}\label{HinWhp}
    \bH\in W^{h,p}_{\loc}(B^n,\R^m).
\end{equation}
For $1\le k\le n$, we set $\bv_k\coloneq \p_k\bP$. Under harmonic coordinates, equations~\eqref{eq:DPhidivV}--\eqref{eq:BkV} give
\begin{numcases}{}
 \p_i\big(g^{ij}\p_j\vec v_k\big)=n\mkt\p_k\bH+\vec B_k,\label{finDPhidivV}\\[0.3ex]
 \vec B_k=
       (\p_i g^{ij})\p_j\vec v_k
       -(\p_k g^{ij})\p_i\p_j\bP.
       \label{finBkV}
    \end{numcases} 
Since $Dg^{ij},D^2\bP\in W^{h-1,q}_{\loc}$, inequality~\eqref{eq:multinonexact} gives $\vec B_k\in W^{h-1,p}_{\loc}$. Applying~\eqref{HinWhp} and Lemma~\ref{lm:ellboot} with $\ell=\si=h$ to the equation~\eqref{finDPhidivV} for $\p_k\bP$ then yields
\begin{equation}\label{phiW4s}
    \bP\in W^{h+2,p}_{\loc}(B^n,\R^m).
\end{equation}
Since $p<q\le 2h/(h-1)<n$, by the Sobolev embedding, we obtain $\bP\in  W^{h+1,\ti q}_{\loc}$ with $1/\ti q=1/p -1/n$. Moreover, since $q\ge q_0>2$, the choice~\eqref{defpave} of $p$ implies
\begin{equation}\label{eq:qbootstrapV}
    \frac 1q-\frac 1{\ti q}=\frac 12 \Big(\frac 12-\frac 1q \Big)\ge \frac 12\Big(\frac 12-\frac 1{q_0} \Big)>0.
\end{equation}
Consequently, starting from~\eqref{phiW3s0} and iterating finitely many times, for some $q_1>2h/(h-1)$, we obtain that 
\begin{equation}\label{phiW3p}
    \bP\in W^{h+1,q_1}_{\loc}(B^n,\R^m).
\end{equation}

Fix such $q_1$. We now prove by induction on $\si$ that $\bP\in W^{\si+1,q_1}_{\loc}$ for all $\si\ge h$. For $\si=h$, this has been proved.
Suppose $\si\ge h$ and $\bP\in W^{\si+1,q_1}_{\loc}$. Since $q_1>2h/(h-1)$, we have 
\begin{equation}\label{relq_1nsi}
    \frac 1{q_1}>\frac 2{q_1}-\frac{\si-1}{n}.
\end{equation}
Using~\eqref{relq_1nsi} in place of~\eqref{eq:pbootstrapV} and arguing as in the derivation of~\eqref{KinW1-hq}--\eqref{eq:QhigherbootstrapV}, we obtain
\begin{equation}\label{lapKinWsi-n}
    \vec K\in W^{\si+1-n,q_1}_{\loc},\qquad *_g
    \sum_{\ga\in I}a_\ga\vec Q_\ga(\bP)
    \in W^{\si+1-n,q_1}_{\loc}.
\end{equation}
Combining~\eqref{equicondia} and~\eqref{relq_1nsi}--\eqref{lapKinWsi-n} with Lemma~\ref{lm:ellboot} then gives $ \vec K\in W^{\si+2-n,q_1}_{\loc}$.
Using Lemma~\ref{lm:multinonexact} in place of Lemma~\ref{lm:proinene} for the product estimates and successively applying Lemma~\ref{lm:ellboot} as in the proof of~\eqref{HinWhp}, we obtain that
\[
    \bH\in W^{\si,q_1}_{\loc}(B^n,\R^m).
\]
Using~\eqref{finDPhidivV}--\eqref{finBkV} together with~\eqref{eq:multinonexact} and Lemma~\ref{lm:ellboot}, it follows that $\bP\in W^{\si+2,q_1}_{\loc}$. This completes the induction, and hence $\bP\in C^\nf_{\loc}(B^n)$.

By~\eqref{pconvecQgafin}--\eqref{condskivQfin}, each $\vec Q_\ga(\bP)$ is an $\R^m\ot \R^n$-valued real-analytic function of $\{D^k\bP\}_{1\le k\le n}$ on the set where $\det g>0$. 
In harmonic coordinates, we have $\lap_g=-g^{ij}\p_i\p_j$ on scalar-valued and $\R^m$-valued functions.
Thus, substituting $\vec K=\lap_g^h\bP$ into equation~\eqref{equicondia}, we obtain an analytic elliptic system of order $n+2$ for $\bP$. Taking the coefficient of $dx^1\we\cdots \we dx^n$, its principal symbol is
\[
    (-1)^h(\det g)^{1/2}(g^{ij}\xi_i\xi_j)^{h+1}\operatorname{Id}_{\R^m}.
\]
By~\cite{Morrey58}, $\bP$ is real-analytic in this harmonic chart, hence the induced metric is also real-analytic.
The transition map to any other $g$-harmonic chart solves a homogeneous elliptic equation with analytic coefficients and is thus real-analytic by~\cite{Morrey58}.
Since the initial chart and its center are arbitrary, $\bP_0$ is real-analytic in any $g_{\bP_0}$-harmonic charts. This completes the proof.
\end{hproof5}


\begin{thebibliography}{99}
\setlength{\itemsep}{0pt} 
\bibitem{Adams75} Adams, David R. \textit{A note on Riesz potentials}. Duke Math. J. {\bf 42} (1975), no.~4, 765--778.

\bibitem{Alexakis12} Alexakis, Spyros. \textit{The decomposition of global conformal invariants}. Annals of Mathematics Studies, Vol. 182. Princeton University Press, Princeton, NJ, 2012. 

\bibitem{AG2024} Allen, Ben F.; Gover, Rod. \textit{Higher Willmore Energies from tractor coupled GJMS operators}. Preprint (2024) \url{https://arxiv.org/abs/2411.01835}.

\bibitem{AmbMan98} Ambrosio, Luigi; Mantegazza, Carlo. \textit{Curvature and distance function from a manifold}. Dedicated to the memory of Fred Almgren. J. Geom. Anal. {\bf8} (1998), no.~5, 723–748.

\bibitem{ArrCapGuv00} Arreaga, G.; Capovilla, R.; Guven, J. \textit{Noether currents for bosonic branes}. Ann. Physics {\bf279} (2000), no.~1, 126–158.

\bibitem{Atiyah73} Atiyah, M.; Bott, R.; Patodi, V. K. \textit{On the heat equation and the index theorem}. Invent. Math. {\bf19} (1973), 279–330.

\bibitem{Bailey94} Bailey, Toby N.; Eastwood, Michael G.; Graham, C. Robin. \textit{Invariant theory for conformal and CR geometry}. Ann. of Math. (2) {\bf139} (1994), no.~3, 491–552.

\bibitem{Behzadan21} Behzadan, A.; Holst, M. \textit{Multiplication in Sobolev spaces, revisited}. Ark. Mat. {\bf59} (2021), no.~2, 275–306.

\bibitem{Bennett88} Bennett, Colin; Sharpley, Robert. \textit{Interpolation of operators}. Pure and Applied Mathematics, Vol. 129. Academic Press, Inc., Boston, MA, 1988.

\bibitem{Ber} Bernard, Yann. \textit{Noether's theorem and the Willmore functional}. Adv. Calc. Var. {\bf 9} (2016), no.~3, 217--234.

\bibitem{Bernard25}Bernard, Yann. \textit{Structural equations for critical points of conformally invariant curvature energies in $4$d}. Preprint (2025), \url{https://arxiv.org/abs/2509.01179}.

\bibitem{BerLanMarRiv26} Bernard, Yann; Lan, Tian; Martino, Dorian; Rivière, Tristan. \textit{The regularity of critical points to scale-invariant curvature energies in dimension 4}. Preprint (2026), \url{https://arxiv.org/abs/2511.01765}.

\bibitem{blitz2023} Blitz, Samuel; Gover, A. Rod; Waldron, Andrew. \textit{Generalized Willmore energies, $Q$-curvatures, extrinsic Paneitz operators, and extrinsic Laplacian powers}, Commun. Contemp. Math. {\bf 26} (2024), no.~5, Paper No.~2350014, 50 pp.

\bibitem{Byun04} Byun, Sun-Sig; Wang, Lihe. \textit{Elliptic equations with BMO coefficients in Reifenberg domains}. Comm. Pure Appl. Math. {\bf57} (2004), no.~10, 1283--1310.

\bibitem{Byun05} Byun, Sun-Sig. \textit{Elliptic equations with BMO coefficients in Lipschitz domains}. Trans. Amer. Math. Soc. {\bf357} (2005), no.~3, 1025--1046.

\bibitem{CafStiViv24} Caffarelli, L.~\'A.; Stinga, P.~R.; Vivas, H.~A. \textit{A PDE approach to the existence and regularity of surfaces of minimum mean curvature variation}. Arch. Ration. Mech. Anal. {\bf 248} (2024), no.~5, Paper No. 70, 17 pp.

\bibitem{CapGuv02}Capovilla, R.; Guven, J. \textit{Stresses in lipid membranes}. J. Phys. A {\bf35} (2002), no.~30, 6233–6247.

\bibitem{Cartan71} Cartan, Henri. \textit{Differential calculus}. Exercises by C. Buttin, F. Rideau and J. L. Verley. Translated from the French. Hermann, Paris; Houghton Mifflin Co., Boston, MA, 1971.

\bibitem{CaseGrahamKuo25} Case, Jeffrey S.; Graham, C. Robin; Kuo, Tzu-Mo. \textit{Extrinsic GJMS operators for submanifolds}. Rev. Mat. Iberoam. {\bf41} (2025), no.~4, 1393--1429.

\bibitem{CaseEtAl26} Case, Jeffrey S.; Khaitan, Ayush; Lin, Yueh-Ju; Tyrrell, Aaron J.; Yuan, Wei. \textit{Local and global conformal invariants of submanifolds}. Preprint (2026), \url{https://arxiv.org/abs/2604.08372}.

\bibitem{CLMS93}Coifman, R.; Lions, P.-L.; Meyer, Y.; Semmes, S. \textit{Compensated compactness and Hardy spaces}. J. Math. Pures. Appl. (9) {\bf72} (1993), no.~3, 247--286.

\bibitem{Costa10}Costabel, Martin; McIntosh, Alan. \textit{On Bogovskiĭ and regularized Poincaré integral operators for de Rham complexes on Lipschitz domains}. Math. Z. {\bf265} (2010), no.~2, 297--320.

\bibitem{CruzUribe16}Cruz-Uribe, David; Moen, Kabe; Rodney, Scott. \textit{Regularity results for weak solutions of elliptic PDEs below the natural exponent}. Ann. Mat. Pura Appl. (4) {\bf195} (2016), no.~3, 725–740.

\bibitem{deLongueville19}de Longueville, Frédéric Louis. \textit{Regularität der Lösungen von Systemen ($2m$)-ter Ordnung vom polyharmonischen Typ in kritischer Dimension}. Diss. Universität Duisburg-Essen, 2019.

\bibitem{deLonGas21} de Longueville, Frédéric Louis; Gastel, Andreas. \textit{Conservation laws for even order systems of polyharmonic map type}. Calc. Var. Partial Differential Equations {\bf60} (2021), no.~4, Paper No.~138, 18 pp.

\bibitem{DeserSchwimmer93} Deser, S.; Schwimmer, A. \textit{Geometric classification of conformal anomalies in arbitrary dimensions}. Phys. Lett. B {\bf309} (1993), no.~3--4, 279--284.

\bibitem{dC92} do Carmo, Manfredo Perdigão. \textit{Riemannian geometry}, translated from the second Portuguese edition by Francis Flaherty. Mathematics: Theory \& Applications. Birkh\"auser Boston, Inc., Boston, MA, 1992.

\bibitem{evans} Evans, Lawrence C. \textit{Partial differential equations}, 2nd edition. Graduate Studies in Mathematics, Vol. 19. American Mathematical Society, Providence, RI, 2010.

\bibitem{Federer96} Federer, Herbert. \textit{Geometric measure theory}, reprint of the 1969 edition. Classics in Mathematics.
Springer-Verlag, Berlin, 1996.

\bibitem{Fefferman85} Fefferman, Charles; Graham, C. Robin. \textit{Conformal invariants}. In: The mathematical heritage of Élie Cartan (Lyon, 1984). Astérisque 1985, Numéro Hors Série, 95–116.

\bibitem{FeffermanGraham12} Fefferman, Charles; Graham, C. Robin. \textit{The Ambient Metric}. Annals of Mathematics Studies, Vol.~178. Princeton University Press, Princeton, NJ, 2012.

\bibitem{Gilkey75} Gilkey, Peter B. \textit{Local invariants of an embedded Riemannian manifold}. Ann. of Math. (2) {\bf102} (1975), no.~2, 187–203.

\bibitem{gover2017} Gover, A. Rod; Waldron, Andrew. \textit{Renormalized volume}. Comm. Math. Phys. {\bf 354} (2017), no.~3, 1205--1244.

\bibitem{gover2020} Gover, A. Rod; Waldron, Andrew. \textit{A calculus for conformal hypersurfaces and new higher Willmore energy functionals}. Adv. Geom. {\bf 20} (2020), no.~1, 29--60.

\bibitem{graham2017} Graham, C. Robin. \textit{Volume renormalization for singular Yamabe metrics}. Proc. Amer. Math. Soc. {\bf 145} (2017), 1781--1792.

\bibitem{grahamkuo26} Graham, C. Robin; Kuo, Tzu-Mo. \textit{Geodesic normal coordinates and natural tensors for pseudo-Riemannian submanifolds}. Proc. Amer. Math. Soc. {\bf154} (2026), no.~1, 339–351.

\bibitem{grahamreichert2020} Graham, C. Robin; Reichert, Nicholas. \textit{Higher-dimensional Willmore energies via minimal submanifold asymptotics}. Asian J. Math. {\bf 24} (2020), no.~4, 571--610.

\bibitem{graham1999} Graham, C. Robin; Witten, Edward. \textit{Conformal anomaly of submanifold observables in AdS/CFT correspondence}. Nuclear Phys. B {\bf 546} (1999), no.~1-2, 52--64.

\bibitem{grisvard11} Grisvard, Pierre. \textit{Elliptic problems in nonsmooth domains}, reprint of the 1985 original. With a foreword by Susanne C. Brenner. Classics in Applied Mathematics, Vol. 69. Society for Industrial and Applied Mathematics (SIAM), Philadelphia, PA, 2011.

\bibitem{Guven05} Guven, Jemal. \textit{Conformally invariant bending energy for hypersurfaces}. J. Phys. A {\bf38} (2005), no.~37, 7943--7955.

\bibitem{H2002} H\'elein, Frédéric. {\it Harmonic maps, conservation laws and moving frames}, translated from the 1996 French original Second edition, Cambridge Tracts in Mathematics, 150, Cambridge Univ. Press, Cambridge, 2002.

\bibitem{Hunt}Hunt, Richard A. \textit{On $L(p,q)$ spaces}.
Enseign. Math. (2) {\bf12} (1966), 249--276.

\bibitem{Jost17} Jost, J\"urgen. \textit{Riemannian geometry and geometric analysis}, 7th edition. Universitext. Springer, Cham, 2017.

\bibitem{KuwLamLi15} Kuwert, Ernst; Lamm, Tobias; Li, Yuxiang. \textit{Two-dimensional curvature functionals with superquadratic growth}. J. Eur. Math. Soc. (JEMS) {\bf17} (2015), no.~12, 3081–3111.

\bibitem{KuwertLi12} Kuwert, Ernst; Li, Yuxiang. \textit{$W^{2,2}$-conformal immersions of a closed Riemann surface into $\R^n$}. Comm. Anal. Geom. {\bf20} (2012), no.~2, 313–340.

\bibitem{laMan20} La Manna, Domenico Angelo; Leone, Chiara; Schiattarella, Roberta. \textit{On the regularity of very weak solutions for linear elliptic equations in divergence form}. NoDEA Nonlinear Differential Equations Appl. {\bf 27} (2020), no.~5, Paper No.~43, 23 pp.

\bibitem{LammRiv08} Lamm, Tobias; Rivi\`ere, Tristan. \textit{Conservation laws for fourth order systems in four dimensions}. Comm. Partial Differential Equations {\bf33} (2008), no.~1--3, 245--262.

\bibitem{LaMaRi26}Lan, Tian; Martino, Dorian; Rivi\`ere, Tristan. \textit{The analysis of Willmore surfaces and its generalizations in higher dimensions}. J. Math. Study {\bf 59} (2026), no.~1, 80--188.

\bibitem{Mantegazza02}Mantegazza, C. \textit{Smooth geometric evolutions of hypersurfaces}. Geom. Funct. Anal. {\bf12} (2002), no.~1, 138–182.

\bibitem{Martino2026} Martino, Dorian. \textit{A duality theorem for a four dimensional Willmore energy}. Preprint (2026) \url{https://arxiv.org/abs/2308.11433}.

\bibitem{MarRiv2} Martino, Dorian; Rivière, Tristan. \textit{Weak immersions with second fundamental form in a critical Sobolev space}. Preprint (2025), \url{https://arxiv.org/abs/2510.10594}.

\bibitem{MarRiv26}Martino, Dorian; Rivière, Tristan. \textit{Construction of harmonic coordinates for weak immersions}. Int. Math. Res. Not. IMRN {\bf2026}, no.~10, Paper No.~rnag083.

\bibitem{Miranda63} Miranda, Carlo. \textit{Sulle equazioni ellittiche del secondo ordine di tipo non variazionale, a coefficienti discontinui}. Ann. Mat. Pura Appl. (4) {\bf63} (1963), 353–386.

\bibitem{mitrea01}Mitrea, Dorina; Mitrea, Marius; Taylor, Michael. \textit{Layer potentials, the Hodge Laplacian, and global boundary problems in nonsmooth Riemannian manifolds}. Mem. Amer. Math. Soc. {\bf150} (2001), no.~713.

\bibitem{mondino2018} Mondino, Andrea; Nguyen, Huy T. \textit{Global conformal invariants of submanifolds}. Ann. Inst. Fourier (Grenoble) {\bf68} (2018), no.~6, 2663--2695.

\bibitem{MonRiv14} Mondino, Andrea; Rivière, Tristan. \textit{Immersed spheres of finite total curvature into manifolds}. Adv. Calc. Var. {\bf7} (2014), no.~4, 493–538.

\bibitem{Morrey58}Morrey, Charles B., Jr. \textit{On the analyticity of the solutions of analytic non-linear elliptic systems of partial differential equations. I. Analyticity in the interior}. Amer. J. Math.~{\bf80} (1958), 198--218.

\bibitem{MullerSverak95}Müller, S.; Šverák, V. \textit{On surfaces of finite total curvature}. J. Differential Geom. {\bf42} (1995), no.~2, 229–258.

\bibitem{Petersen16}Petersen, Peter. \textit{Riemannian geometry}, 3rd edition. Graduate Texts in Mathematics, Vol. 171. Springer, Cham, 2016.

\bibitem{Riv07} Rivi\`ere, Tristan. \textit{Conservation laws for conformally invariant variational problems}. Invent. Math. {\bf168} (2007), no.~1, 1--22.

\bibitem{Riv08} Rivi\`ere, Tristan. \textit{Analysis aspects of Willmore surfaces}. Invent. Math. {\bf174} (2008), no.~1, 1--45.

\bibitem{Riv14} Rivi\`ere, Tristan. \textit{Variational principles for immersed surfaces with $L^2$-bounded second fundamental form}. J. Reine Angew. Math. {\bf695} (2014), 41--98.

\bibitem{Ri16} Rivi\`ere, Tristan. \textit{Weak immersions of surfaces with $L^2$-bounded second fundamental form}. In: Geometric analysis, 303--384.
IAS/Park City Mathematics Series, Vol. 22. American Mathematical Society, Providence, RI, 2016. 

\bibitem{Ri20} Rivi\`ere, Tristan. \textit{The variations of Yang-Mills Lagrangian}. In: Geometric analysis---in honor of Gang Tian's 60th birthday, 305--379. Progress in Mathematics, Vol. 333. Birkh\"auser/Springer, Cham, 2020.

\bibitem{Rych99}Rychkov, Vyacheslav S. \textit{On restrictions and extensions of the Besov and Triebel--Lizorkin spaces with respect to Lipschitz domains}. J. London Math. Soc. (2) {\bf60} (1999), no.~1, 237–257.

\bibitem{Toro94}Toro, Tatiana. \textit{Surfaces with generalized second fundamental form in $L^2$ are Lipschitz manifolds}. J. Differential Geom. {\bf39} (1994), no.~1, 65–101.

\bibitem{Uhlenbeck82} Uhlenbeck, Karen K. \textit{Connections with $L^p$ bounds on curvature}. Comm. Math. Phys. {\bf83} (1982), no.~1, 31--42.

\bibitem{Wente69} Wente, Henry C. \textit{An existence theorem for surfaces of constant mean curvature}. J. Math. Anal. Appl. {\bf26} (1969), 318--344.

\bibitem{weyl97}Weyl, Hermann. \textit{The classical groups: their invariants and representations}, 15th printing. Princeton Landmarks in Mathematics. Princeton Paperbacks. Princeton University Press, Princeton, NJ, 1997.

\bibitem{zhang} Zhang, Yongbing. \textit{Graham-Witten's conformal invariant for closed four-dimensional submanifolds}. J. Math. Study {\bf 54} (2021), no.~2, 200--226.

\end{thebibliography}
\end{document}